%% file: Note_refracted_reflected_SPA_revised_06_clean.tex
\documentclass[a4paper,12pt]{article}

\input{style-common}

\input{style-e}

\input{style-layout-paper}

\usepackage{color}
\usepackage[pdftex]{graphicx}
\usepackage{mathtools}
\mathtoolsset{showonlyrefs=true}

\allowdisplaybreaks[4]

\begin{document}

\begin{center}
{\Large \bf 
On the optimality of the refraction--reflection strategies for L\'evy processes
}
\end{center}
\begin{center}
Kei Noba
\end{center}

\begin{abstract}
In this paper, we study de Finetti's optimal dividend problem with capital injection under the assumption that 
the dividend strategies are absolutely continuous. 
In many previous studies, the process before being controlled was assumed to be a spectrally one-sided L\'evy process, however in this paper we use a L\'evy process that may have both positive and negative jumps. 
In the main theorem, we show that a refraction--reflection strategy is an optimal strategy. 
We also mention the existence and uniqueness of solutions of the stochastic differential equations that define refracted L\'evy processes.
\\
\noindent \small{\noindent  \textbf{AMS 2020 Subject Classifications}: Primary 60G51; Secondary 93E20, 60J99 }
\end{abstract}

\section{Introduction}
We consider the situation in which a company pays dividends to its shareholders when the company has more than enough capital. 
In this situation, the company wants to pay more dividends before bankruptcy. 
The problem of finding a strategy for maximizing the expected net present value (NPV) of the dividend payments is called de Finetti's optimal dividend problem. 
In this paper, in addition to the above situation, we assume that the shareholders provide capital injection to the company to prevent bankruptcy. 
In this situation, the company wants to pay more dividends and receive less capital injection. 
Our purpose in this paper is to find a strategy that maximizes the expected NPV of the dividend payments minus the cost of capital injection. Such a strategy is called an optimal strategy. 
We deal with this problem under the assumption that the cumulative amount of the dividends is absolutely continuous with respect to the Lebesgue measure about time and the density is bounded by the value $\alpha>0$. 
\par
Let $X=\{X_t: t\geq 0\}$ be the surplus process of the company before dividends are deducted and capital injections are added. 
When $X$ is a L\'evy process, many previous studies have shown the optimality of refraction strategies or refraction--reflection strategies 
in de Finetti's optimal dividend problem or its capital injection version, respectively. 
A refraction strategy is a strategy in which the company continues to pay a certain amount of dividends only when the surplus exceeds a certain threshold. 
Additionally, a refraction--reflection strategy is a strategy in which the company pays dividends according to a refraction strategy when the surplus is non-negative, and the shareholders immediately provide capital injection of only the absolute value of the surplus when the surplus is negative. 
In de Finetti's optimal dividend problem, the optimality of the refraction strategies was proved by \cite{JeaShi1995}, \cite{AsuTak1997}, and \cite{Bog2006} when $X$ is a Brownian motion with a drift, by \cite{GerShi2006} when $X$ is a Cr\'amer--Lundberg process with exponential jumps, by \cite{KypLoePer2012} when $X$ is a spectrally negative L\'evy process, and by \cite{YinWenZha2013} when $X$ is a spectrally positive L\'evy process. 
On the other hand, in the capital injection version, 
the optimality of the refraction--reflection strategies was proved by \cite{PerYam2016} when $X$ is a spectrally positive L\'evy process and by \cite{PerYamYu2018} when $X$ is a spectrally negative L\'evy process. 
{In addition, the optimality of the refraction strategies was proved by \cite{HerPerYam2016} in another problem of stochastic control when $X$ is a spectrally negative L\'evy process.} 
However, de Finetti's optimal dividend problems in the above situations for the general L\'evy processes, which may have both positive and negative jumps, 
are the open problems. 
In this paper, we deal with de Finetti's optimal problem with capital injection for general L\'evy processes and prove the optimality of refraction--reflection strategies. 
\par
In addition to dealing with the general L\'evy processes, this paper makes a new attempt at strategies about the refraction. 
In previous studies, it has been assumed that the drift coefficient of $X$ is not contained in $[0, \alpha]$ when $X$ has bounded variation paths. (In this paper, we refer to the case in which $X$ has bounded variation paths and the drift coefficient of $X$ is included in $[0, \alpha]$ as Case $2$, and we refer to the other cases as Case $1$; see, Section \ref{SubSec202}.) 
In this paper, we introduce new refraction--reflection strategies and prove their optimality in Case $2$. 
\par
Before proving the optimality of the refraction--reflection strategies, their existence is an important problem. The resulting processes of the refraction--reflection strategies are refracted--reflected processes, which were introduced by \cite{PerYam2018}. 
The construction of the refracted--reflected L\'evy processes involves the following two steps.
\begin{enumerate}
\item We define the refracted L\'evy processes using stochastic differential equations under non-Lipschitz conditions. 
\item We inductively define the refracted--reflected L\'evy processes by reflecting the refracted L\'evy processes at $0$ using the infimum. 
\end{enumerate}
For the existence of the refraction--reflection strategies, we must prove that the stochastic differential equations that define the refracted L\'evy processes have unique strong solutions. 
The existence and uniqueness of the strong solutions of these stochastic differential equations were proved by \cite{Sit2005} when $X$ has a positive Gaussian coefficient and by \cite{KypLoe2010} when $X$ is a spectrally one-sided L\'evy process in Case $1$. 
In addition, the existence and uniqueness of the strong solutions of the extended stochastic differential equations that define level-dependent L\'evy processes were proved by \cite{CzaPerRolYam2019} when $X$ is a spectrally one-sided L\'evy process. 
In this paper, we describe the existence and uniqueness of the strong solutions when general L\'evy process $X$ has bounded variation paths or has unbounded variation paths with a finite number of jumps on one side on a compact interval. 
Especially in Case $2$, we must introduce new refracted L\'evy processes, 
so we use a new idea for proving the existence of refracted L\'evy processes. 
The uniqueness 
can be proved as in the case of spectrally one-sided L\'evy processes, or {can be proved by \cite[Theorem 5.1]{Tsu2021}}. 
\par
We take the following steps to prove the optimality of the refraction--reflection strategies. 
\begin{enumerate}
\item From previous studies, we can predict the optimal threshold $b^\ast$. 
We obtain a good representation of the density of the expected NPV $v_{b^\ast}$ of dividends and capital injections when we take the refraction--reflection strategy at $b^\ast$. 
\item We give the verification lemma, which gives the sufficient condition for being an optimal strategy. 
Using the result in Step (i), we prove that the function $v_{b^\ast}$ satisfies the conditions that are given in the verification lemma. 
From this argument, we can see the optimality of the refraction--reflection strategy at $b^\ast$. 
\end{enumerate}
In the spectrally one-sided cases, the theory of the scale functions of spectrally one-sided L\'evy processes was at the core of the proof. 
However, there are no corresponding scale functions in the general L\'evy processes, and thus we cannot use the same proof. 
On the other hand, in recent years, \cite{Nob2021} and \cite{NobYam2022} have proved the optimality of the strategies for the problems of stochastic control using the general L\'evy processes. 
\cite{Nob2021} proved the optimality of the double barrier strategies for de Finetti's optimal dividend problem with capital injection (there, the absolute continuity of the cumulative amount of the dividends was not assumed), and \cite{ NobYam2022} proved the optimality of the barrier strategies for the singular control problem. 
In those papers, we investigated the properties of the expected NPV of some values in some strategies by observing the differences in the behavior of two slightly staggered sample paths, and we proved the optimality of the strategies.   
We use the same method for Step (i) of the proofs in this paper. 
The difficulty in the setting of this paper lies in the fact that the variation of the differences in the behavior of two slightly staggered sample paths is much more complicated in the case of the strategies about refraction than in the case of the strategies about only reflection. 
An important aspect of this paper is that this difficult part is solved and Step (i) of the proof is completed. 
The proofs used in this paper are expected to be applied to prove the optimality of various strategies about refraction {for many problems of stochastic control using the L\'evy processes}. 
\par
This paper is organized as follows. In Section \ref{Pre}, we describe the notation and give some
assumptions about L\'evy processes, and we recall refracted and refracted--reflected L\'evy processes. 
In Section \ref{OptimalDividend}, we give a mathematical setting for de Finetti's optimal dividend problem with capital injection that we deal with in this paper, 
and we recall the definition of refraction--reflection strategies and confirm that they are admissible. 
In Section \ref{bast}, we compute the densities of the expected NPVs of dividends and capital injections when we take an optimal refraction--reflection strategy at $b^\ast$. 
In Section \ref{Sec04}, we give the verification lemma and prove the optimality of a refraction--reflection strategy at $b^\ast$ using the results in Section \ref{bast} and the verification lemma; the main result is in this section. 
In Section \ref{numericalresults}, we give numerical results from Monte Carlo simulation, 
confirming numerically the correctness of our main results. 
{In Section \ref{Sec007}, we prove that the expected NPV of an optimal dividend payments and capital injections in this paper converges to that in \cite{Nob2021} as $\alpha\uparrow\infty$.} 
In Appendix \ref{Sec0A}, we represent the infimum of the bounded variation functions using the stagnation time at $0$ of the function reflected at $0$ and the jumps; the result in this section is used in the proofs of results in Appendix \ref{ODE} and 
Appendix \ref{Sec00B}. 
In {Appendix} \ref{ODE}, we give the existence 
of the solution of an ordinary differential equation; 
the results obtained in this section are used to prove the existence of the refracted L\'evy processes in the case of bounded variation. {(The uniqueness of the solution can be proved in the same way as the previous studies, thus we omit it.)}
In Appendix \ref{Sec00B}, we prove that the resulting processes of refraction--reflection strategies are refracted--reflected L\'evy processes. 
In Appendix \ref{two_paths}, we investigate the behavior of two paths of refracted--reflected L\'evy processes with slightly offset starting positions; 
the results in this section are used to obtain the densities of the expected NPVs of dividends and capital injections of the refraction--reflection strategy at $b^\ast$ in Section \ref{bast}. 
In Appendix \ref{Sec0D}, we give a lemma for the continuity of hitting times; 
this lemma is also used to compute the densities of the expected NPVs of dividends and capital injections and to prove an inequality that appears in the verification lemma. 
In Appendix \ref{Ape00D}, we give an example that satisfies the assumptions of the L\'evy processes with unbounded variation paths to which our main results can be applied. 
In Appendix \ref{prfverification}, we give the proof of the verification lemma. 
{In Appendix \ref{Ape00H}, we give the proofs of two lemmas in Section \ref{Sec007}.}

\section{Preliminaries} \label{Pre}
\subsection{L\'evy processes}
We let $X{:=}\{X_t : t \geq 0 \}$ be a L\'evy process that is defined on a probability space $(\Omega , \cF , \bP)$. 
For all $x\in\bR$, we write $\bP_x$ for the law of $X$ when it starts at $x$. 
We write $\Psi$ for the characteristic exponent of $X$, which is the function satisfying 
\begin{align}
e^{-t\Psi(\lambda)}=\bE_0\sbra{e^{i\lambda X_t}},  \quad \lambda \in \bR, \quad t\geq 0. 
\end{align}   
The characteristic exponent $\Psi$ has the form  
\begin{align}
\Psi (\lambda) = -i\gamma\lambda +\frac{1}{2}\sigma^2 \lambda^2 
+\int_{\bR\backslash\{0\}} (1-e^{i\lambda x}+i\lambda x1_{\{\absol{x}<1\}}) \Pi(dx) , \quad\lambda\in\bR, 
\end{align}
where $\gamma\in\bR$, $\sigma\geq 0$, and $\Pi $ is a L\'evy measure on $\bR \backslash \{0\}$ satisfying 
\begin{align}
\int_{\bR\backslash \{ 0\}}(1\land x^2) \Pi (dx). 
\end{align}
It is known that the process $X$ has bounded variation paths if and only if $\sigma= 0$ and the L\'evy measure $\Pi$ satisfies $\int_{(-1, 1)\backslash\{0\}} \absol{x}\Pi(dx)<\infty$. 
When this holds, we denote
\begin{align}
\Psi(\lambda)= -i\delta\lambda+\int_{\bR\backslash\{0\}}(1-e^{i\lambda x})\Pi (dx),  \label{78}
\end{align}
where
\begin{align}
\delta {:=} \gamma-\int_{(-1, 1)\backslash\{0\}}x \Pi(dx). 
\end{align}
\par
We write $\bF=\{\cF_t : t\geq 0\}$ for the natural filtration generated by $X$. 
For $x\in \bR$, we define the hitting times of $X$ by  
\begin{align}
\tau^-_x{:=}\inf\{ t>0: X_t < x\}, \qquad \tau^+_x{:=}\inf\{t> 0 : X_t > x\}.
\end{align}
Here, we assume that $\inf \emptyset {:=} \infty$. 
We fix the value $\alpha>0$ throughout this paper. 
We define the process $X^{(\alpha)}{:=}\{X^{(\alpha)}_t: t\geq 0\}$ as 
\begin{align}
X^{(\alpha)}_t {:=} X_t -\alpha t, \quad t\geq 0. 
\end{align}
For $x \in \bR$, we define a hitting time of $X^{(\alpha)}$ by 
\begin{align}
\tau^{(\alpha),-}_x{:=}\inf\{ t>0: X^{(\alpha)}_t < x\} .
\end{align}
\par
The following assumption is necessary 
for strategies to be admissible (see Remark \ref{Rem302}). 
\begin{Ass}\label{Ass201}
We assume that the L\'evy measure $\Pi$ satisfies 
\begin{align}
\int_{(-\infty , -1)} (-x) \Pi (dx)< \infty. \label{61}
\end{align}
\end{Ass}

\subsection{Refracted L\'evy processes}\label{SubSec202}
The L\'evy process $X$ has the following two cases.
\begin{itemize}
\item[Case $1$] $X$ has unbounded variation paths or has bounded variation paths with $\delta \not \in [0, \alpha]$. 
\item[Case $2$] $X$ has bounded variation paths and $\delta \in [0 , \alpha ]$.
\end{itemize}
In each case, for $b\in\bR$, we define the process $Y^b$ as the solution of the stochastic differential equation 
\begin{align}
Y^b_t = X_t -  \int_0^t h^b (Y^b_s) d s,  \quad t \geq 0, \label{4}
\end{align}
where 
\begin{align}
h^b(y){:=}
\begin{cases}
\alpha1_{(b, \infty)}(y) \quad &\text{ in Case $1$}, \\
\alpha1_{(b, \infty)}(y)+\delta 1_{\{b\}}(y) \quad &\text{ in Case $2$}.
\end{cases}
\end{align}
The process $Y^b$ is called a refracted L\'evy process at $b$. 
We give the following assumption for $X$.  
\begin{Ass}\label{Ass202a}
The stochastic differential equation \eqref{4} has a unique strong solution. 
\end{Ass}
\begin{Rem}\label{Rem203a}
It is known that $X$ satisfies Assumption \ref{Ass202a} in the following cases.
\begin{enumerate}
\item If the process $X$ is a spectrally negative L\'evy process, which is a L\'evy process with no positive jumps and no monotone paths, and satisfies the condition in Case $1$, 
then \eqref{4} has a unique strong solution by \cite[Theorem 1]{KypLoe2010}. 
By the same argument as that of the proof of \cite[Theorem 1]{KypLoe2010}, it is easy to check that 
if the process $X$ is a spectrally positive L\'evy process, which is a L\'evy process with no negative jumps and no monotone paths, then \eqref{4} has a unique strong solution. 
\item If the process $X$ has a positive Gaussian coefficient, i.e., $\sigma>0$, and satisfies some conditions, then \eqref{4} has a unique strong solution by \cite[Theorem 305]{Sit2005}.
\end{enumerate}
\end{Rem}
In addition to Remark \ref{Rem203a}, we give the following. 
\begin{Cor}\label{Lem205}
When $X$ has unbounded variation paths and satisfies $\Pi(-\infty , 0)$ or $\Pi (0, \infty)<\infty$, the stochastic differential equation \eqref{4} has a unique strong solution. 
\end{Cor}
If $X$ has unbounded variation paths and $\Pi(-\infty , 0)<\infty$, by using Remark \ref{Rem203a} (i) again each time a negative jump of $X$ appears, a unique strong solution of \eqref{4} can be constructed inductively. The same is true for the case in which $X$ has unbounded variation paths and $\Pi(0, \infty )<\infty$. Thus, Corollary \ref{Lem205} can be proved soon by Remark \ref{Rem203a} (i). 
We omit the detail of the proof. 
\begin{Thm}\label{Lem204}
When $X$ has bounded variation paths, the stochastic differential equation \eqref{4} has a unique strong solution. 
\end{Thm}
When $X$ has bounded variation paths, \eqref{4} is an ordinary differential equation for fixed $\omega\in\Omega$. 
Therefore, by using Theorem \ref{Thm601} for each path, Theorem \ref{Lem204} can be proved directly.
\par
For $b\in \bR$ and $x\in\bR$, we write 
\begin{align}
\kappa^{b, -}_x{:=}\inf \{t> 0 : Y^b_t < x\}, \qquad 
\kappa^{b, +}_x{:=}\inf \{t> 0 : Y^b_t > x\}. 
\end{align}

\subsection{Refracted--reflected L\'evy processes}\label{SecRefRef}
In this section, we recall the refracted--reflected L\'evy processes that are defined in \cite[Section 2.2]{PerYam2018}. 
For a L\'evy process $X$ satisfying Assumption \ref{Ass202a}, we define the refracted--reflected process $Z^b{:=}\{Z^b_t:t\geq 0\}$ at $b> 0$ as follows. 
\begin{itemize}
\item[Step $0$] Set $Z^b_{0-}=X_0$ and $\bar{T}=\un{T}=0$. If $X_0>0$, go to Step $1$. If $X_0\leq 0$, go to Step $2$. 
\item[Step $1$] Let $\{A_t: t\geq \bar{T}\}$ be the refracted L\'evy process at $b$ that satisfies 
\begin{align}
A_t=Z^b_{\bar{T}-}+(X_t -X_{\bar{T}-})-\int_{\bar{T}}^th^b(A_{{s}})d{s}, \quad t\geq \bar{T}.
\end{align}
We redefine $\un{T}{:=}\inf\{t>\bar{T}: A_t\leq 0\}$. We set $Z^b_t = A_t$ for $t\in[\bar{T}, \un{T})$. 
Go to Step $2$. 
\item[Step $2$] Let $\{A_t: t\geq \un{T}\}$ be the process that satisfies 
\begin{align}
A_t=X_t-X_{\un{T}} -\inf_{{s}\in[\un{T}, t]}(X_{{s}}-X_{\un{T}}), \quad t\geq \un{T}. 
\end{align}
We redefine $\bar{T}{:=}\inf\{t>\un{T}: A_t\geq b\}$. We set $Z^b_t = A_t$ for $t\in[\un{T}, \bar{T})$. 
Go to Step $1$. 
\end{itemize}
From the definition above, it is easy to check that $Z^b$ satisfies
\begin{align}
Z^b_t=X_t -\int_0^t h^b(Z^b_s) ds -\inf_{u\in[0,t]} \rbra{\rbra{X_u -\int_0^u h^b(Z^b_s) ds }\land 0}, \quad t\geq 0.\label{125}
\end{align}
\par
We define the refracted--reflected L\'evy process at $0$ as 
\begin{align}
Z^{0}_t=X^{(\alpha)}_t-\inf_{s\in[0, t]}\rbra{X^{(\alpha)}_s\land 0}, \quad t\geq 0. \label{36}
\end{align}
\begin{Rem}\label{Rem206}
By \cite[Remark 3]{KypLoe2010}, the refracted L\'evy process constructed from the spectrally negative L\'evy process $X$ in Case $1$ has the strong Markov property with respect to $\bF$. 
By the same argument as \cite[Remark 3]{KypLoe2010}, we can check that the refracted L\'evy processes and the refracted--reflected L\'evy processes defined in this paper have the strong Markov property. 
\end{Rem}
\section{Optimal dividend problem with capital injection} \label{OptimalDividend}
In this section, we describe the problem in this paper and recall the refracted--reflected strategies. 
\subsection{Setting of the problem}\label{SubSec301}
We fix 
the cost per unit injected capital $\beta>1$ and the discount rate $q>0$. 
In this paper, a strategy $\pi$ is a set of the processes $L^\pi{:=}\{L^\pi_t:t\geq0\}$ and $R^\pi{:=}\{R^\pi_t:t\geq0\}$
satisfying the following conditions. 
\begin{enumerate}
\item There exists the process $l^\pi{:=}\{l^\pi_t:t\geq0\}$ that is progressively measurable with respect to $\bF$ 
and satisfies 
\begin{align}
l^\pi_t \in [0, \alpha], \quad L^\pi_t=\int_0^t l^\pi_s ds, \quad t\geq 0. \label{5}
\end{align}
\item The process $R^\pi$ is $\bF$-adapted, non-decreasing, right-continuous, and satisfies
\begin{align}
X_t - L^\pi_t+R^\pi_t \geq 0, \quad t\geq 0. \label{91}
\end{align} 
\end{enumerate}
Then, the resulting process $\{U^\pi_t: t\geq 0\}$ of the strategy $\pi$ satisfies 
\begin{align}
U^\pi_t =X_t-L^\pi_t +R^\pi_t= X_t-\int_0^t l^\pi_s ds+R^\pi_t, \quad t\geq 0. \label{Rev020}
\end{align}
The expected NPV $v_\pi$ of the dividend payments minus the cost of capital injections when we use the strategy $\pi$ is written as 
\begin{align}
v_\pi(x){:=} \bE_x \sbra{\int_{[0, \infty)}e^{-qt} dL^\pi_t -\beta\int_{[0, \infty)} e^{-qt}dR^\pi_t } =\bE_x \sbra{\int_0^\infty e^{-qt} l^\pi_tdt -\beta\int_{[0, \infty)} e^{-qt}dR^\pi_t}. 
\end{align}
Let $\cA$ be the set of all strategies that satisfy 
\begin{align}
\bE_x \sbra{\int_{[0, \infty)} e^{-qt} dR^\pi_t}< \infty, \quad x\in\bR.  \label{90}
\end{align}
The purpose of this study is to find a strategy $\pi^\ast\in\cA$ that meets the following condition: 
\begin{align}
v_{\pi^\ast}(x) = \sup_{\pi\in\cA}v_\pi(x), \quad x \in \bR.  \label{3}
\end{align}
The strategy $\pi^\ast$ satisfying \eqref{3} is called an optimal strategy. 
\begin{Rem}\label{Rem203}
Since all strategies $\pi$ satisfy \eqref{5}, we have 
\begin{align}
\sup_{x\in\bR}v_\pi(x) \leq\alpha \int_0^\infty e^{-qt}dt =\frac{\alpha}{q}. \label{6}
\end{align}
\end{Rem}
\begin{Rem}\label{Rem302}
If $X$ does not satisfy \eqref{61}, then for all strategies $\pi$, 
we have  
\begin{align}
\bE_x \sbra{\int_{[0, \infty)} e^{-qt} dR^\pi_t}= \infty, \quad x\in\bR, \label{92}
\end{align} 
and $\cA$ is an empty set. 
To demonstrate, we assume that $X$ does not satisfy \eqref{61} and fix a strategy $\pi$ and a value $x\in\bR$. 
It is sufficient to prove $\bE_x \sbra{R^\pi_1}=\infty$ to show \eqref{92}. 
Since $\pi$ satisfies \eqref{91}, we have 
\begin{align}
-\bE_x\sbra{\inf_{s\in[0,1]} (X_s \land 0)}\leq  \bE_x \sbra{R^\pi_1}. \label{93}
\end{align}
We define the stopping time $\tau^{J, -}_{ -1}$ as 
\begin{align}
\tau^{J, -}_{ -1}{:=}\inf\{t>0 : X_t- X_{t-}<-1\}. 
\end{align}
By the definition of $X$, there exists the value $d\in\bR$ that satisfies 
\begin{align}
\bP_x \rbra{B_d }>0, 
\end{align}
where 
\begin{align}
B_d = \cbra{\sup_{s\in[0, \tau^{J, -}_{ -1}]}X_s <d,\tau^{J, -}_{ -1}\leq 1}\in \cF. 
\end{align}
Since $X_{\tau^{J, -}_{ -1}}-X_{\tau^{J, -}_{ -1}-}$ has the distribution $\frac{\Pi( \cdot \cap (-\infty , -1) )}{\Pi(-\infty , -1)}$ and is independent from $\sup_{s\in[0, \tau^{J, -}_{ -1}]}X_s$ and $\tau^{J,-}_{-1}$, we have 
\begin{align}
\bE_x\sbra{\inf_{s\in[0,1]} (X_s \land 0); B_d}\leq& \bE_x\sbra{d+X_{\tau^{J,-}_{-1}}-X_{\tau^{J,-}_{-1}-}; B_d}\\
\leq&\bP_x\rbra{B_d}\frac{1}{\Pi(-\infty , -1)} \int_{(-\infty , -1)} (d+y) \Pi(dy) =-\infty, \label{94}
\end{align}
where in the last equality we used the fact that $X$ does not satisfy \eqref{61}. 
By \eqref{93} and \eqref{94}, we obtain $\bE_x \sbra{R^\pi_1}=\infty$. 
\end{Rem}

\subsection{Refraction--reflection strategies}\label{Sec302}
The refraction--reflection strategy $\pi^b$ at $b\geq 0$ 
is defined as 
\begin{align}
l^{\pi^b}_t {:=}&
h^{b}(Z^b_t), 
\quad
R^{\pi^b}_t {:=}
-\inf_{s\in[0, t]} \rbra{\rbra{ X_s -L^{\pi^b}_s  } \land 0}. 
\label{38}
\end{align}
Then, the resulting process $U^{\pi^b}$ {with $b>0$} corresponds to $Z^b$ by \eqref{125}. 
{In addition,} we confirm that the process $U^{\pi^0}$ corresponds to the process $Z^0$ in Appendix \ref{Sec00B}.
Thus, 
the refraction--reflection strategy $\pi^b$ is a strategy such that the company pays dividends of $ \alpha $ or $ \delta $ when the surplus process exceeds or equals 
$b$ and receives a capital injection when it falls below $0$. 
From this fact, the refraction--reflection strategies exist.  
\begin{Lem}
The refraction--reflection strategy $\pi^b$ at $b\geq 0$ belongs to $\cA$. 
\end{Lem}
\begin{proof}
From \eqref{38} and Remark \ref{Rem206}, $l^{\pi^b}$ and $R^{\pi^b}$ are $\bF$-adapted. 
In addition, $l^{\pi^b}$ is progressively measurable since $Z^b$ is right-continuous. 
Thus, conditions (i) and (ii) in Subsection \ref{SubSec301} hold. 
\eqref{90} holds by Assumption \ref{Ass201} and the same argument as that in the proof in \cite[Lemma 3.2]{Nob2021}. 
The proof is complete. 
\end{proof}
For simplicity, we write $v_b$ for $v_{\pi^b}$ for $b\geq 0$.

\section{Some properties of the function $v_{b^\ast}$} \label{bast}
In this section, we define the optimal barrier $b^\ast$ and compute the derivative of $v_{b^\ast}$. 
This section is necessary to prove the optimality of the refracted--reflected strategies in Section \ref{Sec04}.
\par
We define the value $b^\ast$ by 
\begin{align}
b^\ast {:=} \inf \cbra{ b>0 :  \beta \nu(b) < 1},  
\end{align}
where $\nu(b ) {:=} \bE_b\sbra{e^{-q\kappa^{b, -}_0 }}$ with $e^{-\infty}=0$. 
\begin{Rem}
The value $b^\ast$ is finite since $\lim_{b\uparrow \infty}\nu(b)=0$. 
\end{Rem}
\begin{Rem}\label{Rem401}
If $0$ is regular for $(-\infty , 0)$ for $Y^0$, then $\nu(0)=1$. In addition, it is easy to check that $\nu$ is right-continuous at $0$. 
Thus, $b^\ast >0$. 
\end{Rem}
{
\begin{Rem}\label{Rem403}
Let we define 
\begin{align}
\hat{b}^{\ast}:= \inf \cbra{ b>0 :  \beta \nu(b) \leq  1}. 
\end{align}
Thereafter we will show that the refraction--reflection strategy at $b^\ast$ is an optimal strategy, but we can prove the optimality of  the refraction--reflection strategy at $b\in[\hat{b}^\ast, b^\ast]$ by the same argument.
\end{Rem}
}
\par
It is easy to check that 
$\nu$ is right-continuous and may not be left-continuous. 
In this case, $\beta \nu(b^\ast)$ may be less than $1$. 
Only when $\beta \nu(b^\ast)<1$, we may need to define the value $p^\ast \in [0, 1)$ as follows. 
If only the case $\beta \nu(b^\ast)=1$ is required, then one can go to Lemma \ref{Lem302} directly. 
\par
We define 
\begin{align}
T^{b^\ast, -}_0{:=}\inf\{t {\geq} 0: Y^{b^\ast}_t\leq 0\}, 
\end{align}
and for $p\in[0, 1]$, we define the independent random variable $V_p$ as 
\begin{align}
V_p {:=} 
\left\{     
\begin{array}{l}
0 \quad \text{with probability } 1-p,\\
1 \quad \text{with probability } p. 
\end{array}
  \right.
\end{align}
In addition, we define 
\begin{align}
K^p_0 {:=} \kappa^{b^\ast, -}_0 1_{\{V_p=1\}}+ T^{b^\ast,-}_0 1_{\{V_p =0\}}. 
\end{align}
Note that $\kappa^{b^\ast,-}_0= K^1_0$, and if $b^\ast >0$, then we have $\beta\bE_{b^\ast} \sbra{e^{-q K^{0}_0}}=\beta\bE_{b^\ast} \sbra{e^{-q  T^{b^\ast,-}_0}}=\lim_{b\uparrow b^\ast}\beta \nu (b)$ and thus
\begin{align}
\beta\bE_{b^\ast} \sbra{e^{-q K^{0}_0}}\geq 1. \label{a004}
\end{align} 
In addition, if $b^\ast=0$
, then $T^{b^\ast,-}_0=0$ $\bP_0$-a.s. and thus we have \eqref{a004}. 
Therefore, we have the following lemma.
\begin{Lem}\label{Lem301}
We assume that $\beta\nu(b^\ast)< 1$ and $X$ satisfies one of the following. 
\begin{enumerate}
\item $b^\ast>0$; 
\item $b^\ast =0$ and $0$ is regular for itself for $Y^0$. 
\end{enumerate}
Then, there exists $p^\ast\in[0, 1)$ satisfying 
\begin{align}
\beta\bE_{b^\ast} \sbra{e^{-q K^{p^\ast}_0}}=1. \label{a05}
\end{align}
\end{Lem}
\begin{Rem}
Under assumption (ii) of Lemma \ref{Lem301}, 
$X$ belongs to Case $2$. 
In fact, if $X$ has bounded variation paths with $\delta > \alpha$, then $0$ is irregular for $(-\infty , 0]$ for $Y^0$, which contradicts the assumptions, and if $X$ has bounded variation paths with $\delta < 0$ or has unbounded variation paths, then $0$ is regular for $(-\infty , 0)$ {for $Y^0$}, which contradicts Remark \ref{Rem401}. 
\end{Rem}
\begin{proof}[Proof {of Lemma \ref{Lem301}}]
We have 
\begin{align}
\beta\bE_{b^\ast} \sbra{e^{-q K^{p}_0}}
= &\beta\bE_{b^\ast} \sbra{e^{-q\kappa^{b^\ast,-}_0}; V_p =1}
+\beta\bE_{b^\ast} \sbra{e^{-qT^{b^\ast,-}_0}; V_p =0} \\
=& 
p\beta\bE_{b^\ast} \sbra{e^{-q\kappa^{b^\ast,-}_0}}
+(1-p)\beta\bE_{b^\ast} \sbra{e^{-qT^{b^\ast,-}_0}}.
\label{a03}
\end{align}
By $\beta \nu (b^\ast)<1$, \eqref{a004}, and \eqref{a03}, the proof is complete. 
\end{proof}
When $X$ satisfies $\beta\nu(b^\ast)=1$ or $b^\ast =0$ and $0$ is irregular for itself for $Y^0$, we assume that $p^\ast = 1$. {Thus \eqref{a05} holds when $X$ satisfies condition (i) or (ii) in Lemma \ref{Lem301}.}
\begin{Lem} \label{Lem302}
We assume that $X$ satisfies condition (i) or (ii) in Lemma \ref{Lem301}. 
The function $v_{b^\ast}$ is concave, and the function 
\begin{align}
\un{\nu}(x){:=} \beta\bE_x\sbra{e^{-qK^{ p^\ast}_0}   }, \quad x\in\bR, \label{a015}
\end{align}
is its density with respect to the Lebesgue measure. 
\end{Lem}
\begin{Rem}
Note that we do not assume that $\beta\n{\nu}(b^\ast)<1$ in Lemma \ref{Lem302} {unlike Lemma \ref{Lem301}}. 
\end{Rem}
\begin{proof}[Proof of Lemma \ref{Lem302}]
{
For $\varepsilon\in\bR$
, we write $X^{[\varepsilon]}_t {:=} X_t+\varepsilon$ and $X^{[\varepsilon], (\alpha)}_t {:=} X_t+\varepsilon- \alpha t$. 
In addition, we write $L^{[\varepsilon], b}{:=}\{L^{[\varepsilon], b}_t: t\geq 0\}$ and $R^{[\varepsilon], b}{:=}\{R^{[\varepsilon], b}_t: t\geq 0\}$ for the cumulative amounts of dividends and capital injections, respectively, when we impose the refraction--reflection strategy at $b$ on $X^{[\varepsilon]}$ and write $Z^{[\varepsilon], b}{:=}\{Z^{[\varepsilon], b}_t:= X^{[\varepsilon]}_t - L^{[\varepsilon], b}_t+ R^{[\varepsilon], b}_t:t\geq 0\}$. 
Note that $Z^{[\varepsilon], b}$ behaves as a refracted--reflected L\'evy process. 
We write ``$K^{[\varepsilon], p^\ast}_0$'' and ``$T^{[\varepsilon], b^\ast}_0$'' for ``$K^{p^\ast}_0$'' and ``$T^{b^\ast,-}_0$'' of $X^{[\varepsilon]}$. 
}
\par
For $x, \varepsilon \in\bR$, we have 
\begin{align}
\begin{aligned}
v_{b^\ast}(x+\varepsilon) -& v_{b^\ast}(x)\\
=\bE_x &\sbra{\int_0^\infty e^{-qt} d(L^{[\varepsilon], b^\ast}_{t}-L^{[0], b^\ast}_{t })}
+\beta\bE_x \sbra{\int_0^\infty e^{-qt} d(R^{[0], b^\ast}_{t}-R^{[\varepsilon], b^\ast}_{t })}. 
\end{aligned}
\label{a006}
\end{align}
Note that for sufficiently small $\varepsilon>0$, by 
Lemmas \ref{Lem_paths} and \ref{Lem_paths_2}, the Stieltjes integrals using processes 
$\{L^{[\varepsilon], b^\ast}_{t}-L^{[0], b^\ast}_{t } : t\geq 0\}$ and $\{R^{[0], b^\ast}_{t}-R^{[\varepsilon], b^\ast}_{t }:t\geq 0\}$, which are monotone processes, are well-defined. 
\par
(i) We want to give bounds of \eqref{a006} to compute the right derivative of $v_{b^\ast}$. 
We fix $\varepsilon>0$ and $\eta>0$. We inductively define stopping times ${\{S^{(n)}_{[\eta]}\}}_{n\in\bN}$, ${\{T^{(n)}_{[\eta]}\}}_{n\in\bN}$ and values ${\{\varepsilon^{(n)}_{[\eta]}\}}_{n\in\bN}$ as 
\begin{align}
S^{(n)}_{[\eta]}{:=} &
\inf \cbra{t > T^{(n-1)}_{[\eta]}: Z^{[0], b^\ast}_{t}\leq b^\ast, Z^{[\varepsilon], b^\ast}_{t}\geq b^\ast}, 
\quad
T^{(n)}_{[\eta]}{:=} 
S^{(n)}_{[\eta]}+\eta,  
\end{align}
\begin{align}
\varepsilon^{(n)}_{[\eta]}{:=}&
\rbra{L^{[\varepsilon], b^\ast}_{T^{(n)}_{[\eta]}}-L^{[0], b^\ast}_{T^{(n)}_{[\eta]}}}-\rbra{L^{[\varepsilon], b^\ast}_{S^{(n)}_{[\eta]}}-L^{[0], b^\ast}_{S^{(n)}_{[\eta]}}},
\end{align} 
where $T^{(0)}_{[\eta]}{:=}0$. 
From Lemmas \ref{Lem_paths} {(3)} and \ref{Lem_paths_2}, we have 
\begin{align}
\varepsilon^{(n)}_{[\eta]}\geq 0, \quad n\in\bN,\label{Rev002}
\end{align} 
and {$L^{[\varepsilon], b^\ast}_{t}-L^{[0], b^\ast}_{t }$ may increases only on $\cup_{n\in\bN} [S^{(n)}_\eta, T^{(n)}_\eta]$, thus we have} 
\begin{align}
\lim_{t\to\infty } (L^{[\varepsilon], b^\ast}_{t}-L^{[0], b^\ast}_{t }) =\sum_{n\in{\bN}}\varepsilon^{(n)}_{[\eta]}. \label{Rev003}
\end{align}
{Note that, by Lemmas \ref{Lem_paths} (2), (4) and \ref{Lem_paths_2}, the values $Z^{[\varepsilon],b^\ast}_t-Z^{[0],b^\ast}_t$ and $-(R^{[\varepsilon],b^\ast}_t-R^{[0],b^\ast}_t)$ are non-negative, thus by Lemmas \ref{Lem_paths} (1) and \ref{Lem_paths_2}, we have $\sum_{n\in{\bN}}\varepsilon^{(n)}_{[\eta]}\leq \varepsilon$.} 
In addition, since {$Z^{[\varepsilon], b^\ast}_{T^{[\varepsilon], b^\ast}_0}=0$ and $Z^{[0], b^\ast}$ takes non-negative values and by Lemmas \ref{Lem_paths} (2) and \ref{Lem_paths_2}, we have
\begin{align}
Z^{[\varepsilon], b^\ast}_{t}-Z^{[0], b^\ast}_{t}=0, \qquad t\geq T^{[\varepsilon], b^\ast}_0. \label{Res001}
\end{align}
} 
{From Lemma \ref{Lem_paths} (3), (4) and \ref{Lem_paths_2}, the processes $\{L^{[\varepsilon], b^\ast}_{t}-L^{[0], b^\ast}_{t }: t\geq 0\}$ and $\{-(R^{[\varepsilon], b^\ast}_{t}-R^{[0], b^\ast}_{t }):t\geq 0\}$ are non-decreasing, thus by \eqref{Res001} and Lemmas \ref{Lem_paths} (1) and \ref{Lem_paths_2},} 
\begin{align}
 L^{[\varepsilon], b^\ast}_{t}-L^{[0], b^\ast}_{t }{\text{ and }- (R^{[\varepsilon], b^\ast}_{t}-R^{[0], b^\ast}_{t })}\text{ {do} not increase after time }{T^{[\varepsilon], b^\ast}_0}. \label{Rev004}
 \end{align} 
 {From \eqref{Rev002}, \eqref{Rev003} and \eqref{Rev004}
 we have}
\begin{align}
\bE_x \sbra{ \sum_{n\in\bN} e^{-q \rbra{T^{(n)}_{[\eta]} \land {T^{[\varepsilon], b^\ast}_0}}}\varepsilon^{(n)}_{[\eta]} }
\leq &\bE_x \sbra{\int_0^\infty e^{-qt} d(L^{[\varepsilon], b^\ast}_{t}-L^{[0], b^\ast}_{t })}\\
&\quad \leq\bE_x \sbra{ \sum_{n\in\bN} e^{-q S^{(n)}_{[\eta]}}\varepsilon^{(n)}_{[\eta]} }. 
\label{a009}
\end{align}
Since $R^{\pi^b}$ satisfies \eqref{38}, 
$R^{[0], b^\ast}_t-R^{[\varepsilon], b^\ast}_t$ may start to increase at time $K^{[0], p^\ast}_0$. 
In addition, it does not increase after time $K^{[\varepsilon], p^\ast}_0{\geq  T^{[\varepsilon], b^\ast}_0}$ {by \eqref{Rev004}}. 
Furthermore, we have $\lim_{t\to\infty } (R^{[0], b^\ast}_{t}-R^{[\varepsilon], b^\ast}_{t }) \leq\varepsilon-\sum_{n\in\bN}\varepsilon^{(n)}_{[\eta]} $, which holds with ``$=$'' on $\{ K^{[\varepsilon], p^\ast}_0< \infty\}$, thus we have
\begin{align}
\bE_x\sbra{e^{-q K^{[\varepsilon], p^\ast}_0}  \rbra{\varepsilon - \sum_{n\in\bN} \varepsilon^{(n)}_{[\eta]}   } }
\leq &\bE_x \sbra{\int_0^\infty e^{-qt} d(R^{[0], b^\ast}_{t}-R^{[\varepsilon], b^\ast}_{t })}\\
&\quad \leq\bE_x\sbra{e^{-qK^{[0], p^\ast}_0}  \rbra{\varepsilon - \sum_{n\in\bN} \varepsilon^{(n)}_{[\eta]}   } }
.
\label{a007}
\end{align}
Using the strong Markov property at $T^{(n)}_{[\eta]}\land {T^{[\varepsilon], b^\ast}_0}$, and since $L^{[\varepsilon], b^\ast}_{t}-L^{[0], b^\ast}_{t }$ does not increase after time {$T^{[\varepsilon], b^\ast}_0$}
, we have 
\begin{align}
\bE_x\sbra{e^{-q K^{[\varepsilon], p^\ast}_0} \varepsilon^{(n)}_{[\eta]}    }
&=\bE_x\sbra{e^{-q \rbra{T^{(n)}_{[\eta] } \land {T^{[\varepsilon], b^\ast}_0}}} \varepsilon^{(n)}_{[\eta]} \un{\nu}\rbra{Z^{[\varepsilon]}_{T^{(n)}_{[\eta]} \land {T^{[\varepsilon], b^\ast}_0}}}  } . \label{a012}
\end{align}
Using the strong Markov property at $T^{(n)}_{[\eta]}$, 
we have
\begin{align}
\bE_x\sbra{e^{-q K^{[0], p^\ast}_0} \varepsilon^{(n)}_{[\eta]}    }
&=\bE_x\sbra{e^{-q T^{(n)}_{[\eta]} } \varepsilon^{(n)}_{[\eta]} \un{\nu}\rbra{Z^{[0]}_{T^{(n)}_{[\eta]}}}  ;T^{(n)}_{[\eta]}<  K^{[0], p^\ast}_0 }\\
& \qquad +
\bE_x\sbra{e^{-q K^{[0], p^\ast}_0} \varepsilon^{(n)}_{[\eta]}    ;  K^{[0], p^\ast}_0<T^{(n)}_{[\eta]}}
\\
&\geq
\bE_x\sbra{e^{-q T^{(n)}_{[\eta]} } \varepsilon^{(n)}_{[\eta]} \un{\nu}\rbra{Z^{[0]}_{T^{(n)}_{[\eta]}}}   }.\label{a008}
\end{align}
Using \eqref{a012} and \eqref{a008}, we have 
\begin{align}
\begin{aligned}
\bE_x \sbra{  e^{-q \rbra{T^{(n)}_{[\eta]} \land {T^{[\varepsilon], b^\ast}_0}}}\varepsilon^{(n)}_{[\eta]} } &-
\beta\bE_x\sbra{e^{-q K^{[\varepsilon], p^\ast}_0}   \varepsilon^{(n)}_{[\eta]}    }\\
=  
 &\bE_x\sbra{e^{-q \rbra{T^{(n)}_{[\eta]}\land {T^{[\varepsilon], b^\ast}_0}}} \varepsilon^{(n)}_{[\eta]} \rbra{1-\beta\un{\nu}\rbra{Z^{[\varepsilon]}_{T^{(n)}_{[\eta]} \land {T^{[\varepsilon], b^\ast}_0}}} }}
\end{aligned}
\label{a010}
\end{align}
and 
\begin{align}
&\bE_x \sbra{  e^{-q S^{(n)}_{[\eta]}}\varepsilon^{(n)}_{[\eta]} } -
\beta\bE_x\sbra{e^{-q K^{[0], p^\ast}_0}   \varepsilon^{(n)}_{[\eta]}    }\\
&\leq
\bE_x\sbra{\rbra{e^{-q S^{(n)}_{[\eta]}} -e^{-qT^{(n)}_{[\eta]}}}\varepsilon^{(n)}_{[\eta]} } +
\bE_x\sbra{e^{-q T^{(n)}_{[\eta]} } \varepsilon^{(n)}_{[\eta]} \rbra{1-\beta\un{\nu}\rbra{Z^{[0]}_{T^{(n)}_{[\eta]}}} }}
\\
&\leq
\rbra{1-e^{-q \eta}}\bE_x\sbra{\varepsilon^{(n)}_{[\eta]} }
+
\bE_x\sbra{e^{-q  T^{(n)}_{[\eta]} } \varepsilon^{(n)}_{[\eta]} \rbra{1-\beta\un{\nu}\rbra{Z^{[0]}_{T^{(n)}_{[\eta]} }} }}
.
\label{a011}
\end{align}
From \eqref{a006}, \eqref{a009}, \eqref{a007}, \eqref{a010}, and \eqref{a011}, 
we have, for $\eta>0$, 
\begin{align}
\begin{aligned}
\varepsilon&\beta\un{\nu}\rbra{x+\varepsilon}+
\sum_{n\in\bN}\bE_x\sbra{e^{-q \rbra{ T^{(n)}_{[\eta]} \land {T^{[\varepsilon], b^\ast}_0}}} \varepsilon^{(n)}_{[\eta]} \rbra{1-\beta\un{\nu}\rbra{Z^{[\varepsilon]}_{T^{(n)}_{[\eta]}\land {T^{[\varepsilon], b^\ast}_0}}} }}\\
&\leq {{v_{b^\ast}(x+\varepsilon) -v_{b^\ast}(x)}}
\leq \varepsilon\beta \un{\nu}(x)
+\rbra{1-e^{-q \eta}}\sum_{n\in\bN}\bE_x\sbra{\varepsilon^{(n)}_{[\eta]} }\\
&\qquad+\sum_{n\in\bN}\bE_x\sbra{e^{-q T^{(n)}_{[\eta]}} \varepsilon^{(n)}_{[\eta]} \rbra{1-\beta\un{\nu}\rbra{Z^{[0]}_{T^{(n)}_{[\eta]} }}}}. 
\end{aligned}
\label{a016}
\end{align}
\par
(ii) We 
prove that for $\varepsilon >0$ and $x\in\bR$, 
\begin{align}
\varepsilon\beta\un{\nu}\rbra{x+\varepsilon}  \leq {{v_{b^\ast}(x+\varepsilon) -v_{b^\ast}(x)}}
.
\label{a013}
\end{align}
By Fubini's theorem, we have
\begin{align}
&\sum_{n\in\bN}\bE_x\sbra{e^{-q \rbra{ T^{(n)}_{[\eta]} \land {T^{[\varepsilon], b^\ast}_0}}} \varepsilon^{(n)}_{[\eta]} \rbra{1-\beta\un{\nu}\rbra{Z^{[\varepsilon]}_{T^{(n)}_{[\eta]} \land {T^{[\varepsilon], b^\ast}_0}}} }}\\
&
\begin{aligned}
& = 
\bE_x\Bigg{[}\int_0^\infty
\sum_{n\in\bN}\Bigg{(} e^{-q \rbra{ T^{(n)}_{[\eta]}\land {T^{[\varepsilon], b^\ast}_0} }} 1_{[S^{(n)}_{[\eta]} \land {T^{[\varepsilon], b^\ast}_0}, T^{(n)}_{[\eta]}\land {T^{[\varepsilon], b^\ast}_0})}(t)  \\
&\qquad \times\frac{\varepsilon^{(n)}_{[\eta]}}{T^{(n)}_{[\eta]}\land {T^{[\varepsilon], b^\ast}_0}-S^{(n)}_{[\eta]} \land {T^{[\varepsilon], b^\ast}_0}} 
\rbra{1-\beta\un{\nu}\rbra{Z^{[\varepsilon]}_{T^{(n)}_{[\eta]}\land {T^{[\varepsilon], b^\ast}_0} }} }\Bigg{)}dt\Bigg{]}, 
\end{aligned}
\label{a018}
\end{align} 
where $\frac{0}{0}=0$. 
{In fact, from} the definitions of $S^{(n)}_{[\eta]}
$, $T^{(n)}_{[\eta]}$, {$T^{[\varepsilon], b^\ast}_0$}, and $\varepsilon^{(n)}_{[\eta]}$, we have
\begin{align}
\frac{\varepsilon^{(n)}_{[\eta]}}{T^{(n)}_{[\eta]}\land {T^{[\varepsilon], b^\ast}_0}-S^{(n)}_{[\eta]} \land {T^{[\varepsilon], b^\ast}_0}}\leq \alpha, \qquad n\in\bN, ~\eta>0,\label{a029}
\end{align}
and thus, for $t\geq 0$, 
\begin{align}
\begin{aligned}
&\sum_{n\in\bN} {\Bigg{|}}e^{-q \rbra{T^{(n)}_{[\eta]}\land {T^{[\varepsilon], b^\ast}_0}} } 1_{[S^{(n)}_{[\eta]} \land {T^{[\varepsilon], b^\ast}_0}, T^{(n)}_{[\eta]}\land {T^{[\varepsilon], b^\ast}_0})}(t)  \\
&
\times\frac{\varepsilon^{(n)}_{[\eta]}}{T^{(n)}_{[\eta]}\land {T^{[\varepsilon], b^\ast}_0}-S^{(n)}_{[\eta]} \land {T^{[\varepsilon], b^\ast}_0}}
\rbra{1-\beta\un{\nu}\rbra{Z^{[\varepsilon]}_{T^{(n)}_{[\eta]} \land {T^{[\varepsilon], b^\ast}_0}}} }\Bigg{|}
\leq e^{-qt}\alpha (1+\beta). 
\end{aligned}
\label{a017}
\end{align}
{By \eqref{a017}, the term in absolute value of the left-hand side of \eqref{a017} is integrable for $\sum_{n\in\bN}\bE_x\sbra{\int_0^\infty (\cdot) dt}$, and we could use Fubini's theorem at \eqref{a018}. 
}
{Since the right-hand side of \eqref{a017} is integrable with respect to the measure $\bE_x \sbra{\int_0^\infty (\cdot)dt}$,} we can apply Fatou's lemma to \eqref{a018} regardless of the plus or minus of the term in the integral, and we have 
\begin{align}
&\liminf_{\eta\downarrow0}\sum_{n\in\bN}\bE_x\sbra{e^{-q \rbra{T^{(n)}_{[\eta]} \land {T^{[\varepsilon], b^\ast}_0}}} \varepsilon^{(n)}_{[\eta]} \rbra{1-\beta\un{\nu}\rbra{Z^{[\varepsilon]}_{T^{(n)}_{[\eta] }\land{T^{[\varepsilon], b^\ast}_0}}} }}\\
&
\begin{aligned}
&\geq 
\bE_x\Bigg{[}\int_0^\infty\liminf_{\eta\downarrow0}
\Bigg{(}\sum_{n\in\bN} e^{-q \rbra{T^{(n)}_{[\eta]} \land {T^{[\varepsilon], b^\ast}_0}}} 1_{[S^{(n)}_{[\eta]} \land {T^{[\varepsilon], b^\ast}_0}, T^{(n)}_{[\eta]}\land {T^{[\varepsilon], b^\ast}_0})}(t)  \\
&\qquad \times\frac{\varepsilon^{(n)}_{[\eta]}}{T^{(n)}_{[\eta]}\land {T^{[\varepsilon], b^\ast}_0}-S^{(n)}_{[\eta]} \land{T^{[\varepsilon], b^\ast}_0}}
\rbra{1-\beta\un{\nu}\rbra{Z^{[\varepsilon]}_{T^{(n)}_{[\eta]} \land{T^{[\varepsilon], b^\ast}_0}}} }\Bigg{)}dt\Bigg{]}.
\end{aligned}
\label{a028}
\end{align}
We fix $\omega\in\Omega$ such that $t\mapsto X_t$ is c\`adl\`ag and the differential equation \eqref{4} has a unique solution (almost surely $\omega$ satisfies these conditions), and we prove that
\begin{align}
\begin{aligned}
&\liminf_{\eta\downarrow0}
\Bigg{(}\sum_{n\in\bN} e^{-q \rbra{T^{(n)}_{[\eta]} \land {T^{[\varepsilon], b^\ast}_0}}} 1_{[S^{(n)}_{[\eta]} \land {T^{[\varepsilon], b^\ast}_0}, T^{(n)}_{[\eta]} \land{T^{[\varepsilon], b^\ast}_0})}(t)  \\
&\qquad \times\frac{\varepsilon^{(n)}_{[\eta]}}{T^{(n)}_{[\eta]} \land {T^{[\varepsilon], b^\ast}_0}-S^{(n)}_{[\eta]} \land {T^{[\varepsilon], b^\ast}_0}}
\rbra{1-\beta\un{\nu}\rbra{Z^{[\varepsilon]}_{T^{(n)}_{[\eta]} \land {T^{[\varepsilon], b^\ast}_0}}} }\Bigg{)}\geq 0 
\end{aligned}
\label{a032}
\end{align}
almost everywhere (a.e.) with respect to $t$. 
Since for fixed $t\geq 0$ there is only one $n\in\bN$ where $ 1_{[S^{(n)}_{[\eta]} \land {T^{[\varepsilon], b^\ast}_0}, T^{(n)}_{[\eta]} \land {T^{[\varepsilon], b^\ast}_0})}(t)$ does not equal to $ 0 $ in the summation and by 
\begin{align}
e^{-q \rbra{T^{(n)}_{[\eta]} \land {T^{[\varepsilon], b^\ast}_0}}}\in[0,1], \quad
\frac{\varepsilon^{(n)}_{[\eta]}}{T^{(n)}_{[\eta]} \land {T^{[\varepsilon], b^\ast}_0}-S^{(n)}_{[\eta]} \land {T^{[\varepsilon], b^\ast}_0}}\in [0, \alpha] ,\qquad  \eta > 0, n\in\bN,
\end{align} 
it is sufficient for \eqref{a032} to prove that 
\begin{align}
\liminf_{\eta\downarrow0}
\rbra{\sum_{n\in\bN}  1_{[S^{(n)}_{[\eta]} \land {T^{[\varepsilon], b^\ast}_0}, T^{(n)}_{[\eta]} \land {T^{[\varepsilon], b^\ast}_0})}(t)  
\rbra{1-\beta\un{\nu}\rbra{Z^{[\varepsilon]}_{T^{(n)}_{[\eta]} \land {T^{[\varepsilon], b^\ast}_0}}} }}\geq 0 \label{a030}
\end{align}
a.e.\ with respect to $t$. 
We assume that $Z^{b^\ast}$ does not jump at $t$ ($Z^{b^\ast}$ does not have jumps a.e.\ with respect to $t$) and assume $b^{\ast}\not\in[Z^{[0],b^\ast}_t, Z^{[\varepsilon],b^\ast}_t ] $.
Then, for sufficiently small $\eta$, we have $b^{\ast}\not\in[Z^{[0],b^\ast}_s, Z^{[\varepsilon],b^\ast}_s ] $ for $s\in(t-\eta, t]$, and thus, by the definitions of $S^{(n)}_{[\eta]} \land {T^{[\varepsilon], b^\ast}_0}$ and $T^{(n)}_{[\eta]} \land {T^{[\varepsilon], b^\ast}_0}$, we have 
$t\not\in [S^{(n)}_{[\eta]} \land K^{[\varepsilon], p^\ast}_0, T^{(n)}_{[\eta]} \land K^{[\varepsilon], p^\ast}_0)$ for all $n\in\bN$. Therefore, we have \eqref{a030}. 
We assume that $b^{\ast}\in[Z^{[0],b^\ast}_t, Z^{[\varepsilon],b^\ast}_t ] $. 
Since the function $\un{\nu}$ is non-increasing and since $T^{(n)}_{[\eta]} \land {T^{[\varepsilon], b^\ast}_0}\in[t, t+\eta]$ when $t\in[S^{(n)}_{[\eta]} \land {T^{[\varepsilon], b^\ast}_0}, T^{(n)}_{[\eta]} \land {T^{[\varepsilon], b^\ast}_0})$, we have 
\begin{align}
\sum_{n\in\bN}  1_{[S^{(n)}_{[\eta]} \land {T^{[\varepsilon], b^\ast}_0}, T^{(n)}_{[\eta]} \land {T^{[\varepsilon], b^\ast}_0})}(t)  
\rbra{1-\beta\un{\nu}\rbra{Z^{[\varepsilon]}_{T^{(n)}_{[\eta]} \land {T^{[\varepsilon], b^\ast}_0}}} }
\geq 1-\beta \un{\nu}\rbra{\inf_{s\in[t, t+\eta]} Z^{[\varepsilon]}_s }.\label{Rev023}
\end{align}
We need to consider the following three cases. 
\begin{enumerate}
\item[(ii-a)] We assume that $Z^{[\varepsilon]}_t > b^\ast$. Then, $\inf_{s\in[t, t+\eta]} Z^{[\varepsilon]}_s>b^\ast$ for sufficiently small $\eta$, and thus 
\begin{align}
\liminf_{\eta\downarrow0}\rbra{1-\beta \un{\nu}\rbra{\inf_{s\in[t, t+\eta]} Z^{[\varepsilon]}_s }}
\geq 1-\beta \un{\nu}\rbra{b^\ast }=0, \label{Rev024}
\end{align}
where in the equality we used \eqref{a05}. 
\item[(ii-b)] We assume that $Z^{[\varepsilon]}_t = b^\ast$ and $0$ is regular for $\bR\backslash\{0\}$ for $Y^0$. 
Then, by Lemma \ref{LemC01} and \eqref{a05}, we have 
\begin{align}
\liminf_{\eta\downarrow0}\rbra{1-\beta \un{\nu}\rbra{\inf_{s\in[t, t+\eta]} Z^{[\varepsilon]}_s }}
= 1-\beta \un{\nu}\rbra{b^\ast }=0.\label{a031}
\end{align}
\item[(ii-c)] We assume that $Z^{[\varepsilon]}_t = b^\ast$ and is irregular for $\bR\backslash\{0\}$ for $Y^0$. 
Then, $Z^{[\varepsilon]}$ takes the value $b^\ast$ for a while after hitting $b^\ast$, and after that, $Z^{[\varepsilon]}$ jumps away from $b^\ast$. Thus $\inf_{s\in[t, t+\eta]} Z^{[\varepsilon]}_s=b^\ast$ for sufficiently small $\eta$. Therefore, we have \eqref{a031}. 
\end{enumerate}
From the arguments above, we have \eqref{a032} a.e.\ with respect to $t$. 
From (i), \eqref{a028}, and \eqref{a032}, we have \eqref{a013}. 
\par
(iii) We 
prove that for $\varepsilon >0$ and $x\in\bR$, 
\begin{align}
 {{v_{b^\ast}(x+\varepsilon) -v_{b^\ast}(x)}}
\leq \varepsilon\beta\un{\nu}(x). \label{a033}
\end{align}
From the definition of $\varepsilon^{(n)}_{[\eta]}$, we have 
\begin{align}
\lim_{\eta\downarrow0}\rbra{1-e^{-q \eta}}\sum_{n\in\bN}\bE_x\sbra{\varepsilon^{(n)}_{[\eta]} }
\leq\lim_{\eta\downarrow0}\rbra{1-e^{-q \eta}}\varepsilon=0. \label{a039}
\end{align}
{We replace ``$S^{(n)}_{[\eta]}\land T^{[\varepsilon], b^\ast}_0$'' by ``$S^{(n)}_{[\eta]}$'', ``$T^{(n)}_{[\eta]}\land T^{[\varepsilon], b^\ast}_0$'' by ``$T^{(n)}_{[\eta]}$'', ``$Z^{[\varepsilon]}$'' by ``$Z^{[0]}$''  and ``$\liminf$'' by ``$\limsup$'' from all the equations in (ii) except for \eqref{a013}. 
In addition, we replace ``$\geq $'' by ``$\leq$'' from \eqref{a028}, \eqref{a032}, \eqref{a030}, \eqref{Rev023} and \eqref{Rev024}. 
We can then confirm that these equations in (ii) except for \eqref{a013} are correct by the same arguments as that in (ii).} 
From (i) {and} \eqref{a039}, {and \eqref{a028} and \eqref{a032} with the symbols replaced as above}, we have \eqref{a033}. 
\par
(iv) From \eqref{a013} and \eqref{a033}, the function $v_{b^\ast}$ is right-continuous on $\bR$. 
Similarly, from \eqref{a013} and \eqref{a033} with $x$ changed to $x-\varepsilon$, the function $v_{b^\ast}$ is left-continuous on $\bR$. 
In addition, we have, for $\varepsilon >0$ and $x \in \bR$, 
\begin{align}
 {{v_{b^\ast}(x+\varepsilon) -v_{b^\ast}(x)}} \leq \varepsilon\beta\un{\nu}(x)\leq
  {{v_{b^\ast}(x) -v_{b^\ast}(x-\varepsilon)}}. \label{a014}
  \end{align}
By \eqref{a014}, the function $v_{b^\ast}$ is mid-point concave and thus concave on $\bR$. 
From \eqref{a013} and \eqref{a033}, 
we have 
 \begin{align}
 \beta\un{\nu}\rbra{x+}
\leq \lim_{\varepsilon\downarrow0}\frac {{v_{b^\ast}(x+\varepsilon) -v_{b^\ast}(x)}}{\varepsilon}
\leq \beta\un{\nu}\rbra{x}. \label{a040}
 \end{align}
Since the function $\un{\nu}$ is non-increasing, this function is continuous a.e.\ with respect to the Lebesgue measure 
and 
by \eqref{a040} and \cite[Proposition 3.1 in Appendix]{RevYor1999}, the right derivative and the left derivative of $v_{b^\ast}$ is 
equal to the function $\un{\nu}$ 
a.e.\ with respect to the Lebesgue measure. 
Furthermore, by \cite[Proposition 3.2 in Appendix]{RevYor1999}, the left-derivative of $v_{b^\ast}$ is the density of $v_{b^\ast}$, 
thus \eqref{a015} is the density of $v_{b^\ast}$. The proof is complete. 
\end{proof}
\begin{Lem} \label{Lem405}
We assume that $b^\ast=0$ and $0$ is irregular for itself for $Y^0$. 
The function $v_{b^\ast}$ is concave and the function $\un{\nu}$ is its density with respect to the Lebesgue measure. 
\end{Lem}
\begin{proof}
By Remark \ref{Rem401}, $0$ is irregular for $(-\infty , 0]$ for $Y^0$ and thus for $X^{(\alpha)}$. 
Therefore, by \eqref{87}, \eqref{38}, and the same argument as in the proof of \cite[Lemma {7}]{ NobYam2022}, we have \eqref{a013} and \eqref{a033} {for $x\in\bR$}.  
By the same argument as in (iv) of the proof of Lemma \ref{Lem302}, the proof is complete. 
\end{proof}
Hereinafter, we assume that the function $v_{b^\ast}^\prime$, which usually symbolizes for the derivative, represents the function $\un{\nu}$. 

\section{Verification} \label{Sec04}
In this section, we prove the optimality of the refraction--reflection strategy $\pi^{b^\ast}$ using the verification lemma and results in Section \ref{bast}. 
Before giving the main theorem, we give an assumption. 
\begin{Ass}\label{Ass402}
When $X$ has unbounded variation paths, for $b>0$, the function 
\begin{align}
\nu_{b} (x) {:=} \bE_x \sbra{e^{-q\kappa^{b,- }_0 }} , \quad x>0 
\end{align}
has a locally bounded density $\nu^\prime_b$ on $(0, \infty)$ with respect to the Lebesgue measure. 
The function $\nu^\prime_b$ is continuous a.e.\ on $(0, \infty)$ with respect to the Lebesgue measure. 
\end{Ass}
This assumption is necessary to prove that the function $v_{b^\ast}$ belongs to $C^{(2)}_{\text{line}}$, which is a class of functions that will appear in the proof of the main theorem, and to prove Lemma \ref{Lem505} when $X$ has unbounded variation paths. 
We give examples of the processes with unbounded variation paths satisfying Assumption \ref{Ass402} in Appendix \ref{Ape00D}.
\par
Under Assumptions \ref{Ass201}, \ref{Ass202a}, and \ref{Ass402}, we have the following main theorem. 
\begin{Thm}\label{Thm502}
The refracted--reflected strategy $\pi^{b^\ast}$ is an optimal strategy. 
\end{Thm}

Before giving its proof, 
we give classes of functions and an operator that are almost the same as those defined in \cite[Section 2]{Nob2021}. 
\par
Let $C^{(1)}_{\text{line}}$ be the set of functions $f\in C(\bR)$ satisfying the following conditions.
\begin{enumerate}
\item 
The function $f$ satisfies 
\begin{align}
\absol{f(x)} < a_1 \absol{x\land 0}+ a_2 ,\quad x\in \bR \label{1}
\end{align}
for some $a_1, a_2 >0$. 
\item
The function $f$ has the locally bounded density $f^\prime$ with respect to the Lebesgue measure on $(0, \infty)$, i.e., there exists a locally bounded measurable function $f^\prime$ on $(0, \infty)$ such that 
\begin{align}
f(y)- f(x) = \int_x^y f^\prime (u) du , \quad x, y\in (0, \infty) \text{ with }x<y . 
\end{align}
\end{enumerate}
Let $C^{(2)}_{\text{line}}$ be the set of functions $f\in C^{(1)}_{\text{line}}$ such that $f$ is continuously differentiable on $(0, \infty)$ and 
the derivative $f^\prime$ has the locally bounded density $f^{\prime\prime}$ with respect to the Lebesgue measure on $(0, \infty)$. 
Let $\cL$ be the operator applied to $f \in C^{(1)}_{\text{line}}$ (resp.\ $C^{(2)}_{\text{line}}$) with a fixed density $f^\prime$ (resp.\ a fixed density of the derivative $f^{\prime\prime}$) for the case in which $X$ has bounded (resp.\ unbounded) variation paths with 
\begin{align}
\cL f(x) {:=} \gamma f^\prime (x) + \frac{1}{2} \sigma^2 f^{\prime\prime}(x) +\int_{\bR \backslash \{0\}} \rbra{f(x+z)-f(x)-f^\prime (x)z1_{\{|z|<1\}}}\Pi(dz), \quad x \in (0, \infty). 
\end{align}
This operator is well defined by \cite[Remark 2.4]{Nob2021}. 
\begin{Rem}\label{Rem504}
By the same argument as in \cite[Remark 2.5]{Nob2021}, for $f\in C^{(1)}_{\text{line}}$ (resp.\ $C^{(2)}_{\text{line}}$), the map
\begin{align}
x&\mapsto \int_{\bR \backslash \{0\}}(f(x+z)-f(x))\Pi (dz) \\
\bigg{(}\text{resp.\ }& x\mapsto \int_{\bR \backslash \{0\}}(f(x+z)-f(x)-f^\prime (x) z1_{\{|z|<1\}})\Pi (dz)\bigg{)}
\end{align}
is continuous on $(0, \infty)$ when $X$ has bounded (resp.\ unbounded) variation paths. 
\end{Rem}
\par
Using the operator $\cL$, we can give the following verification lemma that gives the sufficient condition for being an optimal strategy.  
\begin{Lem}\label{Lem401}
Suppose that $X$ has bounded (resp.\ unbounded) variation paths. 
Let $w$ be a function belonging to $C^{(1)}_{\text{line}}$ (resp.\ $C^{(2)}_{\text{line}}$). 
We fix the density $w^\prime$ of $w$ (resp.\ the density $w^{\prime\prime}$ of the derivative $w^\prime$) with respect to the Lebesgue measure and suppose the following. 
\begin{align}
\sup_{r\in[0, \alpha]} \rbra{\cL w (x)-qw(x)-rw^\prime(x)+r}\leq 0, \quad x>0, \label{veri1}\\
w^\prime (x)\leq \beta, \quad x>0, \label{veri2}\\
\inf_{x\in[0, \infty)}w(x)>-m, \quad\text{ for some }m. \label{veri3}
\end{align}
Then $w(x)\geq \sup_{\pi\in\cA}v_\pi(x)$ for $x\in\bR$. 
\end{Lem}
The proof of Lemma \ref{Lem401} is almost the same as the proof in the spectrally positive cases in \cite[Appendix A]{PerYam2016}.  
However, unlike the spectrally positive cases, it uses a more general operator $\cL$ than the infinitesimal generator, and it is necessary to take the same precautions as in \cite[Remark 5.9]{Nob2021}. Thus, we give the proof of Lemma \ref{Lem401} in Appendix \ref{prfverification}. 
\par
For the optimality of the strategy $\pi^{b^\ast}$, it is sufficient to prove that the function $v_{b^\ast}$ satisfies the conditions in Lemma \ref{Lem401}. 
\par
We confirm that $v_{b^\ast}$ belongs to $C^{(1)}_{\text{line}}$ (resp.\ $C^{(2)}_{\text{line}}$) when $X$ has bounded (resp.\ unbounded) variation paths. 
From Remark \ref{Rem203} and since 
\begin{align}
v_{b^\ast}(x)= v_{b^\ast}(0)+\beta x, \quad x < 0, \label{76}
\end{align}
the function $v_{b^\ast}$ satisfies \eqref{1}. 
In addition, by Lemmas \ref{Lem302} and \ref{Lem405}, 
the function $v_{b^\ast}$ belongs to $C^{(1)}_{\text{line}}$. 
By Lemma \ref{Lem302} and Assumption \ref{Ass402}, the function $v_{b^\ast}$ belongs to $C^{(2)}_{\text{line}}$ when $X$ has unbounded variation paths. 
\par
In the following, we prove that $v_{b^\ast}$ satisfies the other conditions in Lemma \ref{Lem401}. 
\begin{Lem}\label{Lem505}
When $X$ has bounded variation paths, we have 
\begin{align}
\cL v_{b^\ast} (x)-qv_{b^\ast} (x)=0, \qquad &x\in (0,  b^\ast],  \label{001} \\
\cL v_{b^\ast} (x) -\alpha v_{b^\ast}^\prime (x)-qv_{b^\ast} (x)+\alpha=0, \qquad &x > b^\ast.   \label{002} 
\end{align}
When $X$ has unbounded variation paths, we can define $v_{b^\ast}^{\prime\prime}$ to satisfy \eqref{001} and \eqref{002} and to be continuous a.e.\ and locally bounded on $(0, \infty)$.
\end{Lem}
\begin{proof}
(i) We prove \eqref{001} {on $(0, b^\ast)$}. In addition, we prove \eqref{001} {at $b^\ast$} when $X$ has unbounded variation paths. 
By the strong Markov property and by the definitions of $L^{\pi^{b^\ast}}$ and $R^{\pi^{b^\ast}}$, we have 
\begin{align}
v_{b^\ast}(x)=
\bE_x \sbra{e^{-q(\tau^-_0\land \tau^+_{b^\ast})}v_{b^\ast}(X_{\tau^-_0 \land \tau^+_{b^\ast}})  }, \quad x\in\bR.\label{003}
\end{align}
The process $\cbra{e^{-q(t\land\tau^-_0\land\tau^+_{b^\ast})}v_{b^\ast}(X_{t\land\tau^-_0\land\tau^+_{b^\ast}}):t\geq 0}$ is a martingale for $\bP_x$ with $x \in{(0, b^\ast)}$. 
In fact, by Remark \ref{Rem203}, the monotonicity of $v_{b^\ast}$, the compensation formula of Poisson point processes, \eqref{76}, and \eqref{61}, we have, for $x \in (0, b^\ast)$ and $t\geq 0$, 
\begin{align}
\bE_x& \sbra{\absol{e^{-q(t\land \tau^-_0\land\tau^+_{b^\ast})}v_{b^\ast}(X_{t\land \tau^-_0\land\tau^+_{b^\ast}})}}
\leq \frac{\alpha}{q} + \absol{\bE_x \sbra{e^{-q\tau^-_0}
v_{b^\ast}(X_{\tau^-_0}) \rbra{1_{\{X_{\tau^-_0}\geq -1\}}+1_{\{X_{\tau^-_0}<-1\}}} }}\\
&\leq \frac{\alpha}{q} + \rbra{\frac{\alpha}{q} \lor|v_{b^\ast}(-1)| + {\bE_x \sbra{\int_{[0, \infty)\times ( \infty, -1 )}e^{-qt} 
\absol{v_{b^\ast}(X_{t-}+y)}1_{\{X_{t-} \geq 0 \}}
\cN(d t\times dy ) }}}\\
&= \frac{\alpha}{q} + \rbra{\frac{\alpha}{q} \lor|v_{b^\ast}(-1)| + {\bE_x \sbra{\int_0^\infty dt \int_{(-\infty ,-1 )}e^{-qt} 
\absol{v_{b^\ast}(X_{t-}+y)}1_{\{X_{t-} \geq 0 \}}
\Pi ( dy ) }}}\\
&\leq \frac{\alpha}{q} + \rbra{\frac{\alpha}{q} \lor|v_{b^\ast} (-1)|+{ \int_{(-\infty, -1)}\frac{1}{q}\rbra{\frac{\alpha}{q}\lor\absol{v_{b^\ast}(0)+\beta y}}\Pi(dy)}}<\infty, 
\end{align}
where $\cN$ is the Poisson random measure on $([0, \infty)\times \bR, \cB [0, \infty)\times\cB (\bR))$ associated with $ds \times \Pi(dx)$.
In addition, we have, for $x \in (0, b^\ast)$ and $t\geq 0$, 
\begin{align}
&\bE_x \sbra{e^{-q(\tau^-_0\land\tau^+_{b^\ast})}v_{b^\ast}(X_{\tau^-_0\land\tau^+_{b^\ast}})\mid \cF_t}\\
&=\bE_x \sbra{e^{-q(\tau^-_0\land\tau^+_{b^\ast})}v_{b^\ast}(X_{\tau^-_0\land\tau^+_{b^\ast}})1_{\{ \tau^-_0\land\tau^+_{b^\ast}\leq t\}} 
+e^{-q(\tau^-_0\land\tau^+_{b^\ast})}v_{b^\ast}(X_{\tau^-_0\land\tau^+_{b^\ast}})1_{\{t<\tau^-_0\land\tau^+_{b^\ast}\}} \mid \cF_t} \\
&=e^{-q(\tau^-_0\land\tau^+_{b^\ast})}v_{b^\ast}(X_{\tau^-_0\land\tau^+_{b^\ast}})1_{\{ \tau^-_0\land\tau^+_{b^\ast}\leq t\}}\\
&\qquad \qquad \qquad 
+e^{-qt}\bE_x \sbra{\rbra{e^{-q(\tau^-_0\land\tau^+_{b^\ast})}v_{b^\ast}(X_{\tau^-_0\land\tau^+_{b^\ast}})}\circ \theta_t \mid \cF_t}1_{\{t<\tau^-_0\land\tau^+_{b^\ast}\}}\\
&=e^{-q(\tau^-_0\land\tau^+_{b^\ast})}v_{b^\ast}(X_{\tau^-_0\land\tau^+_{b^\ast}})1_{\{ \tau^-_0\land\tau^+_{b^\ast}\leq t\}}
+e^{-qt}v_{b^\ast}(X_t)1_{\{t<\tau^-_0\land\tau^+_{b^\ast}\}}\\
&=e^{-q(t\land\tau^-_0\land\tau^+_{b^\ast})}v_{b^\ast}(X_{t\land\tau^-_0\land\tau^+_{b^\ast}}), 
\end{align} 
where in the third equality, we used the strong Markov property and \eqref{003}. 
By Remark \ref{Rem504}, 
and since $v_{b^\ast}$ is continuous 
and $v_{{b^\ast}}^\prime$ is non-increasing 
by Lemma \ref{Lem302}, the map $x\mapsto \cL v_{b^\ast} (x)-qv_{b^\ast}(x)$ is the sum of a continuous function and a monotone function on $(0, \infty)$ (resp.\ the sum of a continuous function on $(0, \infty)$ and $\frac{1}{2}\sigma^2 v_{b^\ast}^{\prime\prime}$, which {can be taken to be} continuous a.e.\ by Assumption \ref{Ass402}) when $X$ has bounded (resp.\ unbounded) variation paths. 
Thus, by the same argument as that of the proof of \cite[Lemma 5.7]{Nob2021}, we have \eqref{001} {on $(0, b^\ast)$} when $X$ has bounded variation paths. 
In addition, we can take the density $v_{b^\ast}^{\prime\prime}$ of the function $v_{b^\ast}^\prime$ to be 
continuous and satisfying \eqref{001} on $(0, b^\ast]$. 
\par
(ii) We prove \eqref{002}. 
Using the strong Markov property, we have, for $x \in \bR$, 
\begin{align}
v_{b^\ast}(x)&=
\bE_x \sbra{e^{-q\tau^{(\alpha),-}_{b^\ast}}v_{b^\ast}(X^{(\alpha)}_{\tau^{(\alpha),-}_{b^\ast}}) ;\tau^{(\alpha),-}_{b^\ast}<\infty }
+\alpha\bE_x \sbra{\int_0^{\tau^{(\alpha),-}_{b^\ast}}e^{-qt} dt }\\
&=\xi^{(\alpha)}_{b^\ast}(x)
+\frac{\alpha}{q} , \label{62}
\end{align}
where 
\begin{align}
\xi^{(\alpha)}_{b^\ast}(x)=\bE_x \sbra{e^{-q\tau^{(\alpha),-}_{b^\ast}}v_{b^\ast}(X^{(\alpha)}_{\tau^{(\alpha),-}_{b^\ast}}) ;\tau^{(\alpha),-}_{b^\ast}<\infty }
-\frac{\alpha}{q}\bE_x \sbra{e^{-q \tau^{(\alpha),-}_{b^\ast}}}. 
\end{align}
By \eqref{62}, the function $\xi^{(\alpha)}_{b^\ast}$ has the density $\xi^{(\alpha)\prime}_{b^\ast}$ with respect to the Lebesgue measure, which is 
equal to $v_{b^\ast}^\prime$. 
In addition, 
the function $\xi^{(\alpha)\prime}_{b^\ast}$ has a density {which is locally bounded and continuous a.e.\ }on $(0, \infty )$ 
when $X$ has unbounded variation paths {by Lemma \ref{Lem302} and Assumption \ref{Ass402} and since the function $\xi^{(\alpha)}_{b^\ast}$ has the same derivative as $v_{b^\ast}$ by \eqref{62}}. 
By the same argument as in (i), 
the process $\cbra{e^{-q(t\land\tau^{(\alpha),-}_{b^\ast})}\xi^{(\alpha)}_{b^\ast}(X^{(\alpha)}_{t\land\tau^{(\alpha),-}_{b^\ast}}):t\geq 0}$ is a martingale for $\bP_x$ with $x \in{(b^\ast , \infty)}$. 
Thus, again by the same argument as in (i) (here, we may need to use Lemma \ref{Lem405} instead of Lemma \ref{Lem302} when $b^\ast=0$ {and $0$ is irregular for itself for $Y^0$}), we have
\begin{align}
\cL \xi^{(\alpha)}_{b^\ast} (x) -\alpha \xi^{(\alpha)\prime}_{b^\ast} (x)-q\xi^{(\alpha)}_{b^\ast} (x)=0, \qquad x>b^\ast \label{004}
\end{align} 
when $X$ has bounded variation paths. 
In addition, we can take the density $ \xi^{(\alpha)\prime\prime}_{b^\ast}$ of the function $ \xi^{(\alpha)\prime}_{b^\ast} $ to be 
continuous and satisfying \eqref{004} on $(b^\ast, \infty)$. 
We put $v_{b^\ast}^{\prime\prime}(x) = \xi^{(\alpha)\prime\prime}_{b^\ast}(x)$ for $x>b^\ast$, which is the density of $v_{b^\ast}^\prime$ with respect to the Lebesgue measure, when $X$ has unbounded variation paths. 
From \eqref{62} 
and \eqref{004}
, we obtain \eqref{002} in both cases of $X$ having bounded and unbounded variation paths. 
\par
{(iii) 
We prove \eqref{001} {at $b^\ast$} when $X$ has bounded variation paths and $b^\ast>0$. 
We assume that $\delta>0$. Then, $0$ is regular for $(0, \infty)$ for $X$, and thus 
by the strong Markov property and \eqref{003}, and since $v_{b^\ast}^\prime(=\un{\nu})$ is non-increasing, we have
\begin{align}
\lim_{x\uparrow b^\ast}v^\prime_{b^\ast}(x)
=&\lim_{x\uparrow b^\ast}\bE_x \sbra{e^{-q\tau^+_{b^\ast}}v_{b^\ast}^\prime (X_{\tau^+_{b^\ast}});\tau^+_{b^\ast}<T^{p^\ast}_0}
+\lim_{x\uparrow b^\ast}\beta\bE_x \sbra{e^{-qT^{p^\ast}_0} ;T^{p^\ast}_0<\tau^+_{b^\ast}}\\
\leq&  \lim_{x\uparrow b^\ast}\bE_x \sbra{e^{-q\tau^+_{b^\ast}};\tau^+_{b^\ast}<T^{p^\ast}_0}v_{b^\ast}^\prime(b^\ast)
+\lim_{x\uparrow b^\ast}\beta\bE_x \sbra{e^{-qT^{p^\ast}_0} ;T^{p^\ast}_0<\tau^+_{b^\ast}}\\
=&v_{b^\ast}^\prime(b^\ast), \label{127}
\end{align}
where
\begin{align}
T^{p^\ast}_0 {:=} \tau^{ -}_0 1_{\{V_{{p^\ast}}=1\}}+ \inf \{t>0 : X_t \leq 0\} 1_{\{V_{{p^\ast}} =0\}}. 
\end{align}
By \eqref{127} and since $v_{b^\ast}^\prime$ is non-increasing, we have $\lim_{x\uparrow b^\ast}v^\prime_{b^\ast}(x)=v_{b^\ast}^\prime(b^\ast)$, and thus the map $x\mapsto \cL v_{b^\ast} (x)-qv_{b^\ast}(x)$ is left-continuous at $b^\ast$. 
From the above argument and (i), we have 
\begin{align}
\cL v_{b^\ast} (b^\ast)-qv_{b^\ast}(b^\ast)=\lim_{x\uparrow b^\ast}\rbra{\cL v_{b^\ast} (x)-qv_{b^\ast}(x)}=0. \label{129}
\end{align} 
We assume that $\delta\leq 0$ which implies $\delta-\alpha <0$. 
Then, $0$ is regular for $(-\infty, 0)$ for $X^{(\alpha)}$, and thus 
by the strong Markov property and \eqref{62}, and since $v_{b^\ast}^\prime$ is non-increasing, we have
\begin{align}
\lim_{x\downarrow b^\ast}v^\prime_{b^\ast}(x)
=&\lim_{x\downarrow b^\ast}\bE_x \sbra{e^{-q\tau^{(\alpha),-}_{b^\ast}}v_{b^\ast}^\prime(X^{(\alpha)}_{\tau^{(\alpha),-}_{b^\ast} }) ;\tau^{(\alpha),-}_{b^\ast}<\infty}
\\
\geq &\lim_{x\downarrow b^\ast}\bE_x \sbra{e^{-q\tau^{(\alpha),-}_{b^\ast}}}v_{b^\ast}^\prime(b^\ast)
=v_{b^\ast}^\prime(b^\ast).\label{128}
\end{align}
By \eqref{128} and since $v_{b^\ast}^\prime$ is non-increasing, we have $\lim_{x\downarrow b^\ast}v^\prime_{b^\ast}(x)=v_{b^\ast}^\prime(b^\ast)$, and thus the map $x\mapsto \cL v_{b^\ast} (x) -\alpha v_{b^\ast}^\prime (x)-qv_{b^\ast} (x)+\alpha$ is right-continuous at $b^\ast$. 
From \eqref{a05}, the above argument, and (ii), we have
\begin{align}
\cL v_{b^\ast} (b^\ast)-qv_{b^\ast}(b^\ast)
=&\cL v_{b^\ast} (b^\ast) -\alpha v_{b^\ast}^\prime (b^\ast)-qv_{b^\ast} (b^\ast)+\alpha\\
=&\lim_{x\downarrow b^\ast}\rbra{\cL v_{b^\ast} (x) -\alpha v_{b^\ast}^\prime (x)-qv_{b^\ast} (x)+\alpha}=0.\label{130}
\end{align}
By \eqref{129} and \eqref{130}, we have \eqref{001} {at $b^\ast$} when $X$ has bounded variation paths and $b^\ast>0$. 
}
\par
From (i), (ii), and (iii), the proof is complete.  
\end{proof}

\begin{proof}[Proof of Theorem \ref{Thm502}]
We have already confirmed that $v_{b^\ast}$ belongs to $C^{(1)}_{\text{line}}$ (resp.\ $C^{(2)}_{\text{line}}$ under Assumption \ref{Ass402}) when $X$ has bounded (resp.\ unbounded) variation paths. 
Thus, it is sufficient to prove that $v_{b^\ast}$ satisfies \eqref{veri1}, \eqref{veri2}, and \eqref{veri3} by Lemma \ref{Lem401}. 
\par
From Lemmas \ref{Lem302} and \ref{Lem405}, 
$v_{b^\ast}$ satisfies \eqref{veri2}. 
In addition, since $v_{b^\ast}$ is non-decreasing and $v_{b^\ast}(0)\in\bR$ by Remark \ref{Rem203} and \eqref{90} with $\pi=\pi^{b^\ast}$ and $x=0$, $v_{b^\ast}$ satisfies \eqref{veri3}. 
For $x \in (0, b^\ast]$, we have $- v_{b^\ast}^\prime(x) +  1\leq 0$ by Lemma \ref{Lem302}{, the definition of $b^\ast$,} and \eqref{a05}, and thus we have 
\begin{align}
\sup_{r\in[0, \alpha]} \rbra{\cL v_{b^\ast} (x)-qv_{b^\ast}(x)-rv_{b^\ast}^\prime(x)+r}=\cL v_{b^\ast} (x)-qv_{b^\ast}(x)=0, 
\end{align}
where in the last equality we used \eqref{001}. 
For $x > b^\ast$, we have $- v_{b^\ast}^\prime(x) +  1\geq 0$ by {Lemmas \ref{Lem302}, \ref{Lem405}, and the definition of $b^\ast$}, 
and thus we have 
\begin{align}
\sup_{r\in[0, \alpha]} \rbra{\cL v_{b^\ast} (x)-qv_{b^\ast}(x)-rv_{b^\ast}^\prime(x)+r}=\cL v_{b^\ast} (x)-qv_{b^\ast}(x)
-\alpha v_{b^\ast}^\prime(x)+\alpha=0, 
\end{align}
where in the last equality we used \eqref{002}. 
Therefore, the function $v_{b^\ast}$ satisfies \eqref{veri1}. 
The proof is complete. 
\end{proof}

\section{Numerical results}\label{numericalresults}
In this section, we present numerical results for $\un{\nu}$ and $v_b$ via Monte Carlo simulation as in \cite
{ NobYam2022}, and we confirm the correctness of our main theorem. 
Since the purpose of this study is not to think of better simulation methods, we use the classical Euler scheme. 
\par
For the simulation, we use the L\'evy process $X$ that has the form 
\begin{align}
X_t =X_0 + 0.6 t + a B_t + \sum_{k=1}^{N^+_t} Z^+_k -\sum_{k=1}^{N^-_t} Z^-_k, \quad t\geq 0 , 
\end{align}
where $\{B_t : t\geq 0\}$ is a standard Brownian motion, $\{N^+_t:t\geq 0\}$ and $\{N^-_t:t\geq 0\}$ are independent Poisson processes with arrival rate $1$, $\{Z^+_k: k\in\bN\}$ is an independent and identically distributed (i.i.d.) sequence of continuous uniform random variables on $[0, 1]$, and $\{Z^-_k: k\in\bN\}$ is an i.i.d.\ sequence of Weibull random variables with shape parameter $2$ and scale parameter $1$. 
In this section, we give simulations in two cases. 
In Cases $1$ and $2$, we assume that $a=0$ and $a=1$, respectively. 
Note that the L\'evy process $X$ in Case $2$ has unbounded variation paths and we did not confirm that $X$ satisfies the conditions in Assumption \ref{Ass402}. 
For the other parameters, we set $q=0.05$, $\beta=1.5$, and $\alpha =0.5$. 
\par
In our simulation, we truncate the time horizon at $T=100$ and discretize $[0, T]$ as $K= 10,000$ equally spaced points with distance $\Delta t
= \frac{T}{K}$ as in \cite
{ NobYam2022}. 
To approximate $\un{\nu}$ and $v_b$, we prepare the set of $N=100,000$ sample paths that approximate $X$ starting from $0$, and we represent these with 
\begin{align}
\hat{X}^{(n)} = \{\hat{X}^{(n)}_{k \Delta t}: k=0, 1, \cdots , K-1\}, \quad n=1,2, \cdots , N. 
\end{align}
Then the approximated sample paths of $L^{\pi^b}$ and $R^{\pi^b}$ when $X$ starts from $x\in\bR$ are 
\begin{align}
\begin{aligned}
\hat{L}^{(n)}(x, b)  &= \{\hat{L}^{(n)}_{k \Delta t} (x, b): k=0, 1, \cdots , K-1\}, \\
\hat{R}^{(n)}(x, b) &= \{\hat{R}^{(n)}_{k \Delta t}(x, b): k=0, 1, \cdots , K-1\}, 
\end{aligned}
\quad n=1,2, \cdots , N,
\end{align}
defined as follows, inductively: for $k=1,2,\cdots, K-1$, 
if $x+\hat{X}^{(n)}_{k \Delta t}-\hat{L}^{(n)}_{(k-1) \Delta t}(x, b)+\hat{R}^{(n)}_{(k-1) \Delta t}(x, b)<0$, then 
\begin{align}
\hat{L}^{(n)}_{k \Delta t}(x, b)=\hat{L}^{(n)}_{(k-1) \Delta t}(x, b), \quad \hat{R}^{(n)}_{k \Delta t}(x, b)=
-(x+\hat{X}^{(n)}_{k \Delta t}-\hat{L}^{(n)}_{(k-1) \Delta t}(x, b)), 
\end{align}
if $x+\hat{X}^{(n)}_{k \Delta t}-\hat{L}^{(n)}_{(k-1) \Delta t}(x, b)+\hat{R}^{(n)}_{(k-1)\Delta t}(x, b)> b$, then 
\begin{align}
\hat{L}^{(n)}_{k \Delta t}(x, b)=\hat{L}^{(n)}_{(k-1) \Delta t}(x, b)+ \alpha \Delta t, \quad \hat{R}^{(n)}_{k \Delta t}(x, b)=
\hat{R}^{(n)}_{(k-1) \Delta t}(x, b), 
\end{align}
else
\begin{align}
\hat{L}^{(n)}(x, b)_{k \Delta t}=\hat{L}^{(n)}_{(k-1) \Delta t}(x, b), \quad \hat{R}^{(n)}_{k \Delta t}(x, b)=
\hat{R}^{(n)}_{(k-1) \Delta t}(x, b), 
\end{align}
where $\hat{L}^{(n)}_0(x, b)=0$ and $\hat{R}^{(n)}_0(x, b)= 0\lor(-\hat{X}^{(n)}_0) $. 
Thus, we approximate ${\nu}(x)$ by 
\begin{align}
\hat{{\nu}}(b)&= 
\begin{cases}
\frac{1}{N}\sum_{n=1}^{N} \exp\rbra{ -q \Delta t \rbra{\hat{\kappa}^{(n)}_0(b,b) \land 10 K  }}, \quad &x\geq 0,\\
1, \quad &x<0,
\end{cases}
\end{align}
where $\hat{\kappa}^{(n)}_0(x, b) =\min\{ k \in \{0, 1, \cdots , K-1\}: \hat{R}^{(n)}_{k\Delta t}(x, b) >0  \}$ for $b\geq 0 $ and $x\in\bR$ with $\min \emptyset= \infty$. 
In addition, we approximate $v_b (0)$ by
\begin{align}
\check{v}_b(0)&=\frac{1}{N}\sum_{n=1}^{N} \sum_{k=1}^{K-1} 
e^{- qk\Delta t}\bigg{(}\rbra{\hat{L}^{(n)}_{k \Delta t}(0, b) -\hat{L}^{(n)}_{(k-1) \Delta t}(0, b)}\\
&\qquad \qquad \qquad \qquad \qquad \qquad -\beta \rbra{\hat{R}^{(n)}_{k \Delta t}(0, b) -\hat{R}^{(n)}_{(k-1) \Delta t}(0, b)}\bigg{)},
\quad b\geq 0.
\end{align} 
By the strong Markov property at $\kappa^{b, -}_0$, the function $v_b (x)$ with $b\geq 0$ is approximated by 
\begin{align}
\hat{v}_b(x)=
\begin{cases}
-\frac{1}{N}\sum_{n=1}^{N}\beta \hat{R}^{(n)}_{k \Delta t}(x, b) +\check{v}_b (0), \quad& x<0,\\
\begin{aligned}
\frac{1}{N}\sum_{n=1}^{N}\bigg{(} \sum_{k=1}^{(K-1)\land \breve{\kappa}^{(n)}_0(x, b)} 
e^{- qk\Delta t}\bigg{(}\rbra{\hat{L}^{(n)}_{k \Delta t}(x, b) -\hat{L}^{(n)}_{(k-1) \Delta t}(x, b)}\\
~-\beta \rbra{\hat{R}^{(n)}_{k \Delta t}(x, b) -\hat{R}^{(n)}_{(k-1) \Delta t}(x, b)}\bigg{)}
+e^{-q\Delta t \breve{\kappa}^{(n)}_0(x, b)}\check{v}_b (0)\bigg{)}, 
\end{aligned}
\quad  & x\geq 0, 
\end{cases}
\end{align} 
where $\breve{\kappa}^{(n)}_0(x, b)=\min\{k\in\{1,2,\cdots, K-1\}: x+\hat{X}^{(n)}_{k\Delta t} -\hat{L}^{(n)}_{(k-1)\Delta t}\leq0 \}$. 
In the first step, we compute $\hat{{\nu}}(b)$ for $b=-1, -0.99, 0.98, \dots, 3.48, 3.49$ in Case $1$ and for $b=-1, -0.99, 0.98, \dots, 3.98, 3.99$ in Case $2$, and make figures. 
Then, the function ${\nu}$ in Cases $1$ and $2$ is approximated by Figure \ref{Fig01}. 
\begin{figure}[htbp]
  \begin{minipage}{0.5\hsize}
    \begin{center}
      \includegraphics[width=70mm]{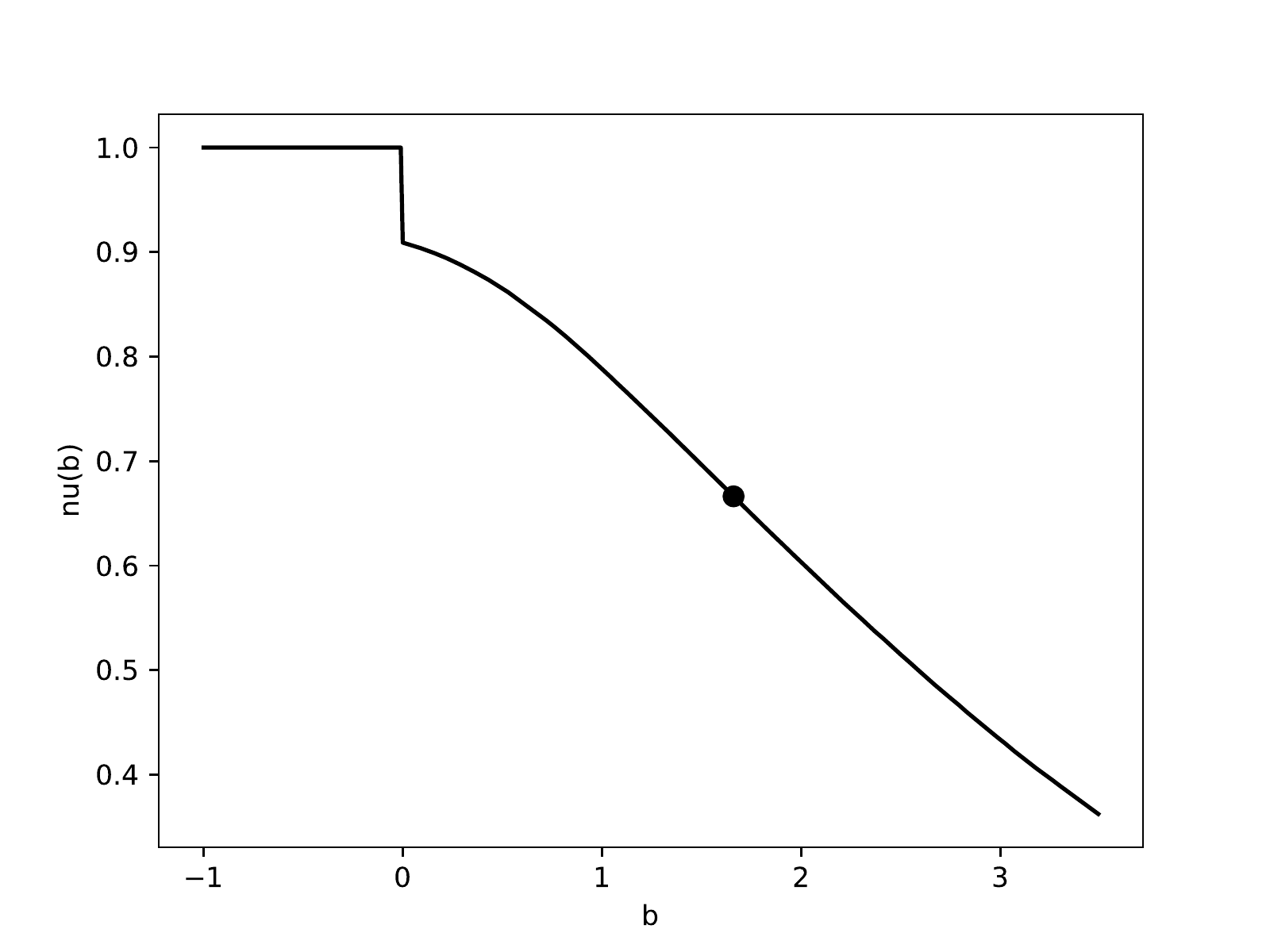}\\
      {\bf{Case 1}}
    \end{center}
    \label{fig:img1}
  \end{minipage}
  \begin{minipage}{0.5\hsize}
    \begin{center}
      \includegraphics[width=70mm]{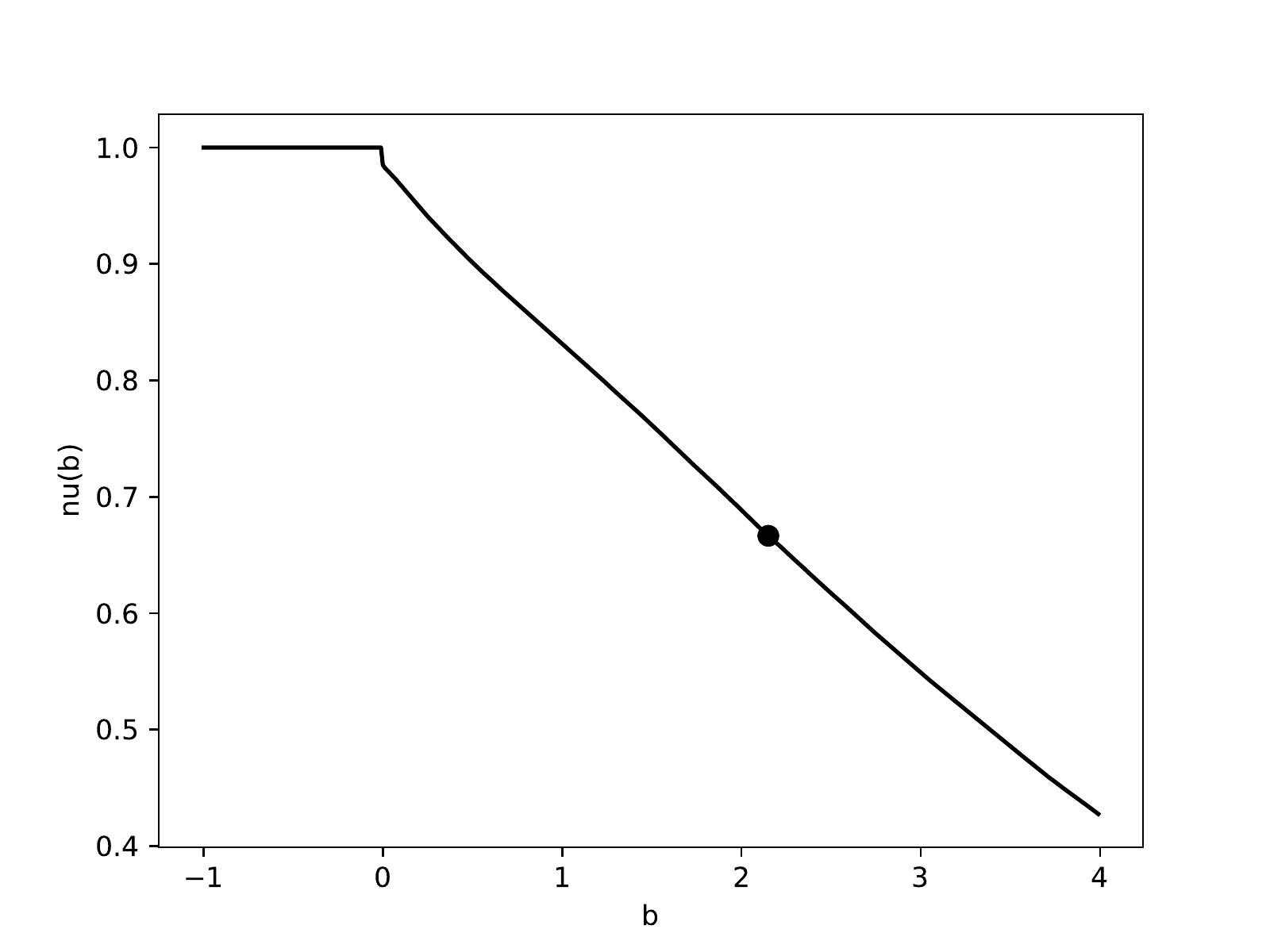}\\
      {\bf{Case 2}}
    \end{center}
    \label{fig:img2} 
  \end{minipage}
  \caption{{Plots of the approximations of $b \mapsto {\nu}(b)$ and $(b^\ast, {\nu}(b^\ast))$.}}\label{Fig01}
\end{figure}
From this computation, $b^\ast$ is approximated by $1.66$ and $2.15$ in Cases $1$ and $2$, respectively. 
In the second step, we compute 
$v_b(x)$ for some $b$ and for $x=-1, -0.99, 0.98, \dots, 3.48, 3.49$ in Case $1$ and $x=-1, -0.99, 0.98, \dots, 3.98, 3.99$ in Case $2$. 
Here, we compute for $b=0.56, 1.11, 1.66 (b^\ast), 2.22, 2.77$ in Case $1$ and for $b=0.72,1.44, 2.15 (b^\ast), 2.87,3.59$ in Case $2$. 
Then the function $v_b$ with $b=\frac{1}{3}b^\ast$, $\frac{2}{3}b^\ast$, $b^\ast$, $\frac{4}{3}b^\ast$, and $\frac{5}{3}b^\ast$ in Cases $1$ and $2$ is approximated by Figure \ref{Fig02}. In these figures, the function $v_{b^\ast}$ is represented by the solid line and the others are represented by the dotted line. 
\begin{figure}[htbp]
  \begin{minipage}{0.5\hsize}
    \begin{center}
      \includegraphics[width=70mm]{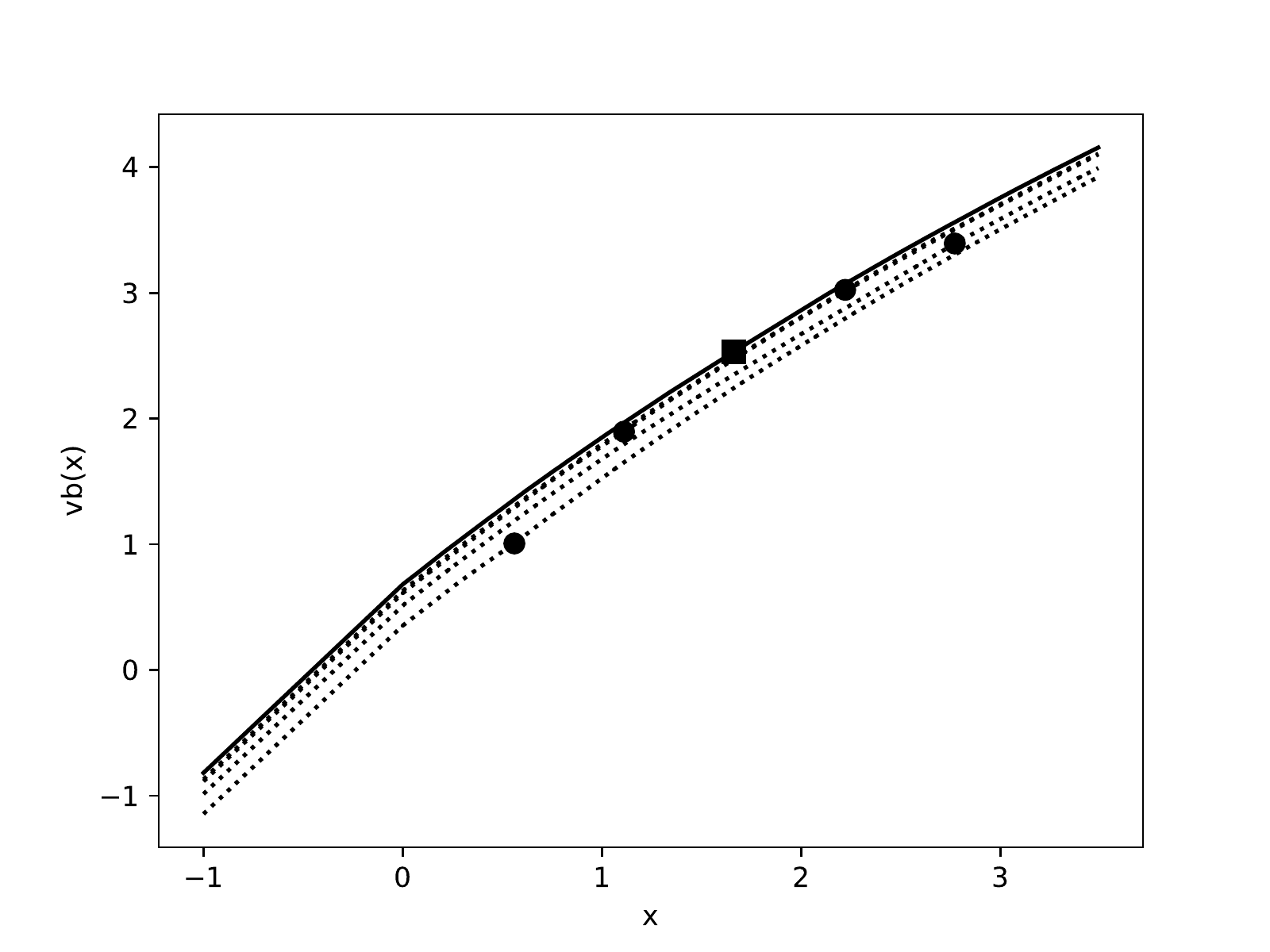}\\
      {\bf{Case 1}}
    \end{center}
    \label{fig:img1}
  \end{minipage}
  \begin{minipage}{0.5\hsize}
    \begin{center}
      \includegraphics[width=70mm]{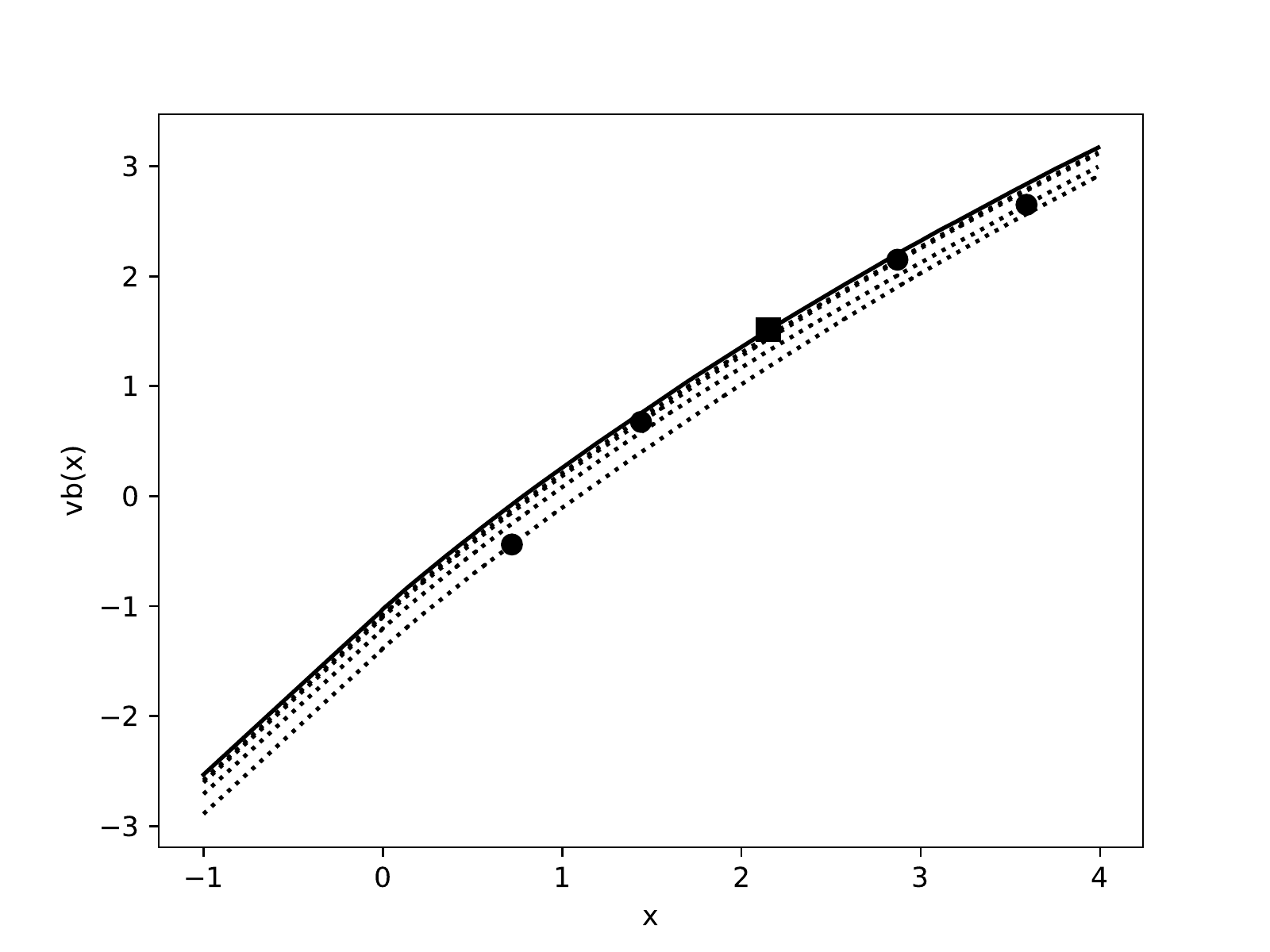}\\
      {\bf{Case 2}}
    \end{center}
    \label{fig:img2} 
  \end{minipage}
  \caption{{Plots of the approximation of $x \mapsto v_{b} (x)$ and $(b, v_{b}(b))$ for $b=\frac{1}{3}b^\ast$, $\frac{2}{3}b^\ast$, $b^\ast$, $\frac{4}{3}b^\ast$, and $\frac{5}{3}b^\ast$.}}\label{Fig02}
\end{figure}
From Figure \ref{Fig02}, we can see by visual inspection that $b^\ast$ is the optimal barrier. 
In particular, we have not confirmed that $X$ in Case $2$ satisfies the assumptions given in this paper, but we can expect that the main result is true for it.

{
\section{Convergence result}\label{Sec007}
In the previous sections we have shown the optimality of refraction--reflection strategies. On the other hand, in \cite{Nob2021}, I have proved the optimality of double barrier strategies under different conditions. In this section, we show that the expected NPV of the dividend payments and capital injections when taking the optimal refraction--reflection strategy converges to that when taking the optimal double barrier strategy under the condition of \cite{Nob2021} by taking the limit of $\alpha$ to $\infty$. 
}
\par
{
In this section, we assume that Assumption \ref{Ass202a} and Assumption \ref{Ass402} hold for all $\alpha>0$. 
In addition, we give \cite[Assumption 2.1]{Nob2021}. 
\begin{Rem}
\cite[Assumption 2.1 (2.4)]{Nob2021} is stronger than Assumption \ref{Ass201}. 
If we only give Assumption \ref{Ass201}, we can predict that the expected NPV of an optimal dividend payments and capital injections may go to infinity by taking the limit of $\alpha$ to $\infty$ since 
the expected NPV of the dividend payments may become infinite when we take double barrier strategies (see, \cite[Remark 3.4]{Nob2021}). 
However, since I only considered cases with \cite[Assumption 2.1 (2.4)]{Nob2021} in \cite{Nob2021}, we give it in this section for simplicity.
\end{Rem}
We write $Y^{\alpha, b}:=\{Y^{\alpha, b}_t: t\geq 0\}$ for refracted L\'evy process at $b\geq 0$ with $\alpha>0$. 
Similarly, we write $Y^{\infty, b}:=\{Y^{\infty, b}_t: t\geq 0\}$ for reflected L\'evy process at $b\geq 0$, i.e. $Y^{\infty, b}_t= X_t - \sup_{s\in[ 0, t] } (X_s -b )\lor 0$ for $t\geq 0$. 
In addition, we denote, for $\alpha>0$, $b\in\bR$ and $x\in\bR$, 
\begin{align}
\kappa^{\alpha, b}_x:=\inf \{t> 0 : Y^{\alpha, b}_t < x\},&\qquad 
\kappa^{\infty, b}_x:=\inf \{t> 0 : Y^{\infty, b}_t < x\}, \\
\nu_{b}^\alpha (x) := \bE_x \sbra{e^{-q\kappa^{\alpha, b }_0 }} ,&
\qquad 
\nu_{b}^\infty (x) := \bE_x \sbra{e^{-q\kappa^{\infty, b }_0 }} , \\
b^{\alpha, \ast}:= \inf \cbra{ b>0 :  \beta\nu_{b}^\alpha (b) \leq 1}, &
\qquad 
b^{\infty, \ast}:= \inf \cbra{ b>0 :  \beta\nu_{b}^\infty (b) \leq 1}.
\end{align}
For fixed $\alpha>0$, an optimal strategy of the problem in Section \ref{SubSec301} is the refraction--reflection strategy $\pi^{\alpha, b^{\alpha, \ast}}$ at $b^{\alpha, \ast}$ by Remark \ref{Rem403} and Theorem \ref{Thm502}.  
Similarly, an optimal strategy of the problem in \cite[Section 2.2]{Nob2021} is the doubly barrier strategy $\pi^{\infty, b^{\infty, \ast}}$ at $b^{\infty, \ast}$ by \cite[Theorem 5.1]{Nob2021}. 
We write the expected NPVs of the dividend payments and capital injections as $v^\alpha_b$ and $v^\infty_b$ when using the strategy $\pi^{\alpha, b}$ and $\pi^{\infty, b}$ with $b\in[0, \infty)$, respectively. 
The purpose of this chapter is to prove that $ b^{\alpha, \ast}$ and $v^\alpha_{b^{\alpha, \ast}}(x)$ converge $ b^{\infty, \ast}$ and $v^\infty_{b^{\infty, \ast}} (x)$ for $x \in\bR$ as $\alpha\uparrow \infty$, respectively.}
{
\begin{Lem}\label{Lem702}
For $x, b\in\bR$, we have 
\begin{align}
Y^{\alpha, b}_t \downarrow Y^{\infty, b}_t , \quad \text{as} \quad \alpha \uparrow \infty ,\label{Rev006}
\end{align}
for all $t\geq 0$ such that $X$ does not have negative jumps at $t$, $\bP_x$-a.s.. 
\end{Lem}
From Lemma \ref{Lem702}, the convergence \eqref{Rev006} is true for a.e. $t\geq 0$, $\bP_x$-a.s.. 
We give the proof of Lemma \ref{Lem702} in Appendix \ref{Ape00H1}. 
}
\par
{
Using Lemma \ref{Lem702}, we can get the following corollary. 
The proof is easy, so we omit it. 
\begin{Cor}
For $x ,b \in\bR$, we have $\nu_{b}^\alpha (x) \uparrow \nu_{b}^\infty (x) $ as $\alpha\uparrow\infty$. 
In addition, we have $b^{\alpha, \ast}\uparrow b^{\infty, \ast}$ as $\alpha \uparrow\infty$. 
\end{Cor}
For fixed $\alpha$, we write $Z^{\alpha, b}=\{Z^{\alpha, b}_t : t\geq 0\}$ for refracted--reflected L\'evy process at $b\geq 0$. 
In addition, we write $L^{\alpha, b}=\{L^{\alpha, b}_t : t\geq 0\}$ and $R^{\alpha, b}=\{R^{\alpha, b}_t : t\geq 0\}$ for the cumulate amount of dividend payments and capital injections when we apply refraction--reflection strategy $\pi^{\alpha, b}$ at $b\geq 0$. 
We represent the symbols about doubly barrier strategies by substituting $\infty$ for $\alpha$ in the symbols above.}
\par
{
The following lemma is necessary to prove Theorem \ref{Thm705}.
\begin{Lem}\label{Lem705}
For $b> 0$, we have 
\begin{align}
Z^{\alpha, b}_t \downarrow Z^{\infty, b}_t , \quad L^{\alpha, b}_t \uparrow L^{\infty, b}_t, \quad R^{\alpha, b}_t \uparrow R^{\infty, b}_t, \quad \text{as} \quad \alpha \uparrow \infty ,\label{Rev019}
\end{align}
for all $t\geq 0$ such that $X$ does not have negative jumps at $t$, $\bP_x$-a.s.. 
In addition, the above is also true when $b=0$ and $X$ has bounded variation paths. 
\end{Lem}}
{
We give the proof of Lemma \ref{Lem705} in Appendix \ref{Ape00H2}. 
}
{
\begin{Thm}\label{Thm705}
For $x\in\bR$, we have 
\begin{align}
v^{\alpha}(x) \uparrow v^{\infty} (x),\quad \text{as}\quad \alpha \uparrow \infty. \label{Rev016}
\end{align}
This convergence is uniformly in $x$ on any compact set. 
\end{Thm}
}
\begin{proof}
{For fixed $\alpha>0$, let $\cA^\alpha$ be the set of all strategies that satisfies the conditions in Section \ref{SubSec301}. 
Similarly, $\cA^\infty$ be the set of all strategies that satisfies the conditions in \cite[Section 2.2]{Nob2021}. Then, by the definitions of $\cA^\alpha$ and $\cA^\infty$, we have, for $\alpha_1, \alpha_2\in(0, \infty]$ with $\alpha_1<\alpha_2$, $\cA^{\alpha_1} \subset \cA^{\alpha_2}$, 
and thus, we have
\begin{align}
v^{\alpha_1}_{b^{\alpha_1, \ast}}(x)\leq v^{\alpha_2}_{b^{\alpha_2, \ast}}(x), 
\qquad x\in\bR. \label{Rev012}
\end{align}
Since $b^{\alpha, \ast}$ is an optimal barrier, we have $v^{\alpha}_{b^{\infty, \ast}}(x) \leq v^{\alpha}_{b^{\alpha, \ast}}(x)$ for $\alpha>0$ and $x\in\bR$, thus by \eqref{Rev012}, we have
\begin{align}
\absol{v^{\infty}_{b^{\infty, \ast}}(x)-v^{\alpha}_{b^{\alpha, \ast}}(x)}
\leq \absol{v^{\infty}_{b^{\infty, \ast}}(x)-v^{\alpha}_{b^{\infty, \ast}}(x)}
,\qquad \alpha>0, \ x\in\bR. \label{Rev013}
\end{align} 
For $\alpha>0$ and $x \in\bR$, we have
\begin{align}
&\lim_{\alpha\uparrow\infty}\absol{\bE_x\sbra{\int_{[0, \infty)} e^{-qt} dL^{\infty, b^{\infty, \ast}}_t} -\bE_x\sbra{\int_{[0, \infty)} e^{-qt} dL^{\alpha, b^{\infty, \ast}}_t}}\\
&\qquad \qquad = q\lim_{\alpha\uparrow\infty}\absol{\bE_x\sbra{\int_0^\infty e^{-qt} L^{\infty, b^{\infty, \ast}}_t dt} -\bE_x\sbra{\int_0^\infty e^{-qt} L^{\alpha, b^{\infty, \ast}}_t dt}}\\
&\qquad\qquad\leq q\lim_{\alpha\uparrow\infty}\bE_x \sbra{\int_0^\infty e^{-qt}\absol{L^{\infty, b^{\infty, \ast}}_t -L^{\alpha, b^{\infty, \ast}}_t }  dt}=0, \label{Rev014}
\end{align}
where in the first equality, we used integration by parts and in the last equality, we used Lemma \ref{Lem705} and the dominated convergence theorem with  
$\absol{L^{\infty, b^{\infty, \ast}}_t -L^{\alpha, b^{\infty, \ast}}_t } \leq 2L^{\infty, b^{\infty, \ast}}_t$ 
and
\begin{align} 
 q \bE_x \sbra{\int_0^\infty e^{-qt} 2 L^{\infty, b^{\infty, \ast}}_t  dt}
 =2\bE_x \sbra{ \int_{[0,\infty)} e^{-qt} dL^{\infty , b^{\infty , \ast}}_t}<\infty. 
\end{align}
By the same argument as above, we have 
\begin{align}
\lim_{\alpha\uparrow\infty}\absol{\bE_x\sbra{\int_{[0, \infty)} e^{-qt} dR^{\infty, b^{\infty, \ast}}_t} -\bE_x\sbra{\int_{[0, \infty)} e^{-qt} dR^{\alpha, b^{\infty, \ast}}_t}}=0. \label{Rev015}
\end{align}
By \eqref{Rev012}, \eqref{Rev013}, \eqref{Rev014} and \eqref{Rev015}, we obtain \eqref{Rev016}. 
Uniform convergence on any compact set follows from Dini's theorem. 
The proof is complete. 
}
\end{proof}

\section*{Acknowledgments}
I express my deepest gratitude to Prof.\ Kazutoshi Yamazaki and Prof.\ Kouji Yano for their comments on the structure of this paper. 
This work was supported by JSPS KAKENHI Grant Number 21K13807 and JSPS Open Partnership Joint Research 
Projects Grant Number JPJSBP120209921.

\appendix 
\section{Lemma for infimum
} \label{Sec0A}
In this appendix, we give a lemma for the infimum, which is used to prove Theorem \ref{Thm601} and in Appendix \ref{Sec00B}. 
\begin{Lem}\label{LemB01}
Let $t\mapsto x_t$ be the c\`adl\`ag function from $[0,\infty)$ to $\bR$ that has bounded variation and satisfies 
\begin{align}
x_t=x_0+ \delta t+\sum_{s\in(0, t]}(x_s-x_{s-}),\quad t\geq 0 \label{97}
\end{align}
for some $\delta\leq 0$. 
We define 
\begin{align}
I^x_t {:=} \inf_{s\in[0, t]} \rbra{x_s\land 0}, \quad\check{x}_t {:=} x_t- I^x_t  ,    \quad t\geq 0. \label{43}
\end{align}
Then we have, for $t\geq 0$, 
\begin{align}
I^x_t  = \delta \int_0^t 1_{\{ 0\}} (\check{x}_s) ds+ (x_0\land 0)+
\sum_{s\in(0, t]} \rbra{ \check{x}_{s-}+(x_s-x_{s-} )  }\land 0.
\label{47}
\end{align}
\end{Lem}
\begin{Rem}
If the L\'evy process $X$ has bounded variation paths and satisfies $\delta \in [0, \alpha]$, then the paths of $X^{(\alpha)}$ satisfy the conditions that are imposed for $x_t$ in Lemma \ref{LemB01}, almost surely. 
\end{Rem}
\begin{proof}[Proof of Lemma \ref{LemB01}]
(i) In the first step, we assume that the function $t\mapsto x_t$ is non-increasing. 
We define 
\begin{align}
D_0{:=}\inf\{t>0: x_t\leq0\}.
\end{align} 
Since $t\mapsto x_t$ is non-increasing and by the definition of $D_0$, we have 
\begin{align}
I^x_t=
\begin{cases}
0, \quad & t< D_0, \\
x_t ,\quad &t\geq D_0, 
\end{cases}
\label{95}
\end{align}
and thus 
\begin{align}
\check{x}_t=
\begin{cases}
x_t, \quad & t< D_0, \\
0 ,\quad &t\geq D_0.  
\end{cases}
\label{96}
\end{align}
By \eqref{96}  and the definition of $D_0$, we have 
\begin{align}
\begin{aligned}
\check{x}_t=x_t>0,&\quad \check{x}_{t-} + (x_t - x_{t-}  )=x_t >0, \quad t\in(0, D_0), \\
\check{x}_{D_0-} +& (x_{D_0} - x_{D_0-}  )=x_{D_0}\leq0 ,  \text{ if }D_0>0\\
&x_0\land 0=
\begin{cases}
0, \quad&\text{ if } D_0 >0,\\
x_0, \quad&\text{ if } D_0 =0, 
\end{cases} 
\end{aligned}
\label{102}
\end{align}
and thus we have 
\begin{align}
(\text{The right-hand side of \eqref{47}})=
\begin{cases}
0, \quad &t<D_0, \\
x_{D_0}, \quad &t=D_0.
\end{cases}
\label{103}
\end{align}
By \eqref{95} and \eqref{103}, \eqref{47} holds for $t\leq D_0$. 
By \eqref{95}, \eqref{97}, and \eqref{96}, we have, for $t> D_0$, 
\begin{align}
I^x_t-I^x_{D_0}=x_t-x_{D_0}
=\delta (t-D_0)+\sum_{s\in(D_0, t]}(x_s -x_{s-})\\
=\delta \int_{D_0}^t 1_{\{ 0\}} (\check{x}_s) ds+ 
\sum_{s\in (D_0, t]} \rbra{ \check{x}_{s-}+(x_s-x_{s-} )  }\land 0. \label{98}
\end{align}
By \eqref{98} and since \eqref{47} holds for $t=D_0$, we have \eqref{47} for $t>D_0$. 
Thus, we obtain \eqref{47} on $t\geq 0$. 
\par
(ii) In the second step, we assume that the map $t\mapsto x_t$ has only a finite number of positive jumps on each compact interval in $[0, \infty)$. 
Let $T^{[k]}$ be the $k$-th positive jump time of $t\mapsto x_t$. 
We consider the value $\inf_{s\in[0, t]} \rbra{ x_s\land 0}$ 
for each interval $[T^{[k]}, T^{[k+1]} )$ with $k\in\bN\cup\{0\}$, where $T^{[0]}=0$. 
Since $t\mapsto x_t$ is non-increasing on $[0, T^{[1]})$ and by (i), we have \eqref{47} for $t\in [0, T^{[1]})$. 
We fix $k\in\bN$. 
For $t\geq T^{[k]}$, we have 
\begin{align}
I^x_t=
\inf_{u\in[T^{[k]}, t]} \rbra{ x_u\land I^x_{T^{[k]}-}}
=I^{x^{[k]}}_{t-T^{[k]}}
+I^x_{T^{[k]}-},
\label{99}
\end{align}
where 
\begin{align}
{x}^{[k]}_t {:=} x_{t+ T^{[k]}}- I^x_{T^{[k]}-},\quad I^{x^{[k]}}_t{:=}\inf_{s\in[0, t]}\rbra{{x}^{[k]}_s \land 0  }, \quad t\geq 0. \label{124}
\end{align}
Since the function $t\mapsto x^{[k]}_t$ is non-increasing on $[0, T^{[k+1]}-T^{[k]})$, we can apply (i) for $t\mapsto x^{[k]}_t$ and have, 
for $t\in[0, T^{[k+1]}-T^{[k]})$, 
\begin{align}
I^{{x}^{[k]}}_t
&=\delta \int_{0}^t 1_{\{ 0\}} (\check{x}^{[k]}_s) ds+ {x}^{[k]}_0 \land 0+
\sum_{s\in(0, t]} \rbra{ \check{x}^{[k]}_{s-}+(x^{[k]}_s-x^{[k]}_{s-} )  }\land 0, \label{117}
\end{align} 
where 
\begin{align}
\check{x}^{[k]}_t{:=}x^{[k]}_t -I^{x^{[k]}}_t, 
\quad t\geq 0.
\end{align}
By \eqref{117} {and} \eqref{124}, and since we have
\begin{align}
\check{x}^{[k]}_t=x_{t+ T^{[k]}}- I^x_{T^{[k]}-}-\inf_{s\in[0,t]}\rbra{\rbra{x_{s+ T^{[k]}}- I^x_{T^{[k]}-}} \land 0}
=\check{x}_{t+T^{[k]}}, \quad t\geq 0, 
\end{align}
we have, for $t\in[0, T^{[k+1]}-T^{[k]})$, 
\begin{align}
I^{{x}^{[k]}}_t
&=\delta \int_{T^{[k]}}^{t+T^{[k]}} 1_{\{ 0\}} (\check{x}_{s}) ds+ {x}^{[k]}_{0}\land 0+
\sum_{s\in(T^{[k]}, t+T^{[k]}]} \rbra{ \check{x}_{s-}+(x_s-x_{s-} )  }\land 0
.\label{100}
\end{align}
Since $t\mapsto x_t$ has a positive jump at $T^{[k]}$, we have 
\begin{align}
{x}^{[k]}_{0}\land 0=0,\quad
\rbra{\check{x}_{T^{[k]}-}+(x_{T^{[k]}}-x_{T^{[k]}-} )}\land 0=0  . \label{101}
\end{align}
By \eqref{99}, \eqref{100}, and \eqref{101}, we have 
\begin{align}
\begin{aligned}
I^x_t
=I^x_{T^{[k]}-}&
+\delta \int_{T^{[k]}}^t 1_{\{0\}}(\check{x}_s)ds \\
&+\sum_{s\in[T^{[k]}, t]} \rbra{ \check{x}_{s-}+(x_s-x_{s-} )  }\land 0, \quad t \in [T^{[k]}, T^{[k+1]}).
\end{aligned}
\label{48}
\end{align}
Since \eqref{47} holds for $t\in [0, T^{(1)})$ and \eqref{48} holds for $k\in \bN$, we obtain \eqref{47} for $t\geq 0$. 
\par
(iii) In the third step, we consider the general cases. 
For $\varepsilon\geq 0$, we define the function $t\mapsto \tilde{x}^{(\varepsilon)}_t$ on $[0, \infty)$ as 
{
\begin{align}
\tilde{x}^{(\varepsilon)}_t = 
x_t-\sum_{s\in(0, t]}( x_s-x_{s-})1_{\{( x_s-x_{s-}) \in (0, \varepsilon) \}}. 
\end{align} }
Here, we assume that $t\mapsto \tilde{x}^{(0)}_t$ implies $t\mapsto x_t$. 
Note that for $\varepsilon>0$, the function $t\mapsto \tilde{x}^{(\varepsilon)}_t$ 
satisfies the assumption in (ii). 
We define
\begin{align}
I^{\tilde{x}^{(\varepsilon)}}_t {:=} \inf_{s\in[0, t]}\rbra{\tilde{x}^{(\varepsilon)}_{{s}}\land 0}, \quad 
\check{x}^{(\varepsilon)}_t {:=} \tilde{x}^{(\varepsilon)}_t - I^{\tilde{x}^{(\varepsilon)}}_t, \quad t\geq 0.
\end{align}
By the definition of $\tilde{x}^{(\varepsilon)}$, for $0\leq \varepsilon_1<\varepsilon_2$, the function 
\begin{align}
t \mapsto \tilde{x}^{(\varepsilon_1)}_t - \tilde{x}^{(\varepsilon_2)}_t
=\sum_{s\in( 0, t]}(x_s-x_{s-})1_{\{x_s - x_{s-} \in[\varepsilon_1 ,\varepsilon_2)\}}\label{104}
\end{align}
is non-decreasing and takes values in $[0, \infty)$. 
For $\varepsilon\geq 0$, we write
\begin{align}
T^{(\varepsilon)}_{\inf} (t){:=}\sup \cbra{ s\in[0, t]: \tilde{x}^{(\varepsilon)}_{s-} \land \tilde{x}^{(\varepsilon)}_s=\inf_{u\in[0,s]}\tilde{x}^{(\varepsilon)}_u } ,\quad \quad t\geq 0, \label{126}
\end{align}
where $\tilde{x}^{(\varepsilon)}_{0-}=\tilde{x}^{(\varepsilon)}_0=x_0$. 
From \eqref{126}, the left-continuity of $s \mapsto \tilde{x}^{(\varepsilon)}_{s-}$, and the right-continuity of $s\mapsto \tilde{x}^{(\varepsilon)}_{s}$ , for $s\in(T_{\inf }^{(\varepsilon)} , t]$,  
there exist $\varepsilon_s, \delta_s>0 $ such that $\tilde{x}^{(\varepsilon)}_\mu > \inf_{u\in [0, s]}\tilde{x}^{(\varepsilon)}_u +\delta_s$ for $\mu\in (s-\varepsilon_s ,s+\varepsilon_s )$, 
and thus we have, for $\varepsilon\geq 0$,  
\begin{align}
I^{\tilde{x}^{(\varepsilon)}}_t = \tilde{x}^{(\varepsilon)}_{T^{(\varepsilon)}_{\inf} (t)-}\land \tilde{x}^{(\varepsilon)}_{T^{(\varepsilon)}_{\inf} (t)}\land 0,\quad t\geq 0. \label{105}
\end{align}
Since the function \eqref{104} is non-negative and by \eqref{105}, we have, for $0\leq \varepsilon_1<\varepsilon_2$, 
\begin{align}
0\leq
\tilde{x}^{(\varepsilon_1)}_{T^{(\varepsilon_1)}_{\inf} (t)-}&\land \tilde{x}^{(\varepsilon_1)}_{T^{(\varepsilon_1)}_{\inf} (t)}\land 0
-\tilde{x}^{(\varepsilon_2)}_{T^{(\varepsilon_1)}_{\inf} (t)-}\land \tilde{x}^{(\varepsilon_2)}_{T^{(\varepsilon_1)}_{\inf} (t)}\land 0
\leq
I^{\tilde{x}^{(\varepsilon_1)}}_t-I^{\tilde{x}^{(\varepsilon_2)}}_t\\
&\leq 
\tilde{x}^{(\varepsilon_1)}_{T^{(\varepsilon_2)}_{\inf} (t)-}\land \tilde{x}^{(\varepsilon_1)}_{T^{(\varepsilon_2)}_{\inf} (t)}\land 0
-\tilde{x}^{(\varepsilon_2)}_{T^{(\varepsilon_2)}_{\inf} (t)-}\land \tilde{x}^{(\varepsilon_2)}_{T^{(\varepsilon_2)}_{\inf} (t)}\land 0.
\label{106}
\end{align}
By the definition of $\tilde{x}^{(\varepsilon_1)}$ and $\tilde{x}^{(\varepsilon_2)}$, 
$\tilde{x}^{(\varepsilon_1)}$ and $\tilde{x}^{(\varepsilon_2)}$ have the same negative jumps at the same times. 
By the above fact and since the function \eqref{104} is non-decreasing, we have, for $0\leq \varepsilon_1<\varepsilon_2$,
\begin{align}
\tilde{x}^{(\varepsilon_1)}_{T^{(\varepsilon_2)}_{\inf} (t)-}\land \tilde{x}^{(\varepsilon_1)}_{T^{(\varepsilon_2)}_{\inf} (t)}\land 0
-\tilde{x}^{(\varepsilon_2)}_{T^{(\varepsilon_2)}_{\inf} (t)-}\land \tilde{x}^{(\varepsilon_2)}_{T^{(\varepsilon_2)}_{\inf} (t)}\land 0
=&\tilde{x}^{(\varepsilon_1)}_{T^{(\varepsilon_2)}_{\inf} (t)-}\land 0
-\tilde{x}^{(\varepsilon_2)}_{T^{(\varepsilon_2)}_{\inf} (t)-}\land 0\\
\leq &\tilde{x}^{(\varepsilon_1)}_t-\tilde{x}^{(\varepsilon_2)}_t,\quad t\geq 0. 
\label{107}
\end{align}
By \eqref{106}, \eqref{107}, and the definition of $\check{x}^{(\cdot)}$, we have, for $0\leq \varepsilon_1<\varepsilon_2$,
\begin{align}
I^{\tilde{x}^{(\varepsilon_1)}}_t-I^{\tilde{x}^{(\varepsilon_2)}}_t 
\in \sbra{0, \tilde{x}^{(\varepsilon_1)}_t-\tilde{x}^{(\varepsilon_2)}_t },\quad  \check{x}^{(\varepsilon_1)}_t-\check{x}^{(\varepsilon_2)}_t\in \sbra{0, \tilde{x}^{(\varepsilon_1)}_t-\tilde{x}^{(\varepsilon_2)}_t }, \quad t\geq 0. \label{108}
\end{align}
From (ii), we have, for $\varepsilon>0$, 
\begin{align}
I^{\tilde{x}^{(\varepsilon)}}_t  = \delta \int_0^t 1_{\{ 0\}} (\check{x}^{(\varepsilon)}_s) ds+ (x_0\land 0)+
\sum_{s\in(0, t]} \rbra{ \check{x}^{(\varepsilon)}_{s-}+(\tilde{x}^{(\varepsilon)}_s-\tilde{x}^{(\varepsilon)}_{s-} )  }\land 0 ,\quad t\geq 0. 
\label{51}
\end{align}
By the definition of $\tilde{x}^{(\varepsilon)}$ and by \eqref{108}, we have 
\begin{align}
&I^{\tilde{x}^{(\varepsilon)}}_t\uparrow I^x_t, \qquad
\check{x}^{(\varepsilon)}_t\uparrow \check{x}_t  ,\qquad t\geq 0  \label{50}
\end{align}
uniformly on compact intervals as $\varepsilon\downarrow 0$. 
In addition, since all $\tilde{x}^{(\varepsilon)}$ for $\varepsilon\geq 0$ have the same negative jumps at the same times and by \eqref{50} and the dominated convergence theorem, we have 
\begin{align}
\sum_{s\in(0, t]} &\rbra{ \check{x}^{(\varepsilon)}_{s-}+(\tilde{x}^{(\varepsilon)}_s-\tilde{x}^{(\varepsilon)}_{s-} )  }\land 0
\uparrow\sum_{s\in(0, t]} \rbra{ \check{x}_{s-}+(x_s-x_{s-} )  }\land 0,
\quad t\geq 0\label{110}
\end{align}
uniformly on compact intervals as $\varepsilon\downarrow 0$. 
{Note that we can use the dominated convergence theorem in \eqref{110} since we have, 
for $\varepsilon >0$ and $s \in (0, t]$,
\begin{align}
 \absol{\rbra{ \check{x}^{(\varepsilon)}_{s-}+(\tilde{x}^{(\varepsilon)}_s-\tilde{x}^{(\varepsilon)}_{s-} )  }\land 0}\leq |\rbra{ x_s-x_{s-}   }\land 0|, 
 \end{align}
and  since $t\mapsto x_t$ has bounded variation paths which implies, 
\begin{align}
\sum_{s\in(0, t]}|\rbra{ x_s-x_{s-}   }\land 0|<\infty. 
\end{align}
}
Since $\check{x}^{(\varepsilon)}_t$ converges to $\check{x}_t$ non-decreasingly as $\varepsilon\downarrow 0$ for $t\geq 0$, we have
\begin{align}
\begin{cases}
\text{if }\check{x}_t>0, \text{ then we have }\check{x}^{(\varepsilon)}_t>0\text{ for sufficiently small }\varepsilon>0,\\
\text{if }\check{x}_t=0, \text{ then we have }\check{x}^{(\varepsilon)}_t=0\text{ for all }\varepsilon>0, 
\end{cases}
\quad t\geq 0 .
\label{111}
\end{align}
By \eqref{111} and the dominated convergence theorem, we have
\begin{align}
\delta \int_0^t 1_{\{ 0\}} (\check{x}^{(\varepsilon)}_s) ds \uparrow \delta \int_0^t 1_{\{ 0\}} (\check{x}_s) ds
,\quad t\geq 0 \label{52}
\end{align}
as $\varepsilon \downarrow 0$. 
By taking the limit of \eqref{51} as $\varepsilon\downarrow 0$ and by \eqref{50}, \eqref{110}, and \eqref{52}, we obtain \eqref{47}. 
The proof is complete. 
\end{proof}
{
\begin{Rem}\label{RemA03}
We consider the case where $\delta>0$ holds in the same situation as Lemma \ref{LemB01}. Then we have, for $t\geq 0 $,  
\begin{align}
I^x_t  = (x_0\land 0)+
\sum_{s\in(0, t]} \rbra{ \check{x}_{s-}+(x_s-x_{s-} )  }\land 0.
\end{align}
It can be proved by first assuming that the map $t\mapsto x_t$ has only a finite number of negative jumps on each compact interval in $[0, \infty)$, and then considering the approximation as in (iii) in the proof of Lemma \ref{LemB01}.  
The proof is similar to A but simpler and is therefore omitted. 
\end{Rem}
}

\section{Existence and uniqueness of the solution of an ordinary differential equation}\label{ODE}
Let $t\mapsto x_t$ be the c\`adl\`ag function from $[0,\infty)$ to $\bR$ that has bounded variation and satisfies \eqref{97} for some $\delta  \in \bR$. 
We 
define the function $t\mapsto x^{(\alpha)}_t$ as 
\begin{align}
x^{(\alpha)}_t {:=} x_t -\alpha t,\quad t\geq 0. 
\end{align}
In this section, we prove the following theorem. 
\begin{Thm}\label{Thm601}
For $b\in\bR$, there exists the unique function $t\mapsto y^b_t$ that satisfies 
\begin{align}
y^b_t = x_t -  \int_0^t h^b (y^b_s) d s,  \quad t \geq 0, \label{114}
\end{align}
where 
\begin{align}
h^b(y){:=}
\begin{cases}
\alpha1_{(b, \infty)}(y), \quad &\delta\not\in[0,\alpha], \\
\alpha1_{(b, \infty)}(y)+\delta 1_{\{b\}}(y), \quad &\delta \in [0, \alpha].
\end{cases}
\end{align}
\end{Thm}
Using this theorem, we can prove Theorem \ref{Lem204} directly. 
After that, we give the proof of Theorem \ref{Thm601}.
\par
The proof of the uniqueness is the same as the proof for spectrally negative cases in \cite[pp.~287]{Kyp2014} but changing from $\alpha 1_{\{\cdot>0\}}$ to $h^b(\cdot)$ and from $X_t$ to $x_t$. 
Thus, it is enough to prove the existence of a solution in the following.  
\subsection{Special cases}\label{SecA01}
We assume that the function $t\mapsto x_t$ satisfies one of the following. 
\begin{itemize}
\item[Case A] We have $\delta \not \in [0, \alpha]$. 
\item[Case B] The function $t\mapsto x_t$ has only a finite number of negative jumps on each compact interval in $[0, \infty)$ and we have $\delta\in[0, \alpha]$. 
\item[Case C] The function $t\mapsto x_t$ has only a finite number of positive jumps on each compact interval in $[0, \infty)$ and we have $\delta\in[0, \alpha]$. 
\end{itemize}
We prove that \eqref{114} has a solution when the function $t\mapsto x_t$ satisfies one of the conditions above. 
Here, we use the same idea as the proof of that in \cite[Theorem 10.4]{Kyp2014}. 
\begin{itemize}
\item[Case A]
In this case, we can construct a solution in the same way as the construction of a strong solution of \eqref{114} in the spectrally negative cases with bounded variation paths; 
see \cite[pp.~30]{KypLoe2010} or \cite[the proof of Theorem 10.4]{Kyp2014}. 
\item[Case B] 
For $k\in\bN$, we define functions and times inductively as follows:
\begin{align}
\begin{aligned}
\un{y}^{b, [k]}_t {:=}& \bar{y}^{b, [k-1]}_{\un{S}^{[k-1]}-}
+ (x_t - x_{\un{S}^{[k-1]}-}), \quad t \geq \un{S}^{[k-1]}, 
\\
\bar{S}^{[k]}{:=}&\inf\cbra{ t \geq  \un{S}^{[k-1]} : \un{y}^{b, [k]}_t  \geq b
} ,
\\
\bar{y}^{b, [k]}_t{:=}&\un{y}^{b, [k]}_{\bar{S}^{[k]}} + (x^{(\alpha)}_t - x^{(\alpha)}_{\bar{S}^{[k]}})\\
&\qquad \qquad -   \inf_{s\in [\bar{S}^{[k]}, t] }\rbra{ (\un{y}^{b, [k]}_{\bar{S}^{[k]}} + (x^{(\alpha)}_s - x^{(\alpha)}_{\bar{S}^{[k]}})-b)\land 0} ,
\quad t\geq \bar{S}^{[k]}, 
\\
\un{S}^{[k]}{:=}&\inf\cbra{ t>\bar{S}^{[k]} : \bar{y}^{b, [k]}_{t-}+ 
(x_t-x_{t-}) < b} ,
\label{30}
\end{aligned}
\end{align}
where 
$\un{S}^{[0]}{:=}0$ and $\bar{y}^{b, [0]}_{0-}=\bar{y}^{b, [0]}_{0}{:=}x_{0-}=x_0$. 
Note that $\bar{y}^{b,[k]}$ takes values in $[b, \infty)$ and behaves as a reflected L\'evy process at $b$ when $t\mapsto x_t$ is a path of L\'evy process. 
Since $t\mapsto x_t$ has finitely many negative jumps on each compact interval, we have 
\begin{align}
\lim_{k\uparrow \infty}\bar{S}^{[k]}=\infty.\label{112}
\end{align}
Thus, we can define the function $t\mapsto y^{b,\ast}_t$ on $[0, \infty)$ as 
\begin{align}
y^{b,\ast}_t {:=}
\begin{cases}
\un{y}^{b, [k]}_t, \quad &t\in[\un{S}^{[k-1]}, \bar{S}^{[k]}) ,
\\
\bar{y}^{b, [k]}_t, \quad & t\in[\bar{S}^{[k]} , \un{S}^{[k]}) . 
\end{cases}
\label{13}
\end{align}
We check that the process $y^{b, \ast}$ satisfies \eqref{114}. 
We have, for $k\in\bN$, 
\begin{align}
\bar{S}^{[k]}=\inf\{t\geq  \un{S}^{[k-1]}:  y^{b,\ast}_t \geq b   \}, \quad
\un{S}^{[k]}=\inf\{ t>\bar{S}^{[k]}:  y^{b,\ast}_t< b  \}. \label{77}
\end{align} 
Thus we have, for $k\in\bN$,
\begin{align}
\un{y}^{b, [k]}_t =& \bar{y}^{b, [k-1]}_{\un{S}^{[k-1]}-}+ (x_t - x_{\un{S}^{[k-1]}-})
\\
=&\bar{y}^{b, [k-1]}_{\un{S}^{[k-1]}-}+ (x_t - x_{\un{S}^{[k-1]}-})
-\int_{\un{S}^{[k-1]}}^t h^b(\un{y}^{b, [k]}_s)ds
, \quad t \in [\un{S}^{[k-1]}, \bar{S}^{[k]}). \label{121}
\end{align}
In addition, by Lemma \ref{LemB01}, we have, for $k\in\bN$, 
\begin{align}
\begin{aligned}
\inf_{s\in [\bar{S}^{[k]}, t] }&\rbra{ (\un{y}^{b, [k]}_{\bar{S}^{[k]}} + (x^{(\alpha)}_s - x^{(\alpha)}_{\bar{S}^{[k]}})-b)\land 0}\\
&\qquad=(\delta - \alpha) \int_{\bar{S}^{[k]}}^t 1_{\{b\}} (\bar{y}^{b, [k]}_s) ds, \quad t \in [\bar{S}^{[k]}, \un{S}^{[k]}). 
\end{aligned}
\label{15}
\end{align}
By the definition of $t\mapsto \bar{y}^{b,[k]}_t$ and \eqref{15}, we have, for $k\in\bN$, 
\begin{align}
\bar{y}^{b,[k]}_t&= \un{y}^{b, [k]}_{\bar{S}^{[k]}} + (x^{(\alpha)}_t - x^{(\alpha)}_{\bar{S}^{[k]}}) - (\delta- \alpha)
\int_{\bar{S}^{[k]}}^t 1_{\{b\}} (\bar{y}^{b, [k]}_s)ds \\
&=\un{y}^{b, [k]}_{\bar{S}^{[k]}} + (x_t - x_{\bar{S}^{[k]}}) -
\int_{\bar{S}^{[k]}}^t h^b(\bar{y}^{b, [k]}_s)ds,\quad t \in [\bar{S}^{[k]}, \un{S}^{[k]}). \label{16}
\end{align}
By \eqref{13}, \eqref{121}, 
and \eqref{16}, the function $t\mapsto y^{b, \ast}_t$ satisfies \eqref{114}. 
\item[Case C] 
{
In Case B, the solution could be constructed by joining paths whose motion is reflected by $b$ when taking values above $b$ and whose drift is subtracted by $\alpha$ when taking values below $b$.  Similarly in Case C, the solution can be constructed by joining paths that have the original motion when taking values above $b$, and paths that have the motion reflected at $b$ with the drift subtracted by $\alpha$ when taking values below $b$. The proof is almost the same as in Case B and is therefore omitted. }
\end{itemize}
\par
The proof in special cases {in Theorem \ref{Thm601}} is complete.

\subsection{Other cases}
We assume that  $\delta\in[0, \alpha]$. 
For $\varepsilon \in {\bR} $, we define the function $t\mapsto \tilde{x}^{(\varepsilon)}_t $ on $[0, \infty)$ 
as 
\begin{align}
\begin{aligned}
\tilde{x}^{(\varepsilon)}_t = 
\begin{cases}
 x_t-\sum_{s\in(0, t]}( x_s-x_{s-})1_{\{ (x_s-x_{s-}) \in (\varepsilon, 0) \}} , \quad &\varepsilon<0,\\
x_t-\sum_{s\in(0, t]}( x_s-x_{s-})1_{\{( x_s-x_{s-}) \in (0, \varepsilon) \}},   \quad &\varepsilon {\geq }0. 
\end{cases}
\end{aligned}
\end{align}
Then, the function $t\mapsto \tilde{x}^{(\varepsilon)}_t$ is a function that belongs to Case B in Section \ref{SecA01} when $\varepsilon<0$ and Case C 
in Section \ref{SecA01} when $\varepsilon > 0$ 
and that converges to the function $t\mapsto x_t$ uniformly on compact intervals as $\varepsilon \to 0$ 
since the function $t\mapsto x_t$ has bounded variation and by the definition of $t\mapsto x^{(\varepsilon)}_t$. 
In addition, the function $t\mapsto \tilde{x}^{(\varepsilon)}_t$ satisfies 
\begin{align}
\tilde{x}^{(\varepsilon_2)}_t \leq \tilde{x}^{(\varepsilon_1)}_t, \qquad-\infty< \varepsilon_1< \varepsilon_2<\infty, ~t\geq 0 .
\end{align}
\par
For $\varepsilon\neq 0$, let the function $t\mapsto \tilde{y}^{(\varepsilon)}_t$ on $[0, \infty)$ be the solution of \eqref{114} driven by the function $t\mapsto\tilde{x}^{(\varepsilon)}_t$. 
For $\varepsilon_1, \varepsilon_2\in\bR\backslash\{0\}$ with $\varepsilon_1< \varepsilon_2$, we have 
\begin{align}
\tilde{y}^{(\varepsilon_1)}_t -\tilde{y}^{(\varepsilon_2)}_t =\tilde{x}^{(\varepsilon_1)}_t -\tilde{x}^{(\varepsilon_2)}_t-
H^b_t, \quad t\geq 0, \label{131}
\end{align}
where 
\begin{align}
H^b_t = \int_0^t \rbra{h^b (\tilde{y}^{(\varepsilon_1)}_s)-h^b (\tilde{y}^{(\varepsilon_2)}_s) } d s, \quad t\geq 0.
\end{align}
Since $h^b$ is non-decreasing, we have
\begin{align}
h^b (\tilde{y}^{(\varepsilon_1)}_s)-h^b (\tilde{y}^{(\varepsilon_2)}_s) 
\begin{cases}
\geq 0 ,\quad &\tilde{y}^{(\varepsilon_1)}_s \geq \tilde{y}^{(\varepsilon_2)}_s,\\
\leq 0 ,\quad &\tilde{y}^{(\varepsilon_1)}_s \leq \tilde{y}^{(\varepsilon_2)}_s, 
\end{cases}
\quad s\geq 0. \label{22}
\end{align}
We want to check that 
\begin{align}
\tilde{y}^{(\varepsilon_1)}_t -\tilde{y}^{(\varepsilon_2)}_t \in[ 0, \tilde{x}^{(\varepsilon_1)}_t -\tilde{x}^{(\varepsilon_2)}_t], \quad t\geq 0. \label{24}
\end{align}
To demonstrate, we assume that $\tilde{y}^{(\varepsilon_1)}_t -\tilde{y}^{(\varepsilon_2)}_t<0$ for some $t>0$. 
Since the function $t\mapsto \tilde{x}^{(\varepsilon_1)}_t -\tilde{x}^{(\varepsilon_2)}_t$ does not have negative jumps and $t\mapsto H^b_t $ is continuous, there exist $t_0, t_1 \in[0,\infty)$ with $t_0<t_1$ such that 
\begin{align}
\tilde{y}^{(\varepsilon_1)}_{t_1-} -\tilde{y}^{(\varepsilon_2)}_{t_1-} <\tilde{y}^{(\varepsilon_1)}_t -\tilde{y}^{(\varepsilon_2)}_t <0
, \quad t\in (t_0, t_1). \label{115}
\end{align}
Since $t\mapsto \tilde{x}^{(\varepsilon_1)}_t -\tilde{x}^{(\varepsilon_2)}_t$ is non-decreasing {and by \eqref{131}}, 
there exists $t_2\in (t_0, t_1)$ such that 
\begin{align}
h^b (\tilde{y}^{(\varepsilon_1)}_{t_2})-h^b (\tilde{y}^{(\varepsilon_2)}_{t_2})>0. \label{116}
\end{align} 
However, \eqref{115} and \eqref{116} contradict \eqref{22}, and thus we have 
\begin{align}
\tilde{y}^{(\varepsilon_1)}_t -\tilde{y}^{(\varepsilon_2)}_t \geq 0, 
\quad 
h^b (\tilde{y}^{(\varepsilon_1)}_t)-h^b (\tilde{y}^{(\varepsilon_2)}_t)\geq 0, 
\quad t\geq 0. \label{113}
\end{align}
By \eqref{113}, $t\mapsto H^b_t$ is a non-decreasing process that takes only non-negative values. 
Therefore, we obtain \eqref{24}. 
\par
From \eqref{24}, the function $\varepsilon \mapsto \tilde{y}^{(\varepsilon)}_t$ is non-increasing on $\bR$ for $t\geq 0$, and thus we can define 
{the function} $t\mapsto \tilde{y}_t$ as 
\begin{align}
\tilde{y}_t = \lim_{\varepsilon \downarrow 0}\tilde{y}^{(\varepsilon)}_t= \lim_{\varepsilon \uparrow 0}\tilde{y}^{(\varepsilon)}_t, \quad t\geq 0,
\end{align}
where in the last equality we used 
\begin{align}
\lim_{\varepsilon \downarrow 0}\rbra{\tilde{y}^{(-\varepsilon)}_t-\tilde{y}^{(\varepsilon)}_t}=0, \quad t\geq 0 ,
\end{align}
which comes from \eqref{24} and $\lim_{\varepsilon \to 0} x^{(\varepsilon)}_t = x_t$ for $t\geq 0$. 
We want to check that the function $t\mapsto \tilde{y}_t$ is a solution of \eqref{114} driven by the function $t\mapsto x_t$. 
By the definition of $h^b$ and since $\tilde{y}^{(\varepsilon)}_t\leq\tilde{y}_t\leq\tilde{y}^{(-\varepsilon)}_t$ for $\varepsilon >0$ and $t\geq 0$, we have
\begin{align}
\int_0^t h^b (\tilde{y}^{(\varepsilon)}_s) d s \leq \int_0^t h^b (\tilde{y}_s) d s \leq \int_0^t h^b (\tilde{y}^{(-\varepsilon)}_s) d s, 
\qquad \varepsilon>0, ~t\geq 0.\label{25}
\end{align}
Since we have $\lim_{\varepsilon\to 0}\tilde{x}^{(\varepsilon)}_t=x_t$ and $\lim_{\varepsilon\to 0}\tilde{y}^{(\varepsilon)}_t=\tilde{y}_t$ for $t\geq 0$ and since the functions $t\mapsto \tilde{y}^{(-\varepsilon)}_t$ and $t\mapsto\tilde{y}^{(\varepsilon)}_t$ satisfy \eqref{114} driven by the functions $t\mapsto \tilde{x}^{(-\varepsilon)}_t$ and $t\mapsto \tilde{x}^{(\varepsilon)}_t$, respectively, we have 
\begin{align}
\lim_{\varepsilon \downarrow 0}\int_0^t h^b (\tilde{y}^{(-\varepsilon)}_s) d s = \lim_{\varepsilon\downarrow0}\int_0^t h^b (\tilde{y}^{(\varepsilon)}_s) d s, \quad t\geq 0.\label{26}
\end{align}
From \eqref{25} and \eqref{26}, we have 
\begin{align}
\lim_{\varepsilon\to0}\int_0^t h^b (\tilde{y}^{(\varepsilon)}_s) d s=\int_0^t h^b (\tilde{y}_s) d s,\quad t\geq 0. \label{27}
\end{align}
Since we have $\lim_{\varepsilon\to 0}\tilde{x}^{{(\varepsilon)}}_t=x_t$ and $\lim_{\varepsilon\to 0}\tilde{y}^{(\varepsilon)}_t=\tilde{y}_t$ for $t\geq 0$ and the function $t\mapsto\tilde{y}^{(\varepsilon)}_t$ satisfies \eqref{114} driven by the function $t\mapsto\tilde{x}^{(\varepsilon)}_t$ for $\varepsilon\in\bR\backslash\{0\}$, and by \eqref{27}, the function $t\mapsto\tilde{y}_t$ satisfies \eqref{114} driven by the function $t\mapsto x_t$. 
\par
The proof {of Theorem \ref{Thm601}} is complete.


\section{Behavior of $U^{\pi^0}$}\label{Sec00B}
In this appendix, we confirm that the process $U^{\pi^0}$, which is the resulting process of the refraction--reflection strategy at $ 0$ defined in Section \ref{Sec302}, corresponds to the process $Z^0$.
\par
We consider the proof in three cases. 
\par
(i) We assume that $X$ 
has unbounded variation paths or bounded variation paths with $\delta > \alpha$. 
When $X$ has unbounded variation paths, since $0$ is regular for $(-\infty , 0]$ for $X^{(\alpha)}$ and by \cite[Theorem 6.7]{Kyp2014}, we have 
\begin{align}
\int_0^t 1_{(0, \infty)} (Z^{0}_s)ds=t ,\quad t\geq 0.\label{53}
\end{align}
When $X$ has bounded variation paths with $\delta > \alpha$, since $0$ is irregular for $(-\infty , 0]$ for $X^{(\alpha)}$, 
the process $Z^0$ takes a value $0$ only a finite number of times in a finite time, and thus 
we have \eqref{53}. 
From \eqref{38} and \eqref{53}, we have 
\begin{align}
L^{\pi^0}_t=\int_0^t l^{\pi^0}_s ds= \int_0^t \alpha 1_{(0, \infty)} (Z^{0}_s)ds=\alpha t, \quad t\geq 0. \label{35}
\end{align}
From \eqref{38} and \eqref{35}, we have
\begin{align}
R^{\pi^0}_t =-\inf_{s\in[0, t]} \rbra{\rbra{ X_s -L^{\pi^0}_s  } \land 0}=-\inf_{s\in[0, t]}\rbra{X^{(\alpha)}_s\land 0},
 \quad   t\geq 0. \label{37}
\end{align}
From \eqref{35}, \eqref{37}, and \eqref{36}, we have 
\begin{align}
U^{\pi^0}_t =X_t-L^{\pi^0}_t+R^{\pi^0}_t =Z^0_t, \quad t\geq 0.
\end{align}
\par
(ii) We assume that $X$ 
has bounded variation paths with $\delta\leq \alpha$. 
From \eqref{38}, we have
\begin{align}
&U^{\pi^0}_t = X_t-\int_0^t h^{0} (Z^{0}_s)ds-\inf_{u\in[0, t]} \rbra{\rbra{ X_u -\int_0^u h^{0} (Z^{0}_s)ds  } \land 0}\\
&=X^{(\alpha)}_t+ (\alpha-{(}\delta\lor 0{)})\int_0^t 1_{\{0\}} (Z^{0}_s)ds-\inf_{u\in[0, t]} \rbra{\rbra{ X^{(\alpha)}_u +(\alpha-{(}\delta \lor 0{)})\int_0^u 1_{\{0\}} (Z^{0}_s)ds  } \land 0}. 
\end{align}
{Note that the definition of $h^0$ depends on the value of $\delta$ since Case $1$ applies if $\delta<0$ holds and Case $2$ applies if $\delta\geq0$ holds.}
By comparing with \eqref{36}, 
it is enough to show that 
\begin{align}
\begin{aligned}
(\alpha-{(}\delta\lor 0{)})\int_0^t 1_{\{0\}} (Z^{0}_s)ds-\inf_{u\in[0, t]} \rbra{\rbra{ X^{(\alpha)}_u +(\alpha-{(}\delta \lor 0{)})\int_0^u 1_{\{0\}}  (Z^{0}_s)ds  } \land 0} \ \ &\\
\qquad \qquad \qquad =-\inf_{s\in[0, t]}\rbra{X^{(\alpha)}_s\land 0}, \quad t\geq 0 .
\end{aligned}
\label{54}
\end{align} 
We fix $t\geq 0$.
If $Z^0_s\neq 0$ for $s \in[0,t]$, we have 
\begin{align}
(\alpha-{(}\delta\lor 0{)})\int_0^t 1_{\{0\}} (Z^{0}_s)ds=0,
\end{align}
and thus \eqref{54} is true. 
We assume that $Z^0_s= 0$ for some $s \in[0,t]$ and define  
\begin{align}
T_{\inf}{:=}& \sup \cbra{s\in [0, t ] : Z^0_{s-}\land Z^0_s =0}\label{55}\\
{=}&{\sup \cbra{s\in [0, t ] : X^{(\alpha)}_{s-} \land X^{(\alpha)}_s = \inf_{u\in[0, s]}X^{(\alpha)}_u }}
. \label{56}
\end{align}
By \eqref{56} and the same argument as that of the proof of \eqref{105}, we have 
\begin{align}
-\inf_{s\in[0, t]}\rbra{X^{(\alpha)}_s\land 0}=-\rbra{X^{(\alpha)}_{T_{\inf}-}\land X^{(\alpha)}_{T_{\inf}}}. \label{59}
\end{align}
Since $T_{\inf}$ exists, and by \eqref{36} with Lemma \ref{LemB01}, and \eqref{55}, we have 
{
\begin{align}
&
\begin{aligned}
&(\text{the left hand side of \eqref{54}})=
(\alpha-{(}\delta \lor 0{)})\int_0^{T_{\inf}} 1_{\{0\}} (Z^{0}_s)ds-\inf_{u\in[0, t]} \bigg{(}Z^0_u \\
&+ (\delta\land 0)\int_0^{u\land T_{\inf}} 1_{\{0\}}(Z^0_s)ds
 +(X^{(\alpha)}_0 \land 0)+\sum_{s\in(0, u\land T_{\inf}]} \rbra{\rbra{ Z^0_{s-}+(X^{(\alpha)}_s-X^{(\alpha)}_{s-} )  }\land 0 }\bigg{)} . 
\end{aligned}
\label{57}
\end{align}
}
Since the map 
\begin{align}
u\mapsto (\delta\land 0)\int_0^{u\land T_{\inf}} 1_{\{0\}}(Z^0_s)ds+\sum_{s\in(0, u\land T_{\inf}]} \rbra{\rbra{ Z^0_{s-}+(X^{(\alpha)}_s-X^{(\alpha)}_{s-} )  }\land 0}
\end{align}
is non-increasing on $[0, t]$, $Z^0$ takes non-negative values, and $Z^0_{T_{\inf}-}\land Z^0_{T_{\inf}}=0$, we have
{
\begin{align}
&(\alpha-{(}\delta \lor 0{)})\int_0^{T_{\inf}} 1_{\{0\}} (Z^{0}_s)ds- \bigg{(}\rbra{Z^0_{T_{\inf}-}\land Z^0_{T_{\inf}}} +(\delta\land 0)\int_0^{ T_{\inf}} 1_{\{0\}}(Z^0_s)ds\\
&+(X^{(\alpha)}_0 \land 0) +\sum_{s\in(0,  T_{\inf}]} \rbra{\rbra{ Z^0_{s-}+(X^{(\alpha)}_s-X^{(\alpha)}_{s-} )  }\land 0} \bigg{)}=- \rbra{X^{(\alpha)}_{T_{\inf}-}\land X^{(\alpha)}_{T_{\inf}}} 
, \label{58}
\end{align}}
where in the last equality we used \eqref{36} with Lemma \ref{LemB01}. 
From \eqref{59}, \eqref{57}, and \eqref{58}, we obtain \eqref{54}. The proof is complete.

\section{Behavior of two paths} \label{two_paths}
In this appendix, we consider the behavior of two paths of refracted--reflected processes. 
{Here, we use the same notation as in the proof of Lemma \ref{Lem302}.}
\begin{Lem}\label{Lem_paths}
We fix $b>0$, $x\in\bR$, and $k, l\in\bR$ with $l-k\in(0, b)$. 
Then, $\bP_x$-a.s., we have the following.
\begin{enumerate}
\renewcommand{\labelenumi}{(\arabic{enumi})}
\item 
We have
\begin{align}
(Z^{[l],b}_t-Z^{[k],b}_t )+(L^{[l],b}_t-L^{[k],b}_t)-(R^{[l],b}_t-R^{[k],b}_t)=l-k, \quad t\geq 0. \label{72}
\end{align}
\item
The process $\{Z^{[l],b}_t-Z^{[k],b}_t: t\geq 0\}$ is non-increasing and takes values in $[0, l-k]$. 
\item
The process $\{L^{[l],b}_t-L^{[k],b}_t: t\geq 0\}$ is non-decreasing and takes values in $[0, l-k]$. 
The support of its Stieltjes measure is included in the closure of $\{t \geq 0 : b\in[Z^{[k],b}_t, Z^{[l],b}_t], Z^{[k],b}_t\neq Z^{[l],b}_t\}$. 
\item
The process $\{R^{[l],b}_t-R^{[k],b}_t: t\geq 0\}$ is non-increasing and takes values in $[ -(l-k), 0]$. 
The support of its Stieltjes measure is included in the closure of $\{t \geq 0 : Z^{[k],b}_{t-} \land Z^{[k],b}_t=0\}$.
\end{enumerate}
{
\begin{Rem}\label{RemD02}
Lemma \ref{Lem_paths} is correct without the assumption "$l-k\in(0, b)$". 
In fact, for $k, l\in \bR$ with $l<k$, there exists $n\in\bN$ and $l_1, \cdots, l_n \in\bR$ with $l<l_1<\cdots<l_n<k$ such that $l_{m+1}-l_m \in (0, b)$ for $m=0 ,1 ,\cdots, n$ where $l_0:=l$ and $l_{n+1}:= k$. Then the statement of Lemma \ref{Lem_paths} with $l=l_m$ and $k=l_{m+1}$ is true for 
$m=1, \cdots, n$ and we can extend Lemma \ref{Lem_paths} to delete "$l-k\in(0, b)$" easily. 
\end{Rem}
}
\end{Lem}
\begin{proof}[Proof {of Lemma \ref{Lem_paths}}]
The identity \eqref{72} is true since the left-hand side of \eqref{72} is equal to $X^{[l]}_t -X^{[k]}_t$ for $t\geq 0$. 
\par
We define the stopping times $\bar{K}^{[n]}_b$ and $\un{K}^{[n]}_0$ as follows: for $n\in\bN$, 
\begin{align}
\bar{K}^{[n]}_b {:=} \inf \{t> \un{K}^{[n-1]}_0 : Z^{[l],b}_t \geq b \}, \quad
\un{K}^{[n]}_0 {:=} \inf \{t> \bar{K}^{[n]}_b : Z^{[k],b}_t =0\},
\end{align}
where $\un{K}^{[0]}_0{:=}0$. We assume that $Z^{[m], b}_{0-}{=X^{[m], b}_{0-}}=X^{[m], b}_0$, $L^{[m], b}_{0-}=0$, and $R^{[m], b}_{0-}=0$ for $m=k, l$. 
Then, we can observe the behavior of sample paths, inductively, as follows.
\par
(i) We observe the behavior for $t\in[\un{K}^{[n-1]}_0, \bar{K}^{[n]}_b)$. 
Note that $Z^{[l],b}_{\un{K}^{[n-1]}_0-}-Z^{[k],b}_{\un{K}^{[n-1]}_0-} { \in[0, l-k]}$ by (ii) on $t\in[\bar{K}^{[n-1]}_b, \un{K}^{[n-1]}_0)$ or the definition of $Z^{[m], b}_{0-}$ for $m=k, l$. 
From the definitions of the processes $Z^{[l],b}$ and $Z^{[k],b}$, they behave as the reflected L\'evy processes of {$X$} at $0$, respectively, 
{
and so we have
\begin{align}
&Z^{[l],{b}}_t-Z^{[k],{b}}_t
=Z^{[l],{b}}_{\un{K}^{[n-1]}_0-}+ \rbra{{X}_t - {X}_{\un{K}^{[n-1]}_0-}}
-\inf_{s\in[\un{K}^{[n-1]}_0, t ]}\bigg{(}\bigg{(}Z^{[l],{b}}_{\un{K}^{[n-1]}_0-}\\
&\quad + \rbra{{X}_s -{X}_{\un{K}^{[n-1]}_0-}}\bigg{)}\land 0\bigg{)}
-\Bigg{(}Z^{[k],{b}}_{\un{K}^{[n-1]}_0-}+ \rbra{{X}_t - {X}_{\un{K}^{[n-1]}_0-}}\\
&\quad -\inf_{s\in[\un{K}^{[n-1]}_0, t ]}\rbra{\rbra{Z^{[k],{b}}_{\un{K}^{[n-1]}_0-}+ \rbra{{X}_s - {X}_{\un{K}^{[n-1]}_0-}}}\land 0}\Bigg{)}
\\
&
\begin{aligned}
&=\rbra{Z^{[l],{b}}_{\un{K}^{[n-1]}_0-}-Z^{[k],{b}}_{\un{K}^{[n-1]}_0-}}
-\Bigg{(}\inf_{s\in[\un{K}^{[n-1]}_0, t ]}\rbra{\rbra{Z^{[l],{b}}_{\un{K}^{[n-1]}_0-}+ \rbra{{X}_s - {X}_{\un{K}^{[n-1]}_0-}}}\land 0}\\
&\quad -\inf_{s\in[\un{K}^{[n-1]}_0, t ]}\rbra{\rbra{Z^{[k],{b}}_{\un{K}^{[n-1]}_0-}+ \rbra{{X}_s - {X}_{\un{K}^{[n-1]}_0-}}}\land 0}
\Bigg{)},
\quad t\in [\un{K}^{[n-1]}_0, \bar{K}^{[n]}_b).
\end{aligned}
\label{Rev001}
\end{align}
}
{Thus} the behavior of $\{Z^{[l],b}_t-Z^{[k],b}_t : t\in [\un{K}^{[n-1]}_0, \bar{K}^{[n]}_b)\}$ is non-increasing 
and takes values 
in $[0, Z^{[l],b}_{\un{K}^{[n-1]}_0-}-Z^{[k],b}_{\un{K}^{[n-1]}_0-}]$. 
{
In fact, we define 
\begin{align}
\tau^{[m],n-1}_0=\inf\cbra{t>\un{K}^{[n-1]}_0: Z^{[m],{b}}_{\un{K}^{[n-1]}_0-}+ \rbra{X_s - X_{\un{K}^{[n-1]}_0-}}<0},\quad m=k, l, 
\end{align}
then we have, for $t\in [\un{K}^{[n-1]}_0, \tau^{[k],n-1}_0\land\bar{K}^{[n]}_b)$, 
\begin{align}
\eqref{Rev001}=Z^{[l],{b}}_{\un{K}^{[n-1]}_0-}-Z^{[k],{b}}_{\un{K}^{[n-1]}_0-}, 
\end{align}
for $t\in [ \tau^{[k],n-1}_0, \tau^{[l],n-1}_0\land\bar{K}^{[n]}_b)$, 
\begin{align}
\eqref{Rev001}=\rbra{Z^{[l],{b}}_{\un{K}^{[n-1]}_0-}-Z^{[k],{b}}_{\un{K}^{[n-1]}_0-}}
+\inf_{s\in[\un{K}^{[n-1]}_0, t ]}\rbra{Z^{[k],{b}}_{\un{K}^{[n-1]}_0-}+ \rbra{X_s - X_{\un{K}^{[n-1]}_0-}}}, 
\end{align}
and for $t\in [\tau^{[l],n-1}_0, \bar{K}^{[n]}_b)$, 
\begin{align}
&\rbra{Z^{[l],{b}}_{\un{K}^{[n-1]}_0-}-Z^{[k],{b}}_{\un{K}^{[n-1]}_0-}}
-\Bigg{(}\inf_{s\in[\un{K}^{[n-1]}_0, t ]}\rbra{Z^{[l],{b}}_{\un{K}^{[n-1]}_0-}+ \rbra{X_s - X_{\un{K}^{[n-1]}_0-}}}\\
&\qquad\qquad\qquad  -\inf_{s\in[\un{K}^{[n-1]}_0, t ]}\rbra{Z^{[k],{b}}_{\un{K}^{[n-1]}_0-}+ \rbra{X_s - X_{\un{K}^{[n-1]}_0-}}}
\Bigg{)}\\
&=\rbra{Z^{[l],{b}}_{\un{K}^{[n-1]}_0-}-Z^{[k],{b}}_{\un{K}^{[n-1]}_0-}}-\rbra{Z^{[l],{b}}_{\un{K}^{[n-1]}_0-}-Z^{[k],{b}}_{\un{K}^{[n-1]}_0-}}=0. 
\end{align}
}
In addition, the support of its Stieltjes measure is included in the closure of $\{t \in [\un{K}^{[n-1]}_0, \bar{K}^{[n]}_b) : Z^{[k],b}_{t-}\land Z^{[k],b}_t=0\}$ since 
the map 
\begin{align}
t\mapsto \inf_{s\in[\un{K}^{[n-1]}_0, t ]}\rbra{\rbra{Z^{[k],0}_{\un{K}^{[n-1]}_0-}+ \rbra{{X}_s - {X}_{\un{K}^{[n-1]}_0-}}}\land 0}
\end{align} 
may decrease on $\{t \geq 0 : Z^{[k],b}_{t-}\land Z^{[k],b}_t=0\}$ by the same argument as that after \eqref{55}. 
Since $Z^{[k],b}_t$ and $Z^{[l],b}_t$ take values in $[0, b)$ for $t\in (\un{K}^{[n-1]}_0, \bar{K}^{[n]}_{b})$, 
the process $\{L^{[l],b}_t-L^{[k],b}_t : t\in [\un{K}^{[n-1]}_0, \bar{K}^{[n]}_b)\}$ is a constant process satisfying 
\begin{align}
L^{[l],b}_t-L^{[k],b}_t=L^{[l],b}_{\bar{K}^{[n]}_b-}-L^{[k],b}_{\bar{K}^{[n]}_b-}, \quad t\in [\un{K}^{[n-1]}_0, \bar{K}^{[n]}_b). 
\end{align} 
From the above and \eqref{72}, the process $\{R^{[l],b}_t-R^{[k],b}_t : t\in [\un{K}^{[n-1]}_0, \bar{K}^{[n]}_b)\}$ is non-increasing (decreases only when $Z^{[l],b}_t-Z^{[k],b}_t$ decreases) and satisfies
\begin{align}
R^{[l],b}_t-R^{[k],b}_t=R^{[l],b}_{\un{K}^{[n-1]}_0-}-R^{[k],b}_{\un{K}^{[n-1]}_0-}+\rbra{\rbra{Z^{[l],b}_t-Z^{[k],b}_t}-\rbra{Z^{[l],b}_{\un{K}^{[n-1]}_0-}-Z^{[k],b}_{\un{K}^{[n-1]}_0-}}}  ,& \\
 t\in [ \un{K}^{[n-1]}_0 , \bar{K}^{[n]}_b)&. 
\end{align}
\par
(ii) We observe the behavior for $t\in[\bar{K}^{[n]}_b, \un{K}^{[n]}_0)$. 
From the definitions of the processes $Z^{[k],b}$ and $Z^{[l],b}$, they behave as the refracted L\'evy processes of $X^{[k]}$ and $X^{[l]}$ at $b$, respectively. 
From the definition of refracted L\'evy processes, we have 
\begin{align}
Z^{[l],b}_t-Z^{[k],b}_t=(Z^{[l],b}_{\bar{K}^{[n]}_b-}-Z^{[k],b}_{\bar{K}^{[n]}_b-})- \Delta(k,l, \bar{K}^{[n]}_b, t) , \label{7}
\end{align}
where 
\begin{align}
\Delta(k,l, s, t){:=} 
\int_s^t \rbra{ h^b(Z^{[l],b}_u)-h^b(Z^{[k],b}_u) } du \label{11}
\end{align}
for $0\leq s\leq t$.  
From \eqref{7} and since $Z^{[l],b}_{\bar{K}^{[n]}_b-}-Z^{[k],b}_{\bar{K}^{[n]}_b-}\in [0, l-k]$ by (i) and $h^b$ is non-decreasing, $t\mapsto Z^{[l],b}_t-Z^{[k],b}_t$ changes continuously, is non-increasing 
until the value reaches $ 0 $ 
and the rate of change over time is no more than $ -\alpha $. 
Thus, if $Z^{[l],b}_t-Z^{[k],b}_t$ takes values below $0$, is takes value $0$ before doing so. 
We confirm that $Z^{[l],b}_t-Z^{[k],b}_t$ does not take values below $0$. 
Assume that $Z^{[l],b}_{s(0)}-Z^{[k],b}_{s(0)}=0$ at time ${s(0)}\in [ \bar{K}^{[n]}_b , \un{K}^{[n]}_0) $. 
Then, by the uniqueness of the strong solution of the stochastic differential equation \eqref{4}, $\Delta(k,l, {s(0)},t)=0$ for $t\in [s(0), \un{K}^{[n]}_0)$. 
In conclusion, the process $\{Z^{[l],b}_t-Z^{[k],b}_t: t\in [ \bar{K}^{[n]}_b , \un{K}^{[n]}_0)\}$ is non-increasing 
and takes values in $[0, Z^{[l],b}_{\bar{K}^{[n]}_b-}-Z^{[k],b}_{\bar{K}^{[n]}_b-}]$. 
In addition, the support of its Stieltjes measure is included in the closure of $\{t \in [\un{K}^{[n-1]}_0, \bar{K}^{[n]}_b) : b\in[Z^{[k],b}_t, Z^{[l],b}_t], Z^{[k],b}_t\neq Z^{[l],b}_t\}$ by \eqref{7}. 
From \eqref{38}, we have 
\begin{align}
L^{[l],b}_t-L^{[k],b}_t= \rbra{L^{[l],b}_{\bar{K}^{[n]}_b-}-L^{[k],b}_{\bar{K}^{[n]}_b-}}+\Delta (k, l, \bar{K}^{[n]}_b , t)
\quad t\in [ \bar{K}^{[n]}_b , \un{K}^{[n]}_0), 
\end{align}
which is non-decreasing. 
From the above arguments and by \eqref{72}, 
the process $\{R^{[l],b}_t-R^{[k],b}_t: t\in [ \bar{K}^{[n]}_b , \un{K}^{[n]}_0)\}$ is a constant process satisfying 
\begin{align}
R^{[l],b}_t-R^{[k],b}_t=R^{[l],b}_{\bar{K}^{[n]}_b-}-R^{[k],b}_{\bar{K}^{[n]}_b-}, \quad t\in [ \bar{K}^{[n]}_b , \un{K}^{[n]}_0). 
\end{align}
\par
By inductively repeating (i) and (ii) for $n\in\bN$, the proof is complete. 
\end{proof}
\begin{Lem}\label{Lem_paths_2}
We assume that 
$X$ belongs to Case $2$. 
We fix $b=0$, $x\in\bR$, and $k, l\in\bR$ with $l-k>0$. 
Then, $\bP_x$-a.s., we have the same statements as (1)--(4) in Lemma \ref{Lem_paths}. 
\end{Lem}
\begin{proof}
The identity \eqref{72} is true for the same reason as in Lemma \ref{Lem_paths} (1).
\par
By \eqref{36}, we have 
\begin{align}
Z^{[l],0}_t&-Z^{[k],0}_t=X^{[l],(\alpha)}_t-X^{[k],(\alpha)}_t
-\rbra{ \inf_{s\in[0, t]}\rbra{X^{[l],(\alpha)}_s\land 0}-\inf_{s\in[0, t]}\rbra{X^{[k],(\alpha)}_s\land 0}}
\\
=&(l-k)
-\rbra{ \inf_{s\in[0, t]}\rbra{X^{[l],(\alpha)}_s\land 0}-\inf_{s\in[0, t]}\rbra{X^{[k],(\alpha)}_s\land 0}}\in[0, l-k],
\quad t\geq 0, \label{84}
\end{align}
and thus we have (2) in Lemma \ref{Lem_paths}. 
By \eqref{38} and \eqref{84}, we have 
\begin{align}
L^{[l],0}_t-L^{[k],0}_t=&\int_0^th^0 (Z^{[l],0}_s)ds-\int_0^th^0(Z^{[k],0}_s)ds\\
=&\int_0^t (\alpha-\delta) 1_{\{Z^{[k],0}_s=0<Z^{[l],0}_s \}}ds\geq 0, \quad t\geq 0 ,\label{85} 
\end{align}
which is non-decreasing. 
Since $U^{\pi^0}$ behaves as $Z^0$ and by \eqref{38} and \eqref{43}, 
we have, for $m\in\bR$,  
\begin{align}
R^{[m],0}_t=&Z^{[m],0}_t-X^{[m]}_t +L^{[l],0}_t\\
=&X^{[m],(\alpha)}_t - \inf_{s\in[0, t]}\rbra{X^{[m],(\alpha)}_s\land 0}-X^{[m]}_t
+\int_0^th^0 (Z^{[m],0}_s)ds\\
=&
- (X^{[m]}_0\land 0)-
\sum_{s\in(0, t]} \rbra{ Z^{[m],0}_{s-}+(X_s-X_{s-} )  }\land 0, \quad t\geq 0. \label{118}
\end{align}
By \eqref{118}, we have 
\begin{align}
\begin{aligned}
R^{[l],0}_t-R^{[k],0}_t=&
 (X^{[k]}_0\land 0)- (X^{[l]}_0\land 0)\\
& +
\sum_{s\in(0, t]}1_{\{s\leq t\}} \bigg{(}\rbra{Z^{[k],0}_{s-}+(X_{s}-X_{s-})}\land 0\\
&\qquad\qquad -\rbra{Z^{[l],0}_{s-}+(X_{s}-X_{s-})}\land 0  \bigg{)}\leq 0,
\end{aligned}
\qquad t\geq0, \label{86}
\end{align}
which is non-increasing {(may decrease only when $Z^{[k],0}_{t-}+(X_{t}-X_{t-})<0$ and thus $Z^{[k],0}_{t}=0$)}. 
By (2), \eqref{85}, \eqref{86}, and \eqref{72}, we have (3) and (4) in Lemma \ref{Lem_paths}. 
\par
The proof is complete. 
\end{proof}
\begin{Lem}\label{Lem_paths_3}
We assume that $(-\infty , 0]$ is irregular for itself for $X^{(\alpha)}$. 
We fix $b=0$, $x\in\bR$, and $k, l\in\bR$ with $l-k>0$. 
Then, $\bP_x$-a.s., we have (1), (2), (4) in Lemma \ref{Lem_paths}. 
In addition, we have 
\begin{align}
L^{[l], 0}_t=L^{[k], 0}_t = \alpha t, \quad t\geq 0. \label{87}
\end{align}
\end{Lem}
\begin{proof}
The identity \eqref{72} is true for the same reason as in Lemma \ref{Lem_paths} (1).
\par
Since $0$ is irregular for $(-\infty , 0]$ for $X^{(\alpha)}$, the process $Z^0$ takes a value $0$ only a finite number of times in a finite time. 
Thus, by \eqref{38}, we have \eqref{87}. 
Since \eqref{84} holds, we have (2). 
By \eqref{72} and \eqref{87}, we have 
\begin{align}
R^{[l],0}_t -R^{[k],0}_t=
(Z^{[l],b}_t-Z^{[k],b}_t )-(l-k), 
\quad t\geq 0,
\end{align}
and by the same argument as in (i) in the proof of Lemma \ref{Lem_paths}, we have (4).
The proof is complete. 
\end{proof}

\section{A property of the Laplace transform of a hitting time} \label{Sec0D}
In this appendix, 
we give a lemma for the continuity of the function $\un{\nu}$. 

\begin{Lem}\label{LemC01}
We assume that $0$ is regular for $\bR\backslash\{0\}$ for $Y^0$. 
Then, 
the function $\un{\nu}$ is right-continuous at $b^\ast$. 
In addition, if $b^\ast>0$ then the function $\un{\nu}$ is left-continuous at $b^\ast$. 
\end{Lem}
\begin{proof}
Note that the limits $\lim_{\varepsilon\downarrow0}\un{\nu}({b^\ast-\varepsilon}) $ and $\lim_{\varepsilon\downarrow0}\un{\nu}({b^\ast+\varepsilon}) $ exist since the function $\un{\nu}$ is non-increasing. 
\par
(i) We prove that
\begin{align}
\un{\nu}(b^\ast)\leq \lim_{\varepsilon\downarrow0}\un{\nu}(b^\ast+\varepsilon).\label{119}
\end{align}
We assume that $0$ is regular for $(-\infty, 0)$ for $X^{(\alpha)}$. 
By the strong Markov property and since the function $\un{\nu}$ is non-increasing on $\bR$, we have, for $\varepsilon>0$,  
\begin{align}
\un{\nu}(b^\ast+\varepsilon)
=\bE_{b^\ast+\varepsilon }\sbra{e^{-q{T^{b^\ast, -}_0} } \un{\nu}\rbra{Y^{b^\ast}_{{T^{b^\ast, -}_0}}} }
\geq \bE_{b^\ast+\varepsilon}\sbra{e^{-q{T^{b^\ast, -}_0} } }\un{\nu}(b^\ast)
=\bE_{b^\ast+\varepsilon}\sbra{e^{-q\tau^{(\alpha), -}_{b^\ast} } }\un{\nu}(b^\ast).
\end{align}
Thus, by taking the limit as $\varepsilon\downarrow 0$, we have \eqref{119}. 
We assume that $0$ is irregular for $(-\infty, 0)$ for $X^{(\alpha)}$. 
Then, for the process $Y^0$, $0$ is irregular for $(-\infty, 0)$ and thus $0$ is regular for $(0, \infty)$ by the hypothesis. 
In addition, 
$0$ is irregular for $(-\infty, 0]$ for $Y^0$ since $Y^0$ does not take a value less than or equal to $0$ for a while after taking a value greater than $0$. 
By the strong Markov property and since the function $\un{\nu}$ is non-increasing on $\bR$, we have, for $\varepsilon>0$,  
\begin{align}
\un{\nu}(b^\ast)
&= \bE_{b^\ast}\sbra{e^{-q\kappa^{b^\ast, +}_{b+\varepsilon} } \un{\nu}\rbra{Y^{b^\ast}_{\kappa^{b^\ast,+}_{b+\varepsilon}}} }
+
\bE_{b^\ast}\sbra{e^{-q K^{p^\ast}_0};K^{p^\ast}_0<\kappa^{b^\ast, +}_{b^\ast+\varepsilon}}\\
\leq \bE_{b^\ast}&\sbra{e^{-q\kappa^{b^\ast, +}_{b^\ast+\varepsilon} } ;\kappa^{b^\ast, +}_{b^\ast+\varepsilon}< K^{p^\ast}_0}
\un{\nu}(b^\ast+\varepsilon)+
\bE_{b^\ast}\sbra{e^{-q K^{p^\ast}_0};K^{p^\ast}_0<\kappa^{b^\ast, +}_{b^\ast+\varepsilon}}. \label{a023}
\end{align}
Thus, by taking the limit as $\varepsilon\downarrow 0$, we have \eqref{119}. 
\par
(ii) We assume that $b^\ast > 0$ and prove that 
\begin{align}
\un{\nu}(b^\ast)\geq \lim_{\varepsilon\downarrow0}\un{\nu}(b^\ast-\varepsilon).\label{120}
\end{align}
We assume that $0$ is regular for $(0, \infty)$ for $X$. 
By the same argument as in the proof of \eqref{a023}, we have, for $\varepsilon \in(0,b^\ast)$,  
\begin{align}
\un{\nu}({b^\ast-\varepsilon})
\leq &\bE_{b^\ast-\varepsilon}\sbra{e^{-q\kappa^{b^\ast, +}_{b^\ast} } ;\kappa^{b^\ast, +}_{b^\ast}< K^{p^\ast}_0}
\un{\nu}(b^\ast)+
\bE_{b^\ast-\varepsilon}\sbra{e^{-q K^{p^\ast}_0};K^{p^\ast}_0<\kappa^{b^\ast, +}_{b^\ast}}
\\
=&\bE_{b^\ast-\varepsilon}\sbra{e^{-q\tau^+_{b^\ast} } ;\tau^+_{b^\ast}< T^{p^\ast}_0}
\un{\nu}(b^\ast)+
\bE_{b^\ast-\varepsilon}\sbra{e^{-q T^{p^\ast}_0};T^{p^\ast}_0<\tau^+_{b^\ast}}.
\end{align}
Thus, by taking the limit as $\varepsilon\downarrow 0$, we have \eqref{120}. 
We assume that $0$ is irregular for $(0, \infty)$ for $X$. 
Then, for the process $Y^0$, $0$ is irregular for $(0, \infty)$ and thus $0$ is regular for $(- \infty, 0)$ by the hypothesis. 
By the strong Markov property and since the function $\un{\nu}$ is non-increasing on $\bR$, we have, for $\varepsilon\in(0, b^\ast)$, 
\begin{align}
\un{\nu} (b^\ast)
&= \bE_{b^\ast}\sbra{e^{-q\kappa^{b^\ast, -}_{b^\ast-\varepsilon} } \un{\nu}\rbra{Y^{b^\ast}_{\kappa^{b^\ast,-}_{b^\ast-\varepsilon}}} }
\geq \bE_{b^\ast}\sbra{e^{-q\kappa^{b^\ast,-}_{b^\ast-\varepsilon} } }\un{\nu}({b^\ast-\varepsilon}). 
\label{a024}
\end{align}
Thus, by taking the limit as $\varepsilon\downarrow 0$, we have \eqref{120}. 
\par
(iii) By \eqref{119}, \eqref{120}, and the fact that the function $\un{\nu}$ is non-increasing on $\bR$, 
$\un{\nu}(b^\ast)$ is equal to $\lim_{\varepsilon\downarrow0}\un{\nu}(b^\ast+\varepsilon)$ regardless of the value of $b^\ast$ and is equal to $\lim_{\varepsilon\downarrow0}\un{\nu}(b^\ast-\varepsilon)$ when $b^\ast>0$. 
The proof is complete. 
\end{proof}

\section{An example satisfying Assumption \ref{Ass402}} \label{Ape00D}
In this appendix, we prove that hyper-exponential L\'evy processes with unbounded variation paths 
satisfies the assumptions {that are necessary to apply our main result}. 
\par
We assume that the L\'evy measure $\Pi$ is absolutely continuous with respect to the Lebesgue measure and has a density 
\begin{align}
\pi(x) {:=}\sum_{k=1}^{\tilde{N}}\tilde{a}_k \tilde{\rho}_k e^{\tilde{\rho}_k x} 1_{\{x<0\}}+ \sum_{k=1}^N a_k \rho_k e^{-\rho_k x} 1_{\{x>0\}}, 
\end{align}
where $N, \tilde{N}\in\bN\cup\{0\}$, and ${\{a_k\}}_{k=1}^N$, ${\{\tilde{a}_k\}}_{k=1}^{\tilde{N}}$, ${\{\rho_k\}}_{k=1}^N $, and ${\{\tilde{\rho}_k\}}_{k=1}^{\tilde{N}} $ are sets of positive values satisfying
\begin{align}
\rho_1<\rho_2< \cdots < \rho_N, \qquad \tilde{\rho}_1<\tilde{\rho}_2< \cdots < \tilde{\rho}_{\tilde{N}}.
\end{align}
Then $X$ is called a hyper-exponential L\'evy process. 
In addition, we assume that $\sigma>0$. Then $X$ has unbounded variation paths. 
\par
We have 
\begin{align}
\int_{(-\infty , -1)} (-x) \Pi (dx)=\sum_{k=1}^{\tilde{N}}\int_{(-\infty , -1)} (-x)\tilde{a}_k \tilde{\rho}_k e^{\tilde{\rho}_k x} dx
=\sum_{k=1}^{\tilde{N}}\tilde{a}_k \rbra{1+ \frac{1}{\tilde{\rho}_k}}e^{-\tilde{\rho}_k}< \infty, 
\end{align}
and thus $X$ satisfies Assumption \ref{Ass201}. 
By Corollary \ref{Lem205}, $X$ satisfies Assumption \ref{Ass202a}. 
By \cite[Remark 1, Theorem 1, 3 and 5]{KuzKypPar2012}, we have, for $b>0$ and a non-negative bounded measurable function $f$,
\begin{align}
\bE_x&\sbra{e^{-q \tau^{(\alpha), -}_0 }f (X^{(\alpha)}_{\tau^-_0}) }
=
\sum_{k=0}^{\tilde{N}}\rbra{{\sum_{l=1}^{\tilde{N}+1} \tilde{a}^{(\alpha)}_{k,l} \tilde{\zeta}^{(\alpha)}_l e^{- \tilde{\zeta}^{(\alpha)}_l x}}
 }\tilde{\cA}_k f ,\quad x>0 , \label{68}\\
\bE_x&\sbra{e^{-q (\tau^-_0 \land \tau^+_b )}f (X_{\tau^-_0\land \tau^+_b}) }\\
&
\begin{aligned}
=
&\sum_{k=0}^{\tilde{N}}\rbra{\sum_{l=1}^{\tilde{N}+1} \tilde{a}^-_{k,l} \tilde{\zeta}_l e^{- \tilde{\zeta}_l x}
+\sum_{l=1}^{N+1} a^-_{k,l} \zeta_l e^{-\zeta_l (b-x)}  }\tilde{\cA}_k f\\
&+\sum_{k=0}^N\rbra{\sum_{l=1}^{\tilde{N}+1} \tilde{a}^+_{k,l} \tilde{\zeta}_l e^{- \tilde{\zeta}_l (b-x)}
+\sum_{l=1}^{N+1} a^+_{k,l} \zeta_l e^{-\zeta_l x}  }\cA^b_k f, \quad x\in (0, b), 
\end{aligned}
\label{69}
\end{align}
where $\zeta^{(\alpha)}_l$, 
$\zeta_l$, and $\tilde{\zeta}_l$ are values in $(0, \infty]$, $\tilde{a}^{(\alpha)}_{k, l}$, $\tilde{a}^-_{k, l}$, $a^-_{k,l}$, $\tilde{a}^+_{k, l}$, and $a^+_{k,l}$ are real values for $k, l\in\bN\cup\{0\}$, and
\begin{align}
\cA^b_k f=&
\begin{cases}
f(b) ,  \quad  &k=0 \\
\int_0^\infty {\rho}_k e^{-\rho_k y}f(y+b)dy  , \quad &k=1, 2,\cdots N, 
\end{cases}
\\
\tilde{\cA}_k f=&
\begin{cases}
f(0), \quad &k=0, \\
\int_{-\infty}^0 \tilde{\rho}_k e^{\tilde{\rho}_k y}f(y)dy, \quad& k=1, 2,\cdots, \tilde{N}.
\end{cases}
\end{align}
Here, we assume that $\infty \times e^{- \infty}=0$. 
By the strong Markov property at $\kappa^{b,-}_b$ and $\kappa^{b,+}_b$, 
and from \eqref{68} and \eqref{69}, we have 
\begin{align}
&\nu_b (x)=
\begin{cases}
\sum_{k=0}^{\tilde{N}}\sum_{l=1}^{\tilde{N}+1} \tilde{a}^{(\alpha)}_{k,l} \tilde{\zeta}^{(\alpha)}_l e^{- \tilde{\zeta}^{(\alpha)}_l (x-b)}
\tilde{\cA}_k  \tilde{\nu}_b ,\quad &x>b,\\
\begin{aligned}
&\sum_{k=0}^{\tilde{N}}\rbra{\sum_{l=1}^{\tilde{N}+1} \tilde{a}^-_{k,l} \tilde{\zeta}_l e^{- \tilde{\zeta}_l x}
+\sum_{l=1}^{N+1} a^-_{k,l} \zeta_l e^{-\zeta_l (b-x)}  }
\\
&+\sum_{k=0}^N\rbra{\sum_{l=1}^{\tilde{N}+1} \tilde{a}^+_{k,l} \tilde{\zeta}_l e^{- \tilde{\zeta}_l (b-x)}
+\sum_{l=1}^{N+1} a^+_{k,l} \zeta_l e^{-\zeta_l x}  }  \cA^b_k \nu_b, 
\end{aligned}
\quad &x\in (0, b), 
\end{cases}
\end{align}
where 
\begin{align}
\tilde{\nu}_b (x) = \nu_b(x+b), \quad x\in\bR. 
\end{align}
From the above equation and since the function $\nu_b$ is continuous on $(0, \infty)$ by the same proof as that of Lemma \ref{LemC01}, 
the function $\nu_b$ has the following density on $(0, \infty)$ with respect to the Lebesgue measure: 
\begin{align}
\nu_b^\prime (x)=
\begin{cases}
-\sum_{k=0}^{\tilde{N}}\rbra{\sum_{l=1}^{\tilde{N}+1} \tilde{a}^{{(\alpha)}}_{k,l} \tilde{\zeta}^{(\alpha)2}_l e^{- \tilde{\zeta}^{(\alpha)}_l (x-b)}
 }\tilde{\cA}_k  \tilde{\nu}_b ,\quad &x>b,\\
 \begin{aligned}
&\sum_{k=0}^{\tilde{N}}\rbra{-\sum_{l=1}^{\tilde{N}+1} \tilde{a}^-_{k,l} \tilde{\zeta}_l^2 e^{- \tilde{\zeta}_l x}
+\sum_{l=1}^{N+1} a^-_{k,l} \zeta_l^2 e^{-\zeta_l (b-x)}  }
\\
&+\sum_{k=0}^N\rbra{\sum_{l=1}^{\tilde{N}+1} \tilde{a}^+_{k,l} \tilde{\zeta}_l^2 e^{- \tilde{\zeta}_l (b-x)}
-\sum_{l=1}^{N+1} a^+_{k,l} \zeta_l^2 e^{-\zeta_l x}  }\cA^b_k \nu_b, 
\end{aligned}
\quad &x\in (0, b). 
\end{cases}
\end{align}
Thus, $X$ satisfies Assumption \ref{Ass402}. 
Therefore, we can apply Theorem \ref{Thm502} for the hyper-exponential L\'evy processes with unbounded variation paths.

\section{Proof of Lemma \ref{Lem401}}\label{prfverification}
We write $w_\varepsilon(x)= w(x+\varepsilon)$, $w^\prime_\varepsilon(x)= w^\prime(x+\varepsilon)$, and $w^{\prime\prime}_\varepsilon(x)= w^{\prime \prime}(x+\varepsilon)$ for $\varepsilon>0$. Then, the function $w_\varepsilon$ belongs to $C^{(1)}_{\text{line}}$ (resp.\ $C^{(2)}_{\text{line}}$) when $X$ has bounded (resp.\ unbounded) variation paths and satisfies
\begin{align}
\sup_{r\in[0, \alpha]} \rbra{\cL w_{\varepsilon} (x)-qw_\varepsilon(x)-rw_\varepsilon^\prime(x)+r}\leq 0, \quad x\geq 0, \label{veriprf1}\\
w_\varepsilon^\prime (x)\leq \beta, \quad x\geq 0, \label{veriprf2}\\
\inf_{x\in[0, \infty)}w_\varepsilon(x)>-m, \quad\text{ for some }m. \label{veriprf3}
\end{align}
We fix a strategy $\pi\in\cA$. 
Using the Meyer--It\^o formula (see, e.g., \cite[Theorem IV.70]{Pro2005} (resp.\ \cite[Theorem IV.71]{Pro2005}) when $X$ has bounded (resp.\ unbounded) variation paths) for the process $\{e^{-qt}w_\varepsilon (U^\pi_t): t\geq 0\}$, 
we have, $\bP_x$-a.s. with $x\in\bR$, 
\begin{align}
&e^{-qt}w_\varepsilon (U^\pi_t)-w_\varepsilon (U^\pi_{0-})=-\int_0^t qe^{-qs} w_\varepsilon (U^\pi_s)ds 
+\frac{\sigma^2}{2}\int_0^te^{-qs}w_\varepsilon^{\prime\prime}(U^\pi_s)ds
\\
&+\int_{[0, t]}e^{-qs}w_\varepsilon^\prime(U^\pi_{s-})dU^\pi_s+ \sum_{s\in[0, t]}\rbra{e^{-qs}w_\varepsilon(U^\pi_s)-e^{-qs}w_\varepsilon(U^\pi_{s-})-e^{-qs}w_\varepsilon^\prime(U^\pi_{s-})( U^\pi_s-U^\pi_{s-})}\\
&
\begin{aligned}
=&\int_0^t e^{-qs}(\cL w_\varepsilon(U^\pi_s)-q w_\varepsilon(U^\pi_s))ds+M_t
-\int_0^t e^{-qs}w_\varepsilon^\prime (U^\pi_{s-})l^\pi_s ds\\
&+\int_{[0, t]}e^{-qs}w_\varepsilon^\prime(U^\pi_{s-})dR^{\pi, c}_s
-\sum_{s\in[0, t]}e^{-qs}\rbra{w_\varepsilon(U^\pi_s-(R^\pi_s-R^\pi_{s-})) -w_\varepsilon (U^\pi_{s})},
\end{aligned}
\quad t\geq 0 ,
\label{64}
\end{align} 
where $U^\pi_{0-}=X_0$, $\{B_t: t\geq0 \}$ is the standard Brownian motion, $\cN$ is the Poisson random measure on $([0, \infty)\times \bR, \cB [0, \infty)\times\cB (\bR))$ associated with $ds \times \Pi(dx)$, 
\begin{align}
M_t=&\sigma \int_0^t e^{-qs}w_\varepsilon^\prime(U^\pi_{s-})dB_s\\
&+\int_{[0,t]\times \bR}e^{-qs}\rbra{w_\varepsilon (U^\pi_{s-}+y)-w_\varepsilon (U^\pi_{s-})}
\rbra{\cN(ds\times dy)-ds \times \Pi(dy)},
\end{align}
and 
\begin{align}
R^{\pi, c}_t=R^{\pi}_t-\sum_{s\in[0, t]} (R^\pi_s-R^\pi_{s-}), \quad t\geq 0. 
\end{align}
Note that the process $\{M_t: t\geq 0\}$ is a local martingale.  
By \eqref{veriprf1} and \eqref{veriprf2}, we have, for $s \geq 0$, 
\begin{align}
\begin{aligned}
&(\cL-q)w_\varepsilon(U^\pi_s)\leq \inf_{r\in[0, \alpha]}r(w_\varepsilon^\prime (U^\pi_s)-1), 
\quad w_\varepsilon^\prime (U^\pi_{s-})\leq \beta, \\
&w_\varepsilon(U^\pi_s-( R^\pi_s-R^\pi_{s-})) -w_\varepsilon (U^\pi_{s})\geq -\beta ( R^\pi_s-R^\pi_{s-}). 
\end{aligned}
\label{65}
\end{align}
From \eqref{64} and \eqref{65}, we have, $\bP_x$-a.s. with $x\in\bR$, 
\begin{align}
e^{-qt}w_\varepsilon (U^\pi_t)-w_\varepsilon (U^\pi_{0-})
\leq &\int_0^t e^{-qs} \rbra{\inf_{r\in[0, \alpha]}r(w_\varepsilon^\prime (U^\pi_s)-1)-w_\varepsilon^\prime (U^\pi_{s-})l^\pi_s }ds \\
& \qquad \qquad \qquad \qquad + M_t+\beta\int_{[0, t]}e^{-qs}dR^{\pi}_s\\
\leq&-\int_0^t e^{-qs}l^\pi_s ds+\beta\int_{[0, t]}e^{-qs}dR^{\pi}_s+ M_t, \quad t\geq 0. 
\label{66}
\end{align}
We take a localizing sequence of stopping times ${\{T_n\}}_{n\in\bN}$ for $M$. 
By taking the expectation using $\bP_x$ with $x\in\bR$ at time $t\land T_n$ and by \eqref{veriprf3} {and \eqref{66}}, we have 
\begin{align}
-m\bE_x \sbra{e^{-q(t\land T_n)}}-w_\varepsilon (x)&\leq 
\bE_x \sbra{e^{-q(t\land T_n)}w_\varepsilon (U^\pi_{t\land T_n})}-w_\varepsilon (x)\\
&\leq 
\bE_x \sbra{-\int_0^{t\land T_n} e^{-qs}l^\pi_s ds+\beta\int_{[0, t\land T_n]}e^{-qs}dR^{\pi}_s}. 
\end{align}
By taking the limit as $t\uparrow \infty$ and $n\uparrow \infty$, we have 
\begin{align}
w_\varepsilon (x)\geq v_\pi (x). \label{67}
\end{align}
Since \eqref{67} is true for all $\varepsilon>0$, $x\in\bR$, and $\pi\in\cA$, 
by taking the limit as $\varepsilon \downarrow 0$, the proof is complete. 
\begin{Rem}
Since the density of $w$ or the density of $w^\prime$ may not be locally bounded on $[0, \infty)$, 
the integrals in \eqref{64} may not be well defined {for $w$}. 
To avoid discussing the definition of \eqref{64}, we used $w_\varepsilon$ in the proof of Lemma \ref{Lem401}.  
\end{Rem}
{
\section{The proofs of the lemmas in Section \ref{Sec007}}\label{Ape00H}
In this appendix, we give the proofs of Lemmas \ref{Lem702} and \ref{Lem705}. 
}
{
\subsection{The proof of Lemma \ref{Lem702}}\label{Ape00H1}}
{We fix $x, b \in\bR$ and $\omega\in\Omega$ such that $X_0 (\omega)= x$, $t\mapsto X_t$ is c\`adl\`ag and the differential equation \eqref{4} has a unique solution for $\alpha>0$.}
\par
{
(i) We take $\alpha_1, \alpha_2\in(0, \infty]$ with $\alpha_1<\alpha_2$ and prove that 
\begin{align}
Y^{\alpha_1, b}_t\geq Y^{\alpha_2, b}_t, \qquad t\geq 0.\label{Rev005}
\end{align}
We assume that $\alpha_2<\infty$. Then, we have 
\begin{align}
Y^{\alpha_1, b}_t- Y^{\alpha_2, b}_t
=\int_0^t (h^{\alpha_2, b}(Y^{\alpha_2, b}_s)-h^{\alpha_1, b}(Y^{\alpha_1, b}_s)) ds,
\qquad t\geq 0,
\end{align}
where
\begin{align}
h^{\alpha, b}(y){:=}
\begin{cases}
\alpha1_{(b, \infty)}(y) \quad &\text{ in Case $1$}, \\
\alpha1_{(b, \infty)}(y)+\delta 1_{\{b\}}(y) \quad &\text{ in Case $2$}.
\end{cases}
\end{align}
Thus, the map $t\mapsto Y^{\alpha_1, b}_t- Y^{\alpha_2, b}_t$ changes continuously. 
If $Y^{\alpha_1, b}_{t_0}- Y^{\alpha_2, b}_{t_0}<0$ for $t_0>0$, then we have 
$u(t_0)<t_0$, 
where $u(t):=\sup\{ u\leq t: Y^{\alpha_1, b}_u- Y^{\alpha_2, b}_u\geq0\}
$ 
for $t\geq 0$, and 
$Y^{\alpha_1, b}_t- Y^{\alpha_2, b}_t<0$ for $t\in (u(t_0), t_0]$. Thus, we have $h^{\alpha_2, b}(Y^{\alpha_2, b}_t)-h^{\alpha_1, b}(Y^{\alpha_1, b}_t) \geq 0$ for $t\in (u(t_0), t_0]$ and 
\begin{align}
Y^{\alpha_1, b}_{t_0}- Y^{\alpha_2, b}_{t_0}=&
(Y^{\alpha_1, b}_{t_0}- Y^{\alpha_2, b}_{t_0})-(Y^{\alpha_1, b}_{u(t_0)}- Y^{\alpha_2, b}_{u(t_0)})\\
=&\int_{u(t_0)}^{t_0}  (h^{\alpha_2, b}(Y^{\alpha_2, b}_s)-h^{\alpha_1, b}(Y^{\alpha_1, b}_s)) ds
\geq0 ,
\end{align}
which contradicts $Y^{\alpha_1, b}_{t_0}- Y^{\alpha_2, b}_{t_0}<0$. 
Therefore, the inequality \eqref{Rev005} is true. 
}
\par
{
We assume that $\alpha_2=\infty$. 
If $Y^{\alpha_1, b}_{t_0}- Y^{\infty, b}_{t_0}<0$ for some $t_0>0$, then we have $Y^{\alpha_1, b}_t< Y^{\infty, b}_t\leq b$ for $t \in (u(t_0), t_0] $,
and thus $Y^{\alpha_1, b}_{u(t_0)}\leq Y^{\infty, b}_{u(t_0)}\leq b$. 
By the definition of $Y^{\alpha_1, b}$, on $[u(t_0), t_0] $, it behaves as the L\'evy process $X$. 
Thus, the map $t\mapsto Y^{\alpha_1, b}_t- Y^{\infty, b}_t$ is non-decreasing on $[u(t_0), t_0] $, and we have $Y^{\alpha_1, b}_{u(t_0)}- Y^{\infty, b}_{u(t_0)}\leq Y^{\alpha_1, b}_{t_0}- Y^{\infty, b}_{t_0}<0$ and $Y^{\alpha_1, b}_{u(t_0)-}- Y^{\infty, b}_{u(t_0)-}\geq 0$. 
However, since all negative jumps of $Y^{\alpha_1, b}$ and $Y^{\infty, b}$ occur at the same time and have the same size, the above situation does not occur. 
Therefore, the inequality \eqref{Rev005} is true. 
}
\par
{
(ii) We prove \eqref{Rev006} when $X$ does not have positive jumps at $t\geq 0$. 
For $t \in [0, \tau^+_b)$, we have $Y^{\alpha, b}_t =Y^{\infty, b}_t=X_t $ for $\alpha>0$. 
Below, we consider the convergence \eqref{Rev006} after $\tau^+_b$. 
}\par
{
We assume that $X$ does not have positive jumps at $t_1 >\tau^+_b$ and prove that 
\begin{align}
\lim_{\alpha \uparrow \infty}Y^{\alpha, b}_{t_1}\leq b . \label{Rev007}
\end{align}
Here, note that $\lim_{\alpha \uparrow \infty}Y^{\alpha, b}_{t_1}$ exists by (i). 
If $Y^{\alpha, b}_{t_1}\leq b$ for some $\alpha>0$, then \eqref{Rev007} is obvious by (i), so we only check cases where $Y^{\alpha, b}_{t_1}> b$ for all $\alpha>0$.
We fix $\alpha_1>0$ and $\varepsilon\in (0, Y^{\alpha_1, b}_{t_1}-b)$, then there exists $\delta \in (0 , t_1)$ such that 
\begin{align}
\inf \{ X_t : t \in (t_1 -\delta ,  t_1)\}>X_{t_1}-\varepsilon . \label{Rev008}
\end{align}
There exists $\alpha_2>0$ and $t_2 \in (t_1 -\delta ,  t_1)$ such that $Y^{\alpha_2, b}_{t_2}\leq b$. 
In fact, if it does not exist, then we have, by \eqref{Rev005}, \eqref{Rev008} and the definition of $Y^{\alpha, b}$, 
\begin{align}
\lim_{\alpha\uparrow\infty}Y^{\alpha, b}_{t_1}\leq 
\lim_{\alpha\uparrow\infty}\rbra{Y^{\alpha_1, b}_{t_1-\delta}+\varepsilon- \alpha \delta}=-\infty
\end{align}
which contradicts that $Y^{\alpha, b}_{t_1}>b$ for $\alpha>0$. Thus, we have 
\begin{align}
\lim_{\alpha\uparrow\infty}Y^{\alpha, b}_{t_1} \leq Y^{\alpha_2, b}_{t_1}\leq Y^{\alpha_2, b}_{t_2}+\varepsilon\leq b+\varepsilon, \label{Rev009}
\end{align}
where in the second inequality, we used \eqref{Rev008}. 
The inequalities \eqref{Rev009} are true for $\varepsilon>0$, we have \eqref{Rev007}. }
\par
{
We assume that $Y^{\infty , b}_t =b$ and $X$ does not have positive jumps at $t\geq\tau^+_b$, then we have \eqref{Rev006} by (i) and \eqref{Rev007}. }
\par{
We assume that $Y^{\infty , b}_t <b$ with $t \geq \tau^+_b$, then we can define $r(t):= \sup \{s \in[0, t] : Y^{\infty, b}_s =b\}$. 
If $X$ does not have negative jumps at $r(t)$, then we have $Y^{\infty,b}_{r(t)}=b$, and for $\varepsilon>0$ there exists $t_3 \in (r(t), t)$ such that $Y^{\infty , b}_{t_3} \in (b-\varepsilon , b)$ and $X$ does not have jump at $t_3$. 
Thus, we have 
\begin{align}
\begin{aligned}
\lim_{\alpha\uparrow\infty} Y^{\alpha, b}_t \leq \lim_{\alpha\uparrow\infty} Y^{\alpha, b}_{t_3}
+(X_t-X_{t_3})\leq b+(X_t-X_{t_3})<&Y^{\infty , b}_{t_3}+\varepsilon +(X_t-X_{t_3})\\
=& Y^{\infty , b}_t+\varepsilon.     
\end{aligned}\label{Rev010}
\end{align}
Since \eqref{Rev010} is true for $\varepsilon>0$ and by (i), we have \eqref{Rev006}. 
If $X$ has a negative jump at $r(t)$, then $Y^{\infty , b}_{r(t)-}=b$, and for $\varepsilon>0$ there exists $t_4 \in (0, r(t))$ such that $\absol{X_{r(t)-}-X_{t_4} } <\varepsilon$ and $X$ does not have jump at $t_4$. 
Thus, we have
\begin{align}
\begin{aligned}
&\lim_{\alpha\uparrow\infty} Y^{\alpha, b}_t \leq \lim_{\alpha\uparrow\infty} Y^{\alpha, b}_{t_4}
+(X_t-X_{t_4})\leq 
 \lim_{\alpha\uparrow\infty} Y^{\alpha, b}_{t_4}
+(X_{r(t)-}-X_{t_4})+(X_{t}-X_{r(t)-})
\\
&<b+\varepsilon +(X_t-X_{r(t)-})= Y^{\infty,b}_{r(t)-}+\varepsilon +(X_t-X_{r(t)-})
= Y^{\infty , b}_t+\varepsilon.     
\end{aligned}
\label{Rev011}
\end{align}
Since \eqref{Rev011} is true for $\varepsilon>0$ and by (i), we have \eqref{Rev006}. 
}
\par
{The proof is complete.}

{\subsection{The proof of Lemma \ref{Lem705}}\label{Ape00H2}}

{We fix $x\in\bR$ and $\omega\in\Omega$ such that $X_0 (\omega)= x$, $t\mapsto X_t$ is c\`adl\`ag and the differential equation \eqref{4} has a unique solution for $\alpha>0$.}
\par
{(i) We assume that $b>0$ and prove that, for $\alpha_1, \alpha_2>0$ with $\alpha_1 < \alpha_2$,
\begin{align}
Z^{\alpha_1, b}_t \geq  Z^{\alpha_2, b}_t , \quad L^{\alpha_1, b}_t \leq  L^{\alpha_2, b}_t, \quad R^{\alpha_1, b}_t \leq R^{\alpha_2, b}_t, \qquad t\geq 0. \label{Rev017}
\end{align}
We define the stopping times $\bar{J}^{[n]}_b$ and $\un{J}^{[n]}_b$ as follows: 
for $n\in\bN$, 
\begin{align}
\bar{J}^{[n]}_b := \inf \{t> \un{J}^{[n-1]}_0 : Z^{\alpha_2,b}_t \geq b \}, \quad
\un{J}^{[n]}_0 := \inf \{t> \bar{J}^{[n]}_b : Z^{\alpha_2,b}_t =0\},
\end{align}
where $\un{J}^{[0]}_0 :=0$. 
We prove \eqref{Rev017} on $[\un{J}^{[n-1]}_0 , \bar{J}^{[n]}_b)$ and $[\bar{J}^{[n]}_b,\un{J}^{[n]}_0 ) $ for $n\in\bN$, inductively.
We assume that $Z^{\alpha_1, b}_{0-} =  Z^{\alpha_2, b}_{0-}=X_0=x$, $L^{\alpha_1, b}_{0-} =  L^{\alpha_2, b}_{0-}=0$, $R^{\alpha_1, b}_{0-} = R^{\alpha_2, b}_{0-}=0$. 
}
\par
{
We assume that \eqref{Rev017} is true for $t < \un{J}^{[n-1]}_0$ and prove that \eqref{Rev017} on $[\un{J}^{[n-1]}_0 , \bar{J}^{[n]}_b)$. 
On $[\un{J}^{[n-1]}_0 , \bar{J}^{[n]}_b)$, $Z^{\alpha_2, b}$ behaves as the reflected L\'evy process $\un{Y}=\{\un{Y}_t :=X_t - \inf_{s\in[0, t]}X_s\land 0 :t\geq 0\}$ since it takes values in $[0, b)$. 
Thus, we can regard $Z^{\alpha_2, b}$ as behaving the refracted--reflected L\'evy process at $b $ with $\alpha_1$ on $[\un{J}^{[n-1]}_0 , \bar{J}^{[n]}_b)$. Therefore, by Lemma \ref{Lem_paths} and \ref{Lem_paths_2} and Remark \ref{RemD02}, we have $Z^{\alpha_1, b}_t \geq  Z^{\alpha_2, b}_t$, $L^{\alpha_1, b}_t -R^{\alpha_1, b}_{\un{J}^{[n-1]}_0-} \leq  L^{\alpha_2, b}_t-L^{\alpha_2, b}_{\un{J}^{[n-1]}_0-}=0$ and 
$R^{\alpha_1, b}_t -R^{\alpha_1, b}_{\un{J}^{[n-1]}_0-} \leq  R^{\alpha_2, b}_t-R^{\alpha_2, b}_{\un{J}^{[n-1]}_0-}$ for $t\in[\un{J}^{[n-1]}_0 , \bar{J}^{[n]}_b)$. 
Thus, we have $L^{\alpha_1, b}_t\geq  L^{\alpha_2, b}_t$ and $R^{\alpha_1, b}_t\leq  R^{\alpha_2, b}_t$ for $t\in[\un{J}^{[n-1]}_0 , \bar{J}^{[n]}_b)$
Therefore, \eqref{Rev017} is true on $[\un{J}^{[n-1]}_0 , \bar{J}^{[n]}_b)$. 
}
\par
{
We assume that \eqref{Rev017} is true for $t < \bar{J}^{[n]}_b$ and prove that \eqref{Rev017} on $[ \bar{J}^{[n]}_b, \un{J}^{[n]}_0 )$. 
After time $\bar{J}^{[n]}_b$, $Z^{\alpha_1, b}$ behaves as refracted L\'evy process at $b$ with $\alpha_1$ until it takes value $0$. 
On the other hand, $Z^{\alpha_2, b}$ behaves as refracted L\'evy process at $b$ with $\alpha_2$ when $\alpha_2<\infty$ and as reflected L\'evy process at $b$ when $\alpha_2=\infty$ until it takes value $0$. 
By the above facts together with Lemma \ref{Lem_paths} (2) and \ref{Lem_paths_2}, Remark \ref{RemD02} and Lemma \ref{Lem702}, we have 
$Z^{\alpha_1, b}_t \geq \hat{Z}^{\alpha_1, b}_t \geq  Z^{\alpha_2, b}_t$, where $\hat{Z}^{\alpha_1, b}:=\{\hat{Z}^{\alpha_1, b}_s : s\geq \un{J}^{[n-1]}_0\}$ is the refracted--reflected L\'evy process at $b$ with $\alpha_1$ satisfying $\hat{Z}^{\alpha_1, b}_{\un{J}^{[n-1]}_0}=Z^{\alpha_2, b}_{\un{J}^{[n-1]}_0}$, and thus
$R^{\alpha_1, b}_t=R^{\alpha_1, b}_{\un{J}^{[n-1]}_0-} \leq R^{\alpha_2, b}_{\un{J}^{[n-1]}_0-} = R^{\alpha_2, b}_t$ for  
$t \in [\un{J}^{[n-1]}_0 , \bar{J}^{[n]}_b)$. 
In addition, by \eqref{Rev020} and the above, we have, for $t\in [ \bar{J}^{[n]}_b, \un{J}^{[n]}_0 )$, 
\begin{align}
L^{\alpha_1, b}_t=X_t+R^{\alpha_1, b}_t-Z^{\alpha_1, b}_t 
\leq X_t+R^{\alpha_2, b}_t-Z^{\alpha_2, b}_t =L^{\alpha_2, b}_t. \label{Rev018}
\end{align} 
Therefore, \eqref{Rev017} is true on $[ \bar{J}^{[n]}_b, \un{J}^{[n]}_0 )$. 
}
\par
{
(ii)  
We assume $b>0$ and prove \eqref{Rev019} for a.e. $t\geq 0$. 
We prove \eqref{Rev019} for a.e. $t\in (\un{J}^{[n-1]}_0, \un{J}^{[n]}_0 ] $ for $n\in\bN$, inductively.
}
\par
{
We assume that \eqref{Rev019} is true at $\un{J}^{[n-1]}_0$ when $n\geq 2$. 
Since $Z^{\infty,b}$ behaves as $Z^{\alpha,b}$ on $ [\un{J}^{[n-1]}_0, \bar{J}^{[n]}_b )$ and 
by Lemma \ref{Lem_paths} and \ref{Lem_paths_2} and Remark \ref{RemD02}, we have, for $\alpha>0$ and $t\in [\un{J}^{[n-1]}_0, \bar{J}^{[n]}_b )$
\begin{align}
Z^{\alpha, b}_t -Z^{\infty, b}_t \in &[0, Z^{\alpha, b}_{\un{J}^{[n-1]}_0} -Z^{\infty, b}_{\un{J}^{[n-1]}_0}], \\
(R^{\alpha, b}_t -R^{\alpha, b}_{\un{J}^{[n-1]}_0})=&(R^{\infty, b}_t-R^{\infty, b}_{\un{J}^{[n-1]}_0})
\in [-( Z^{\alpha, b}_{\un{J}^{[n-1]}_0} -Z^{\infty, b}_{\un{J}^{[n-1]}_0}), 0]. 
\end{align}
Only in this place we assume that $\un{J}^{[0]}_0=0-$. 
From the above and \eqref{Rev020}, the convergence \eqref{Rev019} is true on $(\un{J}^{[n-1]}_0, \bar{J}^{[n]}_b ) $. 
Since $Z^{\alpha , b}_{  \bar{J}^{[n]}_b} \geq Z^{\infty , b}_{  \bar{J}^{[n]}_b}=b$ and by the same argument as that of (ii) of the proof of Lemma \ref{Lem702}, we have $\lim_{\alpha\uparrow\infty}Z^{\alpha , b}_t =Z^{\infty , b}_t$ for 
$[\bar{J}^{[n]}_b, \un{J}^{[n]}_0 ) $ when $X$ does not have positive jumps at $t$. In addition, 
$R^{\alpha, b}$ and $R^{\alpha,b}$ do not increases on $[\bar{J}^{[n]}_b, \un{J}^{[n]}_0 ) $. 
Thus, by \eqref{Rev020}, we have $\lim_{\alpha\uparrow \infty}L^{\alpha , b}_t =L^{\infty , b}_t$ for 
$[\bar{J}^{[n]}_b, \un{J}^{[n]}_0 ) $ when $X$ does not have positive jumps at $t$. 
We also have $\lim_{\alpha\uparrow\infty}\rbra{Z^{\alpha, b}_{{\un{J}^{[n]}_0- }}+(X_{\un{J}^{[n]}_0 }-X_{\un{J}^{[n]}_0- })}=Z^{\infty, b}_{{\un{J}^{[n]}_0- }}+(X_{\un{J}^{[n]}_0 }-X_{\un{J}^{[n]}_0- })$ by the same argument as above and since $X$ does not have positive jumps at $\un{J}_0^{[n]}$, 
so we have 
$\lim_{\alpha\uparrow\infty}Z^{\alpha, b}_{{\un{J}^{[n]}_0 }}=Z^{\infty, b}_{{\un{J}^{[n]}_0}}=0$ and 
\begin{align}
\lim_{\alpha\uparrow\infty}R^{\alpha, b}_{{\un{J}^{[n]}_0 }}=&\lim_{\alpha\uparrow\infty}R^{\alpha ,b}_{\bar{J}^{[n]}_b}
-\lim_{\alpha\uparrow\infty}\rbra{\rbra{Z^{\alpha, b}_{{\un{J}^{[n]}_0- }}+(X_{\un{J}^{[n]}_0 }-X_{\un{J}^{[n]}_0- })}\land 0}\\
=&R^{\infty ,b}_{\bar{J}^{[n]}_b}-\rbra{\rbra{Z^{\infty, b}_{{\un{J}^{[n]}_0- }}+(X_{\un{J}^{[n]}_0 }-X_{\un{J}^{[n]}_0- })}\land 0}
=R^{\infty, b}_{{\un{J}^{[n]}_0}}. 
\end{align}
Therefore, by \eqref{Rev020} and the above, we have $\lim_{\alpha\uparrow\infty}L^{\alpha, b}_{{\un{J}^{[n]}_0 }}=L^{\infty, b}_{{\un{J}^{[n]}_0}}$ and \eqref{Rev019} is true on $(\un{J}^{[n-1]}_0, \un{J}^{[n]}_0 ] $ when $X$ does not have positive jumps.
}
\par
{
(iii) We assume that $b=0$ and $X$ has bounded variation paths and non-negative drift. 
Under the condition above, we prove \eqref{Rev017} for $\alpha_1, \alpha_2\in(0, \infty]$ with $\alpha_1<\alpha_2$. 
}
\par
{If $\alpha_2 =\infty$, we have $Z^{\alpha_1, 0}_t \geq 0=Z^{\infty , 0}_t $ for $t\geq 0$ by \cite[Section 3]{Nob2021}.
On the other hand, if $\alpha_2 < \infty$, we have $Z^{\alpha_1, 0}_t \geq Z^{\alpha_2 , 0}_t $ for $t\geq 0$ by \eqref{36} and the facts that 
\begin{align}
X^{(\alpha_1)}_t -X^{(\alpha_2)}_t=(\alpha_2 - \alpha_1)t,\  
\inf_{s\in[0, t]}\rbra{X^{(\alpha_1)}_s\land 0}
-\inf_{s\in[0, t]}\rbra{X^{(\alpha_2)}_s\land 0}\leq (\alpha_2 - \alpha_1)t, \quad t\geq 0. 
\end{align}
By \eqref{36}, \eqref{38}, Lemma \ref{LemB01}, Remark \ref{RemA03} and  \cite[Section 3]{Nob2021}, we have, for $t\geq 0 $, 
\begin{align}
R^{\alpha_2 , 0}_t =&-(\delta\land 0)\int_0^t1_{\{0\}}(Z^{\alpha_2 ,b}_{t}) dt -\sum_{s\in[0, t]}(Z^{\alpha_2 ,b}_{s-}+X_s -X_{s-})1_{\{Z^{\alpha_2 ,b}_{s-}+X_s -X_{s-}<0\}}\\
\geq& -(\delta\land 0)\int_0^t1_{\{0\}}(Z^{\alpha_1 ,b}_{t}) dt 
-\sum_{s\in[0, t]}(Z^{\alpha_1 ,b}_{s-}+X_s -X_{s-})1_{\{Z^{\alpha_1 ,b}_{s-}+X_s -X_{s-}<0\}}=
R^{\alpha_1 , 0}_t. 
\end{align}
Thus, by \eqref{Rev020}, we have \eqref{Rev018} for $t\geq 0$, and so \eqref{Rev017} is true. 
}
\par
{
(iv) We give the same assumption as (iii) and prove \eqref{Rev019} for a.e. $t\geq 0$. 
We can prove $\lim_{\alpha\uparrow \infty} Z^{\alpha ,0}_t= 0$ when $X$ does not have jump at $t\geq 0$ by the similar argument as that of the proof of \eqref{Rev007}, thus we omit it. 
By \eqref{36}, \eqref{38}, Lemma \ref{LemB01}, Remark \ref{RemA03}, the monotone convergence theorem and \cite[Section 3]{Nob2021}, we have, for $t\geq 0$,
\begin{align}
\lim_{\alpha\uparrow\infty }
R^{\alpha , 0}_t
=&\lim_{\alpha\uparrow\infty }\rbra{-(\delta\land 0)\int_0^t1_{\{0\}}(Z^{\alpha ,b}_{t}) dt -\sum_{s\in[0, t]}(U^{\alpha ,b}_{s-}+X_s -X_{s-})1_{\{Z^{\alpha ,b}_{s-}+X_s -X_{s-}<0\}}}\\
=&-(\delta\land 0)t -\sum_{s\in[0, t]}(X_s -X_{s-})1_{\{X_s -X_{s-}<0\}}
=R^{\infty , 0}_t.
\end{align}
Therefore, by \eqref{Rev020}, we have $\lim_{\alpha\uparrow\infty }L^{\alpha , 0}_t=L^{\infty , 0}_t$ and also \eqref{Rev019} when $X$ does not have positive jumps at $t\geq 0$. 
}
\par
{The proof is complete.}


\bibliographystyle{jplain}
\bibliography{NOBA_references_05} 

\end{document}

%% file: style-common.tex
\makeatletter
\@addtoreset{footnote}{page}
\makeatother

\usepackage{style-enumitem}

\renewcommand{\labelenumi}{(\theenumi)}

\usepackage{amsthm}
\usepackage{amsmath,amssymb,latexsym,amsfonts,mathrsfs}

\newcommand{\n}{\nonumber}

\renewcommand{\hat}{\widehat}
\renewcommand{\tilde}{\widetilde}
\renewcommand{\bar}{\overline}
\newcommand{\un}[1]{\underline{#1}}

\newcommand{\absol}[1]{\left| #1 \right|} 
\newcommand{\rbra}[1]{\!\left( #1 \right)} 
\newcommand{\cbra}[1]{\!\left\{ #1 \right\}} 
\newcommand{\sbra}[1]{\!\left[ #1 \right]} 

\newcommand{\bE}{\ensuremath{\mathbb{E}}}
\newcommand{\bF}{\ensuremath{\mathbb{F}}}

\newcommand{\bN}{\ensuremath{\mathbb{N}}}

\newcommand{\bP}{\ensuremath{\mathbb{P}}}

\newcommand{\bR}{\ensuremath{\mathbb{R}}}

\newcommand{\cA}{\ensuremath{\mathcal{A}}}
\newcommand{\cB}{\ensuremath{\mathcal{B}}}

\newcommand{\cF}{\ensuremath{\mathcal{F}}}

\newcommand{\cL}{\ensuremath{\mathcal{L}}}

\newcommand{\cN}{\ensuremath{\mathcal{N}}}

%% file: style-e.tex
\theoremstyle{plain}
\newtheorem{Thm}{Theorem}[section]

\newtheorem{Lem}[Thm]{Lemma}

\newtheorem{Cor}[Thm]{Corollary}

\theoremstyle{definition}
\newtheorem{Ass}[Thm]{Assumption}

\newtheorem{Rem}[Thm]{Remark}

%% file: style-layout-paper.tex
\numberwithin{equation}{section}

\makeatletter
\renewcommand\section{\@startsection {section}{1}{\z@}%
                                   {-3.5ex \@plus -1ex \@minus -.2ex}%
                                   {2.3ex \@plus.2ex}%
                                   {\normalfont\large\bf}}
\makeatother

\makeatletter
\renewcommand\subsection{\@startsection {subsection}{1}{\z@}%
                                   {-3.5ex \@plus -1ex \@minus -.2ex}%
                                   {2.3ex \@plus.2ex}%
                                   {\normalfont\normalsize\bf}}
\makeatother